%% file: badia_hornkjol_khan_mardal_martin_baier_2023.tex
\documentclass[oneside, reqno, 10pt, a4paper]{amsart}

\usepackage{dutchcal}
\usepackage{microtype,times}
\AtBeginDocument{
  \DeclareSymbolFont{AMSb}{U}{msb}{m}{n}
  \DeclareSymbolFontAlphabet{\mathbb}{AMSb}
}
\DeclareFontFamily{U}{mathx}{\hyphenchar\font45}
\DeclareFontShape{U}{mathx}{m}{n}{<-> mathx10}{}
\DeclareSymbolFont{mathx}{U}{mathx}{m}{n}
\DeclareMathAccent{\widebar}{0}{mathx}{"73}

\usepackage[dvipsnames]{xcolor}
\usepackage{graphicx}
\usepackage{subfig}
\usepackage{svg}
\usepackage{tikz}
\tikzstyle{arrow} = [thick,->,>=stealth]
\usetikzlibrary{shapes.misc,matrix,fit,positioning,arrows.meta,decorations.pathreplacing,calc,shapes.geometric,arrows}
\tikzset{Matrix/.style={matrix of nodes, font=\footnotesize,text height=1pt, text depth=0.5pt, text width=8.5pt, align=center, column sep=0pt, row sep=0pt, nodes in empty cells}}

\usepackage[a4paper, pdftex, left=1.5cm, top=2cm, right=1.5cm, bottom=2cm]{geometry}
\usepackage[final]{pdfpages}
\usepackage{changepage}
\usepackage[foot]{amsaddr}

\usepackage{url}
\usepackage{hyperref}
\hypersetup{
    breaklinks=true,
    bookmarksopen=true,
    pdftitle={Efficient and reliable divergence-conforming methods for an elasticity-poroelasticity interface problem}, 
    pdfauthor={S. Badia, M. HornkjØl, A. Khan, K.-A. Mardal, A. F. Martín, and R. Ruiz-Baier}, 
    pdfsubject={}, 
    pdfkeywords={Biot--elasticity transmission equations, mixed finite element methods, divergence-conforming schemes, a priori error analysis, a posteriori error analysis, operator preconditioning.}, 
    colorlinks=true,
    linkcolor=blue,
    citecolor=blue,
    filecolor=black,
    urlcolor=blue
}

\usepackage{csquotes}
\usepackage[backend=biber, defernumbers=true, maxbibnames=5, style=numeric-comp, isbn=false, bibencoding=utf8, safeinputenc, url=false, doi=true, giveninits=true]{biblatex}

\usepackage{float}
\usepackage{framed}
\usepackage{verbatim}
\usepackage{fancyvrb}
\usepackage{booktabs}
\usepackage{colortbl}
\usepackage{multirow}
\usepackage{algorithm}
\usepackage{algpseudocode}
\makeatletter
\algdef{S}[IF]{IfNoThen}[1]{\algorithmicif\ #1}
\makeatother
\usepackage[inline]{enumitem}
\usepackage{siunitx}

\newtheorem{remark}{Remark}
\newtheorem{lemma}{Lemma}

\newtheorem{theorem}{Theorem}
\usepackage[breakable]{tcolorbox}
\tcbuselibrary{theorems}
\tcbuselibrary{skins}
\tcbset{
	commonstyle/.style={
		theorem style=plain,
		enhanced jigsaw,
		fonttitle=\bfseries,
		fontupper=\itshape,
		halign=justify,
		separator sign=:,
		description delimiters none,
		description font=\bfseries, 
		terminator sign={.\hspace{0.25em}},
		arc=0mm,outer arc=0mm,
		boxrule=0pt,toprule=0pt,bottomrule=0pt,leftrule=0pt,rightrule=0pt,
		titlerule=0pt,toptitle=0pt,bottomtitle=0pt,top=0pt,
		colback=white,coltitle=black,
		boxsep=0pt, bottom=0pt, left=0pt, 
	}
}
\newtcbtheorem[]{myproblem}{Problem}%
{center, commonstyle, fonttitle=\bfseries}{pb}
\newtcbtheorem[]{mydefinition}{Definition}
{center, commonstyle, fonttitle=\bfseries}{pb}
\newtcbtheorem[]{myassumption}{Assumption}
{center, commonstyle, fonttitle=\bfseries}{pb}

\usepackage{amsmath, amsfonts, amssymb, amscd, bm, mathtools, stmaryrd}

\definecolor{darkred}{rgb}{0.82,0.15,0.20}
\definecolor{darkblue}{rgb}{0.82,0.15,0.12}

\numberwithin{equation}{section}
\numberwithin{figure}{section}
\numberwithin{table}{section}
\numberwithin{lemma}{section}
\numberwithin{corollary}{section}
\numberwithin{theorem}{section}
\numberwithin{remark}{section}

\newcommand\cero{\boldsymbol{0}}

\newcommand\bH{\mathbf{H}}
\newcommand\bL{\mathbf{L}}
\newcommand\bI{\mathbf{I}}

\newcommand\bbP{\mathbb{P}}

\newcommand\bV{\mathbf{V}}

\newcommand\bX{\boldsymbol{X}}

\newcommand\beps{\boldsymbol{\varepsilon}}
\newcommand\ff{\boldsymbol{f}}
\newcommand\bb{\boldsymbol{b}}

\newcommand\bg{\boldsymbol{g}}
\newcommand\nn{\boldsymbol{n}}

\newcommand\bsigma{\boldsymbol{\sigma}}

\newcommand\bu{\boldsymbol{u}}
\newcommand\bv{\boldsymbol{v}}

\newcommand\bx{\boldsymbol{x}}

\newcommand\cA{\mathcal{A}}
\newcommand\cB{\mathcal{B}}
\newcommand\cC{\mathcal{C}}
\newcommand\cD{\mathcal{D}}
\newcommand\cE{\mathcal{E}}

\newcommand\cJ{\mathcal{J}}
\newcommand\cM{\mathcal{M}}

\newcommand\cP{\mathcal{P}}

\newcommand\RR{\mathbb{R}}
\newcommand\cT{\mathcal{T}}

\newcommand\bzeta{\boldsymbol{\zeta}}
\newcommand{\rQ}{\mathrm{Q}}
\newcommand{\rZ}{\mathrm{Z}}
\newcommand{\rL}{\mathrm{L}}
\newcommand{\estE}{\Theta_{K}}
\newcommand{\estP}{\Psi_{K}}
\newcommand{\estint}{\Lambda_{e}}
\newcommand{\UpsilonE}{\widetilde{\Upsilon}_{\!K}}
\newcommand{\UpsilonP}{\widehat{\Upsilon}_{\!K}}

\newcommand\rH{\mathrm{H}}

\newcommand{\norm}[1]{\left\|#1\right\|}

\newcommand\bdiv{\mathop{\mathbf{div}}\nolimits}
\newcommand\vdiv{\mathop{\mathrm{div}}\nolimits}

\newcommand\bnabla{\boldsymbol{\nabla}}

\newcommand\Tend{t_{\mathrm{final}}}

\newcommand\OmP{\Omega^{\mathrm{P}}}
\newcommand\OmE{\Omega^{\mathrm{E}}}
\newcommand\rmP{\mathrm{P}}
\newcommand\rmE{\mathrm{E}}

\newcommand{\mean}[1]{{\left\{\kern-0.7ex\left\{ #1 
		\right\}\kern-0.7ex\right\}}}
\DeclarePairedDelimiter\jump{\llbracket}{\rrbracket}
\DeclareFontEncoding{FMS}{}{}
\DeclareFontSubstitution{FMS}{futm}{m}{n}
\DeclareFontEncoding{FMX}{}{}
\DeclareFontSubstitution{FMX}{futm}{m}{n}
\DeclareSymbolFont{fouriersymbols}{FMS}{futm}{m}{n}
\DeclareSymbolFont{fourierlargesymbols}{FMX}{futm}{m}{n}
\DeclareMathDelimiter{\VERT}{\mathord}{fouriersymbols}{152}{fourierlargesymbols}{147}

\xdefinecolor{mygrey}{rgb}{0.65,0.65,0.6375}
\newcolumntype{g}{ >{\columncolor{mygrey}} c }

\renewenvironment{proof}{\noindent{\it Proof.}}{\hfill$\square$}

\address{$^1$School of mathematics, Monash university, 9 Rainforest Walk, Melbourne 3800 VIC, Australia; 
and Centre Internacional de M\`etodes Num\`erics a l'Enginyeria, Campus Nord, 08034, Barcelona, Spain.}
\address{$^2$Department of Mathematics,  University of Oslo, Norway.}
\address{$^3$Department of Mathematics, Indian Institute of Technology Roorkee, Roorkee 247667, India.}
\address{$^4$School of Computing, Australian National University, Acton ACT 2601, Australia.}
\address{$^5$School of Mathematics and Victorian Heart Institute, Monash University, 9 Rainforest Walk, Melbourne 3800 VIC, Australia; 
and Universidad Adventista de Chile, Casilla 7-D, Chill\'an, Chile.}
\author[S. Badia]{Santiago Badia$^{1}$}
\author[M. Hornkj{\o}l]{Martin Hornkj{\o}l$^{2}$}
\author[A. Khan]{Arbaz Khan$^{3}$}
\author[K.-A. Mardal]{Kent-Andr\'e Mardal$^{2}$}
\author[A. F. Mart\'{\i}n]{Alberto F. Mart\'{\i}n$^{4}$}
\author[R. Ruiz-Baier]{Ricardo Ruiz-Baier$^{5}$}
\email{santiago.badia@monash.edu}
\email{marhorn@math.uio.no}
\email{arbaz@ma.iitr.ac.in}
\email{kent-and@simula.no}
\email{alberto.f.martin@anu.edu.au}
\email{ricardo.ruizbaier@monash.edu}  

\begin{document}

		\title[H(div)-conforming elasticity-poroelasticity coupling]{Efficient and reliable divergence-conforming methods for an elasticity-poroelasticity interface problem}
                \date{\today}
                \maketitle
		\begin{abstract}
		We present a finite element discretisation  to model the interaction between a poroelastic structure and an elastic medium. The consolidation problem considers fully coupled deformations across an interface, ensuring continuity of displacement and total traction, as well as no-flux for the fluid phase. Our formulation of the poroelasticity equations incorporates displacement, fluid pressure, and total pressure, while the elasticity equations adopt a displacement-pressure formulation. Notably, the transmission conditions at the interface are enforced without the need for Lagrange multipliers. 
    We demonstrate the stability and convergence of the divergence-conforming finite element method across various polynomial degrees. The \emph{a priori} error bounds remain robust, even when considering large variations in intricate model parameters such as Lam\'e constants, permeability, and storativity coefficient. To enhance computational efficiency and reliability, we develop residual-based \emph{a posteriori} error estimators that are independent of the aforementioned coefficients. Additionally, we devise parameter-robust and optimal block diagonal preconditioners. Through numerical examples, including adaptive scenarios, we illustrate the scheme's properties such as convergence and parameter robustness. 
	
	\smallskip
    \noindent \textbf{Keywords.} Biot--elasticity transmission equations, mixed finite element methods, divergence-conforming schemes, \emph{a priori} error analysis, \emph{a posteriori} error analysis, operator preconditioning.
		\end{abstract}

\maketitle

\input{content.tex}

\printbibliography
 
\end{document}

%% file: content.tex
\section{Introduction}
In a variety of engineering and biomedical applications, poroelastic bodies are either surrounded or in contact with a purely elastic material. Examples include filter design, prosthetics, simulation of oil extraction from reservoirs, carbon sequestration, and sound insulation structures. From the viewpoint of constructing and analysing numerical methods,  recent works for the interfacial Biot/elasticity problem can be found in \cite{anaya20,anaya22,gira11,gira20,girault15}. These contributions include mortar-type discretisations, formulations using rotations, and extensions to lubrication models. In this work we focus on H(div)-conforming discretisations of displacement for the transmission problem, in combination with a total pressure formulation for both the elastic and poroelastic sub-domains.   Divergence-conforming methods with tangential jump penalisation for elasticity were already  proposed in  \cite{hong16}. Their counterpart  for Biot poroelasticity equations (in the two-field formulation) have been introduced in \cite{kanschat18,zeng19,zeng20}, while a much more abundant literature is available for Brinkman flows (as well as coupled flow-transport problems) in \cite{buerger22,buerger19,hong16b,konno11}. The extension to interfacial porous media flow has been addressed in \cite{carvalho20,chen11,fu18,kanschat10}.  In general, such type of discretisations offer appealing features such as  local conservation of mass, and the ability to produce robust and locking--free schemes. These type of schemes are required because of the large number of parameters upon which robustness is targeted (especially in the limits of near incompressibility and near impermeability).  

As regularity of the solution is not always available (due to possibly high contrast in material parameters, domain singularities, etc), we are also interested in deriving \emph{a posteriori} error estimates which allow us to apply adaptive mesh refinement in the regions where it is most required. A coupled elliptic--parabolic \emph{a posteriori} error analysis for Biot poroelasticity and multiple network poroelasticity is available from the works \cite{ahmed19,eliseussen,li20}. On the other hand, robust estimates for the elasticity--poroelasticity coupling have been obtained only  recently, for enriched Galerkin methods in \cite{gira20}, and in  \cite{anaya22} for rotation-based formulations.

Here the analysis is carried out considering two examples of fluid pressure approximation: either continuous or discontinuous piecewise polynomials. For the DG case we use a classical symmetric interior penalty (SIP) method.  In all cases the proposed formulation is robust with respect to material parameters that can assume very small or very large values, including the extreme cases of near incompressibility, near impermeability, and near zero storativity. This parameter independence in the stability of the discrete problem is critical in the \emph{a priori} error bounds, in the derivation of \emph{a posteriori} error estimates, and in the design of robust preconditioners. 

Finally, we design optimal preconditioner that are robust with respect to parameters. The preconditioner is block-diagonal and its definition relies on the stability properties of the proposed numerical scheme, i.e., it consists in a discretisation of the continuous Riesz map (see \cite{lee17,hong16b} for similar approaches). The definition of the pressure block is motivated by \cite{olshanskii06}, where a robust preconditioner for the interface Stokes problems with high contrast is proposed.

The remainder of the paper has been structured as follows. Section~\ref{sec:problem} is devoted to the description of the interfacial problem, it states boundary and transmission conditions, and there we give the continuous weak formulation. The discrete problem in two different formulations is defined in Section~\ref{sec:FE}. The stability and solvability of the H(div)-conforming methods is addressed in Section~\ref{sec:apriori}.  Residual-based \emph{a posteriori} error estimators are constructed and analysed in Section~\ref{sec:aposte}. The operator framework and definition of a norm-equivalent preconditioner is addressed in Section~\ref{sec:robustness}, and numerical methods that confirm the properties of the proposed methods are collected in Section~\ref{sec:results}.

\section{Problem statement}\label{sec:problem}	 
\subsection{Preliminaries}
Standard notation on Lebesgue and Sobolev spaces together with their associated norms will be adopted throughout the presentation. For $s\geq 0$ and a generic domain $\cD$, the symbol  $\rH^s(\cD)$ denotes the usual Sobolev space equipped with the norm $\|\cdot\|_{s,\cD}$ and seminorm $|\cdot|_{s,\cD}$. The case $s=0$ is understood as the space $\rL^2(\cD)$. Boldfaces will be used to denote vector-valued spaces, maintaining the same notation as scalar spaces for the norms and seminorms. For a Banach space $V$, we will use the symbol $V'$ to denote its dual space. 
We also recall the definition of the space $\bH(\vdiv,\cD):=\{\bv\in \bL^2(\cD): \vdiv \bv \in \rL^2(\cD)\}$, which is of Hilbert type when endowed with the norm $\|\bv\|^2_{\vdiv,\cD} =  \|\bv\|^2_{0,\cD}+\|\vdiv\bv\|^2_{0,\cD}$.  As usual, throughout the paper the notation $A \lesssim B$ will abbreviate the inequality $A\leq CB$ where $C$ is a generic constant that does not depend on the maximal mesh sizes $h$ nor on the sensitive parameters of the model, in particular the Lam\'e parameters on each subdomain (and will proceed similarly for $A\gtrsim B$). The constants in the inequalities will be specified whenever necessary from the context.

\subsection{The transmission problem}

Following the problem setup from \cite{gira11,gira20}, let us consider a bounded Lipschitz domain $\Omega\subset\RR^d$, $d \in \{2,3\}$,
together with a partition into non-overlapping and connected subdomains $\OmE$, $\OmP$ representing zones occupied by an elastic body (e.g., a non-pay rock, in the context of reservoir modelling) and a fluid-saturated poroelastic region (e.g., a reservoir), respectively. The interface between the two subdomains is denoted as $\Sigma=\partial\OmP\cap \partial\OmE$, and on it the normal vector $\nn$ is assumed to point from $\OmP$ to $\OmE$. The boundary of the domain $\Omega$ is separated in terms of the boundaries of two individual subdomains, that is $\partial \Omega:=\Gamma^\mathrm{P} \cup \Gamma^\mathrm{E}$, and then subdivided as the disjoint Dirichlet and Neumann type condition as $\Gamma^\mathrm{P}:= \Gamma^\mathrm{P}_D \cup \Gamma^\mathrm{P}_N$ and $\Gamma^\mathrm{E}:= \Gamma^\mathrm{E}_D \cup \Gamma^\mathrm{E}_N$, respectively. We assume that all sub-boundaries have positive $(d-1)$--Hausdorff measure.

In the overall domain we state the momentum balance of the fluid and solid phases on the poroelastic region, the mass conservation of the total amount of fluid, and the balance of linear momentum on the elastic region. In doing so, and following \cite{anaya20}, in addition to the usual variables of elastic displacement, poroelastic displacement, and fluid pressure, we employ the total pressure in the poroelastic subdomain, and the Herrmann pressure in the elastic subdomain. For given body loads $\bb^\mathrm{P}(t):\OmP\to \RR^d$, $\bb^\mathrm{E}(t):\OmE\to \RR^d$, and a volumetric source or sink $\ell^\mathrm{P}(t):\OmP\to \RR$, one seeks for each time $t\in (0,\Tend]$, the vector of solid displacements $\bu^\mathrm{E}:\OmE\to \RR^d$ of the non-pay zone, the elastic pressure  $\varphi^\mathrm{E}:\OmE\to\RR$, the displacement $\bu^\mathrm{P}(t):\OmP\to \RR^d$, the pore fluid pressure $p^\mathrm{P}(t):\OmP\to\RR$, and the total pressure $\varphi^\mathrm{P}(t):\OmP\to\RR$ of the reservoir, satisfying:
\begin{subequations}\label{eq:coupled}
\begin{align}
-\bdiv( 2{\mu^\mathrm{P}} \beps(\bu^\mathrm{P})- \varphi^\mathrm{P} \bI)&
= \bb^\mathrm{P} & \text{in $\OmP\times(0,\Tend]$},\label{eq:poro-momentum}\\
\biggl(c_0
+\frac{\alpha^2}{{\lambda^\mathrm{P}}}\biggr) \partial_t p^\mathrm{P} -\frac{\alpha}{{\lambda^\mathrm{P}}}  \partial_t \varphi^\mathrm{P}
- \frac{1}{\eta} \vdiv(\kappa  \nabla p^\mathrm{P} ) &= \ell^\mathrm{P}
& \text{in $\OmP\times(0,\Tend]$},\label{eq:Biot} \\
\varphi^\mathrm{P} - \alpha p^\mathrm{P} + {\lambda^\mathrm{P}} \vdiv\bu^\mathrm{P} &=
0 & \text{in $\OmP\times(0,\Tend]$},\label{eq:poro-constitutive}\\
-\bdiv( 2{\mu^\mathrm{E}} \beps(\bu^\mathrm{E})- \varphi^\mathrm{E} \bI)&
= \bb^\mathrm{E} & \text{in $\OmE\times(0,\Tend]$},\label{eq:Elast}\\
\varphi^\mathrm{E} + {\lambda^\mathrm{E}}\vdiv\bu^\mathrm{E} &=
0 & \text{in $\OmE\times(0,\Tend]$}. \label{eq:elast-constitutive}
\end{align}\end{subequations}
Here $\kappa(\bx)$ is the hydraulic conductivity of the porous medium, $\eta$ is the constant viscosity of the interstitial fluid, $c_0$ is the storativity coefficient, $\alpha$ is the Biot--Willis consolidation parameter, and $\mu^\mathrm{E},\lambda^\mathrm{E}$ and $\mu^\mathrm{P},\lambda^\mathrm{P}$ are the Lam\'e parameters associated with the constitutive law of the solid on the elastic and on the poroelastic subdomain, respectively.  The poroelastic stress $\widetilde{\bsigma} = \bsigma-\alpha p^{\rmP} \bI$ is composed by the effective mechanical stress ${\lambda^\mathrm{P}} (\vdiv \bu^{\rmP})\bI + 2{\mu^\mathrm{P}} \beps(\bu^{\rmP})$ plus  the non-viscous fluid stress (the fluid pressure scaled with $\alpha$). This system is complemented by the following set of boundary conditions
\begin{subequations}
\begin{align}
\label{bc:GammaP}
\bu^\mathrm{P} = \cero \quad \text{and} \quad \frac{\kappa}{\eta} \nabla p^\mathrm{P} \cdot\nn^{\Gamma} = 0  & &\text{on $\Gamma^\mathrm{P}_D\times(0,\Tend]$},\\
\label{bc:SigmaP}
[2{\mu^\mathrm{P}} \beps(\bu^\mathrm{P})- \varphi^\mathrm{P} \bI]
\nn^{\Gamma} = \cero \quad  \text{and}\quad p^\mathrm{P}=0
& &\text{on $\Gamma^\mathrm{P}_N\times(0,\Tend]$},\\
\label{bc:Gamma}
\bu^\mathrm{E} = \cero & &\text{on $\Gamma^\mathrm{E}_D\times(0,\Tend]$},\\
\label{bc:Sigma}
[2{\mu^\mathrm{E}} \beps(\bu^\mathrm{E})- \varphi^\mathrm{E} \bI]\nn^{\Gamma} = \cero
& &\text{on $\Gamma^\mathrm{E}_N\times(0,\Tend]$}.
\end{align}\end{subequations}
Here, the partition $\Gamma^\mathrm{P}:= \Gamma^\mathrm{P}_D \cup \Gamma^\mathrm{P}_N$ denotes the sub-boundaries where we 
impose essential (i.e., $\bu^\mathrm{P} = \cero$) and natural boundary conditions (i.e., $\widetilde{\bsigma}\nn^{\Gamma} = \cero$) corresponding to equation \eqref{eq:poro-momentum}. For ease of notation, we note that, in this definition, we are assuming that the essential and natural sub-boundaries corresponding to equation \eqref{eq:Biot} (i.e., the ones where we impose $p^\mathrm{P}=0$ and $\nabla p^\mathrm{P} \cdot\nn^{\Gamma} = 0$, respectively) match the natural and essential sub-boundaries associated to \eqref{eq:poro-momentum}, respectively. However, in general, this does not have to be case, and one may 
choose separately the partition into essential and natural for each of these 
two equations separately. 

Along with the previous set of boundary conditions, the system is also complemented
by  transmission conditions in the absence of external forces (derived by means of homogenisation in \cite{mikelic12}) that take the following form 
\begin{equation}\label{eq:transmission2}
\bu^\mathrm{P} = \bu^\mathrm{E}, \quad  
[2{\mu^\mathrm{P}} \beps(\bu^\mathrm{P})- \varphi^\mathrm{P} \bI]\nn
=[2{\mu^\mathrm{E}} \beps(\bu^\mathrm{E})- \varphi^\mathrm{E} \bI]\nn, \quad
\frac{\kappa}{\eta}\nabla p^\mathrm{P}\cdot\nn = 0
\quad \text{on $\Sigma\times(0,\Tend]$},
\end{equation}
which represent continuity of the medium, the balance of total tractions, and no-flux of fluid at the interface, respectively. An advantage with respect to \cite{anaya20} is that here we can use the full poroelastic stresses to impose the transmission conditions. We also consider the following initial conditions 
\begin{equation*}
p^{\rmP}(0) =0, \quad \bu^{\rmP}(0) = \cero  \qquad \qquad  \text{in $\Omega^\mathrm{P}$}.
\end{equation*}
Homogeneity of the boundary and initial conditions is only assumed to simplify the exposition of the subsequent analysis, however the results remain valid for more general assumptions.  {We also note that non-homogeneous boundary conditions are used in the numerical experiments.}

\subsection{A weak formulation}
 As the paper focuses on the spatial discretisation, we will restrict the problem formulation to the steady case. For this, we  
 define the function spaces 
$$\bV^\mathrm{E}:=\bH^1_{\Gamma_D^\rmE}(\Omega^\mathrm{E}),
\quad \bV^\mathrm{P}:=\bH^1_{\Gamma_D^\rmP}(\Omega^\mathrm{P}),
\quad \rQ^\mathrm{P}:= \rH^1_{\Gamma_N^\rmP}(\OmP),\quad
\rZ^\mathrm{E}:=\rL^2(\OmE),\quad \rZ^\mathrm{P}:=\rL^2(\OmP).$$ 

Considering a backward Euler discretisation in time with constant time step $\Delta t$, 
multiplying the time-discrete version of  \eqref{eq:Biot} by adequate test functions, integrating by parts (in space) whenever appropriate, and using the boundary conditions \eqref{bc:Gamma}-\eqref{bc:Sigma}, leads to the following  weak formulation: 
Find $\bu^\mathrm{P} \in \bV^\mathrm{P}, \bu^\mathrm{E} \in \bV^\mathrm{E}, p^\mathrm{P} \in \rQ^\mathrm{P}, \varphi^\mathrm{E} \in \rZ^\mathrm{E}, \varphi^\mathrm{P}(t) \in \rZ^\mathrm{P}$ such that
\begin{align*}
2{\mu^\mathrm{P}} ( \beps(\bu^\mathrm{P}),\beps(\bv^\mathrm{P}))_{0,\OmP}
-(\varphi^\mathrm{P},\vdiv \bv^\mathrm{P})_{0,\OmP}
-\langle [2{\mu^\mathrm{P}} \beps(\bu^\mathrm{P})
- \varphi^\mathrm{P}\bI]\nn, \bv^\mathrm{P} \rangle_{{\Gamma^P_D}} &= (\bb^\mathrm{P},\bv^\mathrm{P})_{0,\OmP} 
,\\
(\varphi^\mathrm{P},\psi^\mathrm{P})_{0,\OmP}-\alpha(p^\mathrm{P},\psi^\mathrm{P})_{0,\OmP}
+{\lambda^\mathrm{P}}(\psi^\mathrm{P},\vdiv \bu^\mathrm{P})_{0,\OmP} &=0,\\
\frac{1}{\Delta t}\biggl( c_0 +\frac{\alpha^2}{{\lambda^\mathrm{P}}}\biggr)( p^\mathrm{P}, q^\mathrm{P})_{0,\OmP}
-\frac{1}{\Delta t}\frac{\alpha}{{\lambda^\mathrm{P}}} ( \varphi^\mathrm{P},q^\mathrm{P})_{0,\OmP}
+\frac{1}{\eta}(\kappa \nabla p^\mathrm{P},\nabla q^\mathrm{P})_{0,\OmP}
-\frac{1}{\eta}\langle \kappa \nabla p^\mathrm{P}\cdot\nn,q^\mathrm{P} \rangle_{\partial\OmP}
&=(\ell^{\mathrm{P}},q^\mathrm{P})_{0,\OmP},\\
2{\mu^\mathrm{E}}( \beps(\bu^\mathrm{E}),\beps(\bv^\mathrm{E}))_{0,\OmE}
-(\varphi^\mathrm{E},\vdiv \bv^\mathrm{E})_{0,\OmE}
-\langle [2{\mu^\mathrm{E}} \beps(\bu^\mathrm{E})
- \varphi^\mathrm{E}\bI]\nn^{\partial\OmE}, \bv^\mathrm{E}\rangle_{{\Gamma_D^E}}&=(\bb^\mathrm{E},\bv^\mathrm{E})_{0,\OmE} 
,\\
(\varphi^\mathrm{E},\psi^\mathrm{E})_{0,\OmE} +{\lambda^\mathrm{E}}(\psi^\mathrm{E},\vdiv \bu^\mathrm{E})_{0,\OmE} &=0.
\end{align*}
Note that we can simply define a \emph{global displacement} $\bu\in\bV:=\bH^1_{\Gamma_D}(\Omega)$ (through continuity of the medium in \eqref{eq:transmission2}) such that $\bu|_{\OmP}=\bu^\mathrm{P}$ and $\bu|_{\OmE}=\bu^\mathrm{E}$; as well as a \emph{global pressure} (it is the total pressure on the poroelastic medium and the elastic hydrostatic pressure on the elastic subdomain)   $\varphi\in \rZ:=\rL^2(\Omega)$ such that $\varphi|_{\OmP}=\varphi^\mathrm{P}$ and $\varphi|_{\OmE}=\varphi^\mathrm{E}$. Similarly,  we define the body load $\bb\in \bL^2(\Omega)$ composed by 	$\bb|_{\OmP}=\bb^\mathrm{P}$ and $\bb|_{\OmE}=\bb^\mathrm{E}$, and also the global Lam\'e parameters $\mu$ and $\lambda$ as $\mu|_{\OmP}=\mu^\mathrm{P},\,\lambda|_{\OmP}=\lambda^\mathrm{P}$ and $\mu|_{\OmE}=\mu^\mathrm{E},\,\lambda|_{\OmE}=\lambda^\mathrm{E}$. 
{We also multiply the weak form of the mass conservation equation by -1}.  
The steps above in combination with the second and third transmission conditions in \eqref{eq:transmission2}, yield: Find $\bu\in \bV,\, p^\mathrm{P} \in \rQ^\mathrm{P},\,\varphi\in\rZ$ such that
\begin{subequations}\label{weakEP}
\begin{alignat}{5}
&&   a_1(\bu,\bv)   &&                 &\;+&\; b_1(\bv,\varphi)     &=&\;F(\bv)&\quad\forall \bv\in\bV, \label{weak-u}\\
{-}\tilde{a}_2\biggl(\frac{1}{\Delta t} p^\mathrm{P},q^\mathrm{P}\biggr) &\; {-}&               &&      a_2(p^\mathrm{P},q^\mathrm{P})   &\;{+}&\;    b_2\biggl(q^\mathrm{P}, \, \frac{1}{\Delta t} \varphi\biggr) &=&\;G(q^\mathrm{P}) &\quad\forall q^\mathrm{P}\in \rQ^\mathrm{P}, \label{weak-p}\\
&&b_1(\bu,\psi)  &\;+\;& b_2(p^\mathrm{P},\psi)&\;-&\; a_3(\varphi,\psi) &=&0 &\quad\forall\psi\in\rZ, \label{weak-psi} 
\end{alignat}\end{subequations}
where the bilinear forms 
$a_1:\bV\times\bV \to \RR$,
$a_2:Q^\mathrm{P}\times \rQ^\mathrm{P} \to \RR$, $a_3:\rZ\times\rZ\to \RR$,
$b_1:\bV \times \rZ \to \RR$, $b_2:Q^\mathrm{P} \times \rZ\to \RR$,
and linear functionals $F:\bV \to\RR$, $G:Q^\mathrm{P}\to\RR$, adopt the following form 
\begin{gather}
a_1(\bu,\bv):=   2( \mu\,\beps(\bu),\beps(\bv))_{0,\Omega},
\qquad b_1(\bv,\psi):= -(\psi,\vdiv \bv)_{0,\Omega}, \qquad F(\bv) := ( \bb ,\bv)_{0,\Omega},\nonumber\\
\tilde{a}_2(p^\mathrm{P},q^\mathrm{P})  := \biggl( c_0 +\frac{\alpha^2}{{\lambda^\mathrm{P}}}\biggr)
(p^\mathrm{P}, q^\mathrm{P})_{0,\OmP} , \qquad a_2(p^\mathrm{P},q^\mathrm{P})  :=\frac{1}{\eta}(\kappa \nabla p^\mathrm{P},\nabla q^\mathrm{P})_{0,\OmP} ,
\label{bilinearforms}\\
b_2(p^\mathrm{P},\psi):=  \frac{\alpha}{{\lambda^\mathrm{P}}}(p^\mathrm{P},\psi^\mathrm{P})_{0,\OmP} , \qquad 
a_3(\varphi,\psi):= 
( \frac{1}{\lambda}\varphi, \psi)_{0,\Omega}, 
\qquad
G(q^\mathrm{P}) := {-}( \ell^\mathrm{P} , q^\mathrm{P})_{0,\OmP} . \nonumber
\end{gather}

\subsection{Properties of the continuous weak form and further assumptions}\label{sec:properties}
The variational forms above satisfy the continuity bounds 
\begin{gather*}
a_1(\bu, \bv)  \le  
 \| \sqrt{2\mu}\beps(\bu) \|_{0,\Omega} \| \sqrt{2\mu}\beps (\bv) \|_{0,\Omega} \lesssim \| \sqrt{2\mu} \bu \|_{1,\Omega} \| \sqrt{2\mu}\bv \|_{1,\Omega}, \\
b_1(\bv, \psi)   \le \| \vdiv \bv \|_{0,\Omega} \| \psi \|_{0,\Omega} \lesssim \| \sqrt{2\mu} \bv \|_{1,\Omega} \| \frac{1}{\sqrt{2\mu}}\psi \|_{0,\Omega}, \\
a_2(p^\mathrm{P}, q^\mathrm{P})  \le  \| \sqrt{\frac{\kappa}{ \eta}}  \nabla p^\mathrm{P}\|_{0,\OmP} \| \sqrt{\frac{\kappa}{ \eta}} \nabla q^\mathrm{P}\|_{0,\OmP} 
, \qquad a_3(\varphi, \psi)  \le  \| \frac{1}{\sqrt{\lambda}} \varphi \|_{0,\Omega} \|  \frac{1}{\sqrt{\lambda}}  \psi \|_{0,\Omega},\\
b_2(q^\mathrm{P}, \psi)   \le \| \frac{\alpha}{\sqrt{\lambda^\mathrm{P}}}  q^\mathrm{P} \|_{0,\OmP} \| \frac{1}{\sqrt{\lambda^\mathrm{P}}}  \psi^\rmP \|_{0,\OmP}, \qquad 
F(\bv)   \lesssim \| \bb \|_{0,\Omega} \| \bv \|_{1,\Omega}, \qquad  G(q^\mathrm{P})   \le  \| \ell^\rmP \|_{0,\OmP} \| q^\mathrm{P} \|_{0,\OmP}, 
\end{gather*}
for all $\bu, \bv \in \bV$, $\psi,\varphi \in \rZ$, $p^\mathrm{P}, q^\rmP \in \rQ^\rmP$.
There also holds coercivity of the diagonal bilinear forms 
\begin{equation*}
a_1(\bv, \bv)  \geq  \| \sqrt{2\mu}\beps(\bv) \|_{0,\Omega}^2 \gtrsim \| \sqrt{2\mu} \bv \|_{1,\Omega}^2, \quad 
{|a_2(q^\rmP, q^\rmP)|}  \ge   \|  \sqrt{\frac{\kappa}{ \eta}} \nabla q^\rmP \|_{0,\OmP}^2,\quad 
a_3(\psi, \psi)  \geq  \| \frac{1}{\sqrt{\lambda}}\psi \|_{0,\Omega}^2, 
\end{equation*}
for all $\bv \in \bV$, $q^\rmP \in \rQ^\rmP$, $\psi \in \rZ$, 
and the following inf-sup condition (see, e.g., \cite{ern04}): There exists   $\xi >0$ such that 
\begin{equation}\label{eq:inf-sup}
\sup_{\bv (\neq \cero) \in \bV} \frac{b_1(\bv, \psi)}{\| \bv \|_{1,\Omega}} \ge \xi \| \psi \|_{0,\Omega} \qquad  \forall \psi \in \rZ .
\end{equation}
Details on the unique solvability of the continuous problem can be found in \cite{gira11,girault15}, or, for the steady case with rotation-based formulations, in \cite{anaya20,anaya22} (but in those references the analysis assumes that the pay-zone poroelastic subdomain is completely confined by the elastic structure).

Similarly to the relevant inf-sup condition \eqref{eq:inf-sup}, we have that for each $\varphi_0\in \rL^2(\Omega)$ with $\varphi_0|_{\Omega^\mathrm{E}}=\varphi_0^{\mathrm{E}}\in\rL^2(\Omega^{\mathrm{E}})$ and $\varphi_0|_{\Omega^{\mathrm{P}}}=\varphi_0^{\mathrm{P}}\in\rL^2(\Omega^{\mathrm{P}})$,
we can find $\bv_0^{\mathrm{E}}\in\bH^1_{\Gamma_D^\rmE,0}(\Omega^\mathrm{E})$ and $\bv_0^{\mathrm{P}}\in\bH^1_{\Gamma_D^\rmP,0}(\Omega^\mathrm{P})$, where $\bH^1_{\Gamma_D^\rmE,0}(\Omega^\mathrm{E})= \{\bv : \bv\in \bH^1_{\Gamma_D^\rmE}(\Omega^\mathrm{E}) \;\mbox{and}\; \bv|_{\Sigma} = \cero \}$ and $\bH^1_{\Gamma_D^\rmP,0}(\Omega^\mathrm{P})= \{\bv : \bv\in \bH^1_{\Gamma_D^\rmP}(\Omega^\mathrm{P}) \;\mbox{and}\; \bv|_{\Sigma} = \cero \}$,
such that
\begin{align*}
&(\vdiv\bv_0^{\mathrm{E}}, \varphi^{\mathrm{E}}_0)_{0,\Omega^{\mathrm{E}}}\ge C_{\Omega^{\mathrm{E}}}/ \mu^{\mathrm{E}}\|\varphi^{\mathrm{E}}_0\|_{0,\Omega^{\mathrm{E}}}^2, \quad \sqrt{2\mu^{\mathrm{E}}}\|\bnabla \bv_0^{\mathrm{E}}\|_{0,\Omega^{\mathrm{E}}}\le 1/\sqrt{2\mu^{\mathrm{E}}}\|\varphi^{\mathrm{E}}_0\|_{0,\Omega^{\mathrm{E}}},\\
&(\vdiv\bv_0^{\mathrm{P}}, \varphi^{\mathrm{P}}_0)_{0,\Omega^{\mathrm{P}}}\ge C_{\Omega^{\mathrm{P}}}/ \mu^{\mathrm{P}}\|\varphi^{\mathrm{P}}_0\|_{0,\Omega^{\mathrm{P}}}^2, \quad \sqrt{2\mu^{\mathrm{P}}}\|\bnabla \bv_0^{\mathrm{P}}\|_{0,\Omega^{\mathrm{P}}}\le 1/\sqrt{2\mu^{\mathrm{P}}}\|\varphi^{\mathrm{P}}_0\|_{0,\Omega^{\mathrm{P}}}.
\end{align*}
Hence,  there exists $\bv_0\in\bV$ such that $\bv_0|_{\Omega^{\mathrm{E}}}=\bv_0^{\mathrm{E}}$ and $\bv_0|_{\Omega^{\mathrm{P}}}=\bv_0^{\mathrm{P}}$. Moreover we have
\[\sup_{\cero\neq \bv\in\bV}\frac{b_{1}(\bv,\varphi_0)}{\|\sqrt{2\mu}\bv\|_{1,\Omega}}\ge\tilde{C} \|\frac{1}{\sqrt{2\mu} }\varphi_0\|_{0,\Omega},\]
or, following also \cite{olshanskii06}, we can write 
\[\sup_{\cero\neq \bv\in\bV}\frac{b_{1}(\bv,\varphi_0)}{\|\sqrt{2\mu}\beps(\bv)\|_{0,\Omega}}\ge\tilde{C} \|\frac{1}{\sqrt{2\mu} }\varphi_0\|_{0,\Omega},\]
for a positive constant $\tilde{C}$ independent of $\mu$.

\section{An H(div)-conforming finite element approximation} \label{sec:FE}
 
We denote by  $\{{\mathcal T}_h^\mathrm{P}\}_h$ and $\{{\mathcal T}_h^\mathrm{E}\}_h$  sequences of triangular (or tetrahedral in 3D) partitions of the poroelastic and elastic subdomains $\Omega^\mathrm{P}$ and   $\Omega^\mathrm{E}$, respectively  having diameter $h_K$, and being such that the partitions are conforming with the interface $\Sigma$.   We   label by  $K^-$ and $K^+$ the  two elements adjacent to a facet (an edge in 2D or a face in 3D), while $h_e$  stands for the maximum diameter of the facet.  By  $\cE_h$ we will denote the set of all facets and will distinguish between facets lying on the elastic, poroelastic, and interfacial regions $\cE_h = \cE_h^\mathrm{E} \cup \cE_h^\mathrm{P} \cup \cE_h^\Sigma$. 

For a smooth vector, scalar, or tensor field $w$ defined on~$\cT_h$,  $w^\pm$  denote its traces taken from the interior of~$K^+$ and~$K^-$, respectively. We also denote by $\nn^\pm$ the outward unit normal vector to $K^\pm$. The symbols  $\mean{\cdot}$ and $\jump{\cdot }$ denote, respectively, the average and jump operators, defined as 
\begin{align}\label{eq:mean-jump}
  \mean{w} \coloneqq \frac12 (w^-+w^+), \quad  \jump{w \odot \nn } \coloneqq  (w^- \odot \nn^- + w^+ \odot \nn^+),
\end{align}
for a generic multiplication operator $\odot$, which applies to interior edges, whereas for boundary jumps and averages we adopt the conventions $\mean{w} = w$, and $\jump{w \odot \nn} = w \odot \nn$. The element-wise action of a differential operator is denoted with a subindex $h$, for example, $\nabla_h$, $\bnabla_h$ will  denote  the broken gradient operators for scalar and vector quantities, respectively and $\beps_h(\cdot) = \frac12(\bnabla_h \cdot + (\bnabla_h\cdot)^T)$ is the symmetrised vector broken gradient. 

Let $\bbP_k(K)$ denote the local space spanned by polynomials of degree up to  $k \geq 0$, and  let us consider the following discrete spaces 
\begin{equation}\label{eq:discrete-spaces0}
  \begin{split}
\bV_h &\coloneqq \bigl\{ \bv_h \in \bH(\vdiv; \Omega) : \bv_h|_K \in [\bbP_{k+1}(K)]^{d}\quad\forall K \in \cT_h, \quad\bv_h\cdot\nn|_{\Gamma_D^\mathrm{E}\cup\Gamma_D^\mathrm{P}} = 0 \bigr\},\\
\rQ^\mathrm{P}_h &\coloneqq \bigl\{ q^\mathrm{P}_h \in \rQ^\mathrm{P} : q_h|_K \in \bbP_{k+1}(K)\quad\forall K \in \cT_h^\mathrm{P} \bigr\}, \\
 \rZ_h  & \coloneqq \bigl\{ \psi_h \in \rZ: \psi_h|_K \in  \bbP_{k}(K)\quad\forall K \in \cT_h\bigr\}, 
\end{split}\end{equation}
which, in particular, satisfy  the so-called equilibrium property
\begin{equation}\label{eq:equilibrium}
  \vdiv\bV_h 
  =   \rZ_h.\end{equation}
Note that in this case $\bV_h$ is the space of divergence-conforming BDM  elements \cite{bdm85}, and it is not conforming with $\bV$.  Its basic approximation property, locally on $K \in \cT_h$, is that for all $\bv\in \bH^s(K)$, there exists an interpolant $\bv_I \in \bV_h(K)$ such that
\begin{equation}\label{eq:approx}
 \| \bv-\bv_I\|_{0,K} + h_K|\bv-\bv_I|_{1,K} + h_K^2 |\bv-\bv_I|_{2,K} \lesssim h_K^s|\bv|_{s,K}, \qquad 2 \leq s \leq k,\end{equation}
see, e.g., \cite{bdm85,buerger19,konno11}. 

\subsection{Formulation with continuous fluid pressure}
The  Galerkin finite element formulation  then reads:  Find $(\bu_h, p^\mathrm{P}_h, \varphi_h) \in \bV_h \times \rQ^\mathrm{P}_h \times \rZ_h $ such that:
\begin{subequations}\label{semidis11}
\begin{alignat}{5}
&&   a_1^h(\bu_h,\bv_h)   &&                 &\;+&\; {b}_1(\bv_h,\varphi_h)     &=&\;F(\bv_h)&\quad\forall \bv_h\in\bV_h, \label{weak-u-h-11}\\
{-}\tilde{a}_2(p_h^\mathrm{P},q_h^\mathrm{P}) &\; {-}&               &&      a_2(p_h^\mathrm{P},q_h^\mathrm{P})   &\;{+}&\;    b_2(q_h^\mathrm{P}, \, \varphi_h) &=&\;G(q_h^\mathrm{P}) &\quad\forall q_h^\mathrm{P}\in \rQ_h^\mathrm{P}, \label{weak-p-h-11}\\
&&{b}_1(\bu_h,\psi_h)  &\;+\;& b_2(p_h^\mathrm{P},\psi_h)&\;-&\; a_3(\varphi_h,\psi_h) &=&0 &\quad\forall\psi_h\in \rZ_h, \label{weak-psi-h-11} 
\end{alignat}\end{subequations}
where $a_1^h(\cdot,\cdot)$ is the discrete version of the bilinear form $a_1(\cdot,\cdot)$ and it is defined using a symmetric interior penalty from \cite{konno11} (see also \cite{kanschat10} for its use in the context of poroelasticity)
\begin{align}
 a^h_1(\bu_h, \bv_h)  
& \coloneqq 2(\mu \beps_h(\bu_h), \beps_h(\bv_h)  )_{0,\Omega} 
- 2 \sum_{e\in\cE_h{\cup \Gamma_D^*}} \bigl( \langle \mean{\mu \beps_h(\bu_h)} , \jump{\bv_h \otimes \nn} \rangle_{0,e} 
+ \langle \mean{\mu\beps_h(\bv_h)} , \jump{\bu_h \otimes \nn}\rangle_{0,e}  \bigr)  \nonumber \\
& \quad + 2\sum_{e\in\cE_h{\cup \Gamma_D^*}\setminus \cE_h^\Sigma}  \frac{\beta_{\bu}}{h_e} \langle \mu \jump{\bu_h\otimes\nn}, \jump{\bv_h\otimes\nn} \rangle_{0,e} + 2 \sum_{e\in\cE_h^\Sigma}  \frac{\beta_{\bu}}{h_e} \langle \mu^0 \jump{\bu_h\otimes\nn}, \jump{\bv_h\otimes\nn} \rangle_{0,e}, \label{a1hdef11}
\end{align} 
where $\Gamma_D^* = \Gamma_D^{\mathrm{P}}\cup \Gamma_D^{\mathrm{E}}$ denotes the part of the boundary where Dirichlet conditions are imposed on displacement, $\beta_{\bu}>0$ is a parameter penalising the (tangential) displacement  jumps (and also serving as a Nitsche parameter to enforce the tangential part of the displacement boundary condition) so that the bilinear form $a^h_1(\cdot,\cdot)$ is positive definite, and 
$\mu^0=\max\{\mu^{\mathrm{E}},\mu^{\mathrm{P}}\}$.
Now we write down the above weak formulation \eqref{semidis11} in the following compact form: Find $(\bu_h, p^\mathrm{P}_h, \varphi_h) \in \bV_h \times \rQ^\mathrm{P}_h \times \rZ_h $ such that
\begin{align*}
M_h(\bu_h, p^\mathrm{P}_h, \varphi_h; \bv_h, q^\mathrm{P}_h, \psi_h)&=F(\bv_h)+G(q_h^\mathrm{P}) \quad\forall (\bv_h, q^\mathrm{P}_h, \psi_h) \in \bV_h \times \rQ^\mathrm{P}_h \times \rZ_h,
\end{align*}
where the multilinear form is defined as 
\begin{align*}
M_h(\bu_h, p^\mathrm{P}_h, \varphi_h; \bv_h, q^\mathrm{P}_h, \psi_h)&= a_1^h(\bu_h,\bv_h)+{b}_1(\bv_h,\varphi_h)  {-}\tilde{a}_2(p_h^\mathrm{P},q_h^\mathrm{P}){-}a_2(p_h^\mathrm{P},q_h^\mathrm{P})\\
& \qquad+b_2(q_h^\mathrm{P}, \, \varphi_h)+{b}_1(\bu_h,\psi_h)+ b_2(p_h^\mathrm{P},\psi_h)-a_3(\varphi_h,\psi_h).
\end{align*}
Continuity and coercivity also hold for the modified bilinear form $a_1^h(\cdot,\cdot)$, but they do over the discrete space $\bV_h$ and with respect to the following mesh-dependent and parameter-dependent broken norms 
\begin{align}\label{eq:h-norms}
\norm{\bv_h}_{*,\cT_h}^2  &\coloneqq \sum_{K\in \cT_h} \norm{\sqrt{2\mu}\beps_h(\bv_h)}_{0,K}^2 + \sum_{e\in\cE_h\cup \Gamma_D^*\setminus \cE_h^\Sigma}2{\mu} \frac{\beta_{\bu}}{h_e} \norm{ \jump{\bv_h\otimes\nn } }_{0,e}^2 + \sum_{e\in\cE_h^\Sigma}2{\mu_0} \frac{\beta_{\bu}}{h_e} \norm{ \jump{\bv_h\otimes\nn } }_{0,e}^2, \nonumber \\ 
\norm{\bv_h}_{1,\cT_h}^2 &  \coloneqq \norm{\sqrt{2\mu} \bv_h}_{0,\Omega}^2 + \norm{\bv_h}_{*,\cT_h}^2 \  \text{for all $\bv_h \in \bH^1(\cT_h)$,}  \\
\norm{\bv_h}_{2,\cT_h}^2 &\coloneqq \norm{\bv_h}_{1,\cT_h}^2 + \sum_{K\in \cT_h}  h_K^2 |\sqrt{2\mu}\bv|_{2,K}^2 \  \text{for all $\bv_h \in \bH^2(\cT_h)$,} \nonumber
\end{align}
where, for a generic $s\geq 0$, the broken vectorial Sobolev spaces are defined as 
\[
\bH^s(\cT_h) \coloneqq \bigl\{ \bv_h \in \bL^2(\Omega) : \bv_h|_K \in \bH^s(K),\ K \in \cT_h \bigr\}. 
\]
Thanks to the discrete version of Korn's inequality \cite[eq. (1.12)]{brenner04}, the norms above are uniformly equivalent over the appropriate spaces.  
Similarly, 
for the present choice of finite element spaces, a discrete inf-sup condition for $b_1(\cdot,\cdot)$ holds naturally in the discrete norm $\norm{\cdot}_{1,\cT_h}$ when we do not have a weighting with the space-dependent parameter $\mu$  (cf. \cite{konno11}). However, combined with the arguments in Section~\ref{sec:properties}, we can assert that there exists  $\hat{\xi}>0$, independent of $h$ and of $\mu$, such that
\begin{equation*}
\sup_{\bv_h (\neq \cero) \in \bV_h} \frac{b_1(\bv_h, \psi_h)}{\| \bv_h \|_{1,\cT_h}} \ge \hat{ \xi} \| \frac{1}{\sqrt{2\mu}}\psi_h \|_{0,\Omega} \qquad \forall \psi_h \in \rZ_h. 
\end{equation*}

\subsection{Formulation with discontinuous  fluid pressure}
Consider now the spaces
\begin{equation}\label{eq:discrete-spaces}
  \begin{split}
\bV_h \coloneqq \bigl\{ \bv_h \in \bH(\vdiv; \Omega) : \bv_h|_K \in [\bbP_{k+1}(K)]^{d}\quad\forall K \in \cT_h, \quad\bv_h\cdot\nn|_{\Gamma_D^\mathrm{E}\cup\Gamma_D^\mathrm{P}} = 0 \bigr\},\\
\widetilde{\rQ}^\mathrm{P}_h \coloneqq \bigl\{ q^\mathrm{P}_h \in \rL^2(\Omega^\mathrm{P}) : q_h|_K \in \bbP_{k+1}(K)\quad\forall K \in \cT_h^\mathrm{P} \bigr\}, \qquad 
 \rZ_h   \coloneqq \bigl\{ \psi_h \in \rZ: \psi_h|_K \in  \bbP_{k}(K)\quad\forall K \in \cT_h\bigr\}.
\end{split}\end{equation}
The formulation  in this case reads:  Find $(\bu_h, p^\mathrm{P}_h, \varphi_h) \in \bV_h \times \widetilde{\rQ}^\mathrm{P}_h \times \rZ_h $ such that:
\begin{subequations}\label{semidis12}
\begin{alignat}{5}
&&   a_1^h(\bu_h,\bv_h)   &&                 &\;+&\; {b}_1(\bv_h,\varphi_h)     &=&\;F(\bv_h)&\quad\forall \bv_h\in\bV_h, \label{weak-u-h-12}\\
{-}\tilde{a}_2(p_h^\mathrm{P},q_h^\mathrm{P}) &\; {-}&               &&      a_2^{h}(p_h^\mathrm{P},q_h^\mathrm{P})   &\;{+}&\;    b_2(q_h^\mathrm{P}, \, \varphi_h) &=&\;G(q_h^\mathrm{P}) &\quad\forall q_h^\mathrm{P}\in \widetilde{\rQ}_h^\mathrm{P}, \label{weak-p-h-12}\\
&&{b}_1(\bu_h,\psi_h)  &\;+\;& b_2(p_h^\mathrm{P},\psi_h)&\;-&\; a_3(\varphi_h,\psi_h) &=&0 &\quad\forall\psi_h\in \rZ_h, \label{weak-psi-h-12} 
\end{alignat}\end{subequations}
where $a_2^{h}(\cdot,\cdot)$ is the discrete version of the bilinear form $a_2(\cdot,\cdot)$ and it is defined using a symmetric interior penalty from, e.g., the classical paper \cite{arnold82} 
\begin{align}
 a^{h}_2(p_h^\mathrm{P},q_h^\mathrm{P})  
& \coloneqq \biggl(\frac{\kappa}{\eta} \nabla_hp_h^\mathrm{P},\nabla_hq_h^\mathrm{P}\biggr)_{0,\Omega} 
-  \sum_{e\in\cE_h^{\mathrm{P}}{\cup \Gamma_D^\mathrm{P}} } \biggl( \biggl\langle \mean{\frac{\kappa}{\eta} \nabla_hp_h^\mathrm{P} } , \jump{q_h^\mathrm{P}\nn} \biggr\rangle_{0,e} 
+ \biggl\langle \mean{\frac{\kappa}{\eta}\nabla_h q_h^{\mathrm{P}}} , \jump{p_h^{\mathrm{P}} \nn}\biggr\rangle_{0,e}  \biggr)  \nonumber \\
& \quad +  \sum_{e\in\cE_h^{\mathrm{P}}{\cup \Gamma_D^\mathrm{P}}}  \frac{\beta_{p^\mathrm{P}}}{h_e} \biggl\langle \frac{\kappa}{\eta} \jump{p_h^{\mathrm{P}} \nn }, \jump{q_h^{\mathrm{P}} \nn} \biggr\rangle_{0,e}, \label{a1hdef12}
\end{align} 
where $\beta_{p^\mathrm{P}}>0$ is a parameter penalising the pressure  jumps. If $\beta_{p^\mathrm{P}}$ is sufficiently large, this yields the coercivity of the bilinear form $ a^{h}_2$. 

Next, and as done for the case of continuous pressure approximation, we write down the compact form of the above weak formulation (\ref{semidis12}): Find $(\bu_h, p^\mathrm{P}_h, \varphi_h) \in \bV_h \times \widetilde{\rQ}^\mathrm{P}_h \times \rZ_h $ such that
\begin{align*}
\widetilde{M}_h(\bu_h, p^\mathrm{P}_h, \varphi_h; \bv_h, q^\mathrm{P}_h, \psi_h)&=F(\bv_h)+G(q_h^\mathrm{P}) \quad\forall (\bv_h, q^\mathrm{P}_h, \psi_h) \in \bV_h \times \widetilde{\rQ}^\mathrm{P}_h \times \rZ_h,
\end{align*}
where
\begin{align*}
\widetilde{M}_h(\bu_h, p^\mathrm{P}_h, \varphi_h; \bv_h, q^\mathrm{P}_h, \psi_h)&= a_1^h(\bu_h,\bv_h)+{b}_1(\bv_h,\varphi_h)  {-}\tilde{a}_2(p_h^\mathrm{P},q_h^\mathrm{P}){-}a_2^h(p_h^\mathrm{P},q_h^\mathrm{P})\\ & \qquad +b_2(q_h^\mathrm{P}, \, \varphi_h)+{b}_1(\bu_h,\psi_h) + b_2(p_h^\mathrm{P},\psi_h)-a_3(\varphi_h,\psi_h).
\end{align*}

\section{Unique solvability of the discrete problems and \textit{a priori} error estimates} \label{sec:apriori}
\subsection{Well-posedness analysis for formulation (\ref{semidis11})}
We proceed by means of a Fortin argument and consider the canonical interpolation operator $\Pi_h: \bV\rightarrow \bV_h$ such that 
\begin{subequations}\begin{align}
b_1(\bv-\Pi_h\bv,\varphi_h) & =0\qquad\forall \varphi_h\in \rZ_h,\label{property1}\\
|\bv-\Pi_h\bv|_{s,K} & \le Ch^{t-s}|\bv|_{t,K}\quad\forall K\in\mathcal{T}_h,\label{property2}
\end{align}\end{subequations}
where $C$ is a positive constant which depends only on the shape of $K$ and $1\le t\le r+1$ (see, e.g., \cite{boffi13}).

Using the trace inequality and property \eqref{property2}, we have the following bound in one of the norms from \eqref{eq:h-norms} 
\[\|\bv-\Pi_h\bv\|_{\ast,\cT_h}\le \|\bv\|_{\ast,\cT_h}\quad\forall \bv\in \bV.\]
Moreover, we have
\[\|\Pi_h \bv\|_{\ast,\cT_h}\le C\|\bv\|_{\ast,\cT_h}.\]

With these properties satisfied by the operator $\Pi_h$,  we can use the continuous inf-sup condition to readily show that there exists $\xi>0$ such that 
a discrete inf-sup condition for the bilinear form $b_1$: 
\begin{align*}
\sup_{\bv\in\bV_h\setminus\{\cero\}}\frac{b_1(\bv,\varphi_h)}{\|\bv\|_{\ast,\cT_h}}\ge \sup_{\Pi_h\bv\in\bV_h\setminus\{\cero\}}\frac{b_1(\Pi_h\bv,\varphi_h)}{\|\Pi_h\bv\|_{\ast,\cT_h}}=\sup_{\bv\in\bV\setminus\{\cero\}}\frac{b_1(\bv,\varphi_h)}{\|\Pi_h\bv\|_{\ast,\cT_h}}\ge \frac{1}{C}\sup_{\bv\in\bV\setminus\{\cero\}}\frac{b_1(\bv,\varphi_h)}{\|\bv\|_{\ast,\cT_h}}\ge \xi \|\frac{1}{\sqrt{2\mu}}\varphi_h\|_{0,\Omega}\quad\forall \varphi_h\in \rZ_h.
\end{align*}

We are now in a position to establish a global inf-sup condition.  
\begin{theorem}\label{new11}
For every $(\bu_h, p^\mathrm{P}_h, \varphi_h) \in \bV_h \times \rQ^\mathrm{P}_h \times \rZ_h$, there exists $(\bv_h, q^\mathrm{P}_h, \psi_h) \in \bV_h \times \rQ^\mathrm{P}_h \times \rZ_h$ with $\VERT(\bv_h, q^\mathrm{P}_h, \psi_h)\VERT\lesssim 
\VERT(\bu_h, p^\mathrm{P}_h, \varphi_h)\VERT$ such that
\[
M_h(\bu_h, p^\mathrm{P}_h, \varphi_h; \bv_h, q^\mathrm{P}_h, \psi_h)\gtrsim 
\VERT(\bu_h, p^\mathrm{P}_h, \varphi_h)\VERT^2,
\]
where 
\begin{equation}\label{eq:triplenorm}
\VERT(\bv_h, q^\mathrm{P}_h, \psi_h)\VERT^2:=\norm{\bv_h}_{*,\cT_h}^2 + \|\frac{1}{\sqrt{2\mu}}\psi_h\|^2_{0,\Omega}+\frac{1}{\lambda^{\mathrm{E}}}\|\psi_h\|^2_{0,\Omega^{\mathrm{E}}}+\frac{1}{\lambda^{\mathrm{P}}}\|\psi_h-\alpha q^\mathrm{P}_h\|^2_{0,\Omega^{\mathrm{P}}}+c_0\|q^\mathrm{P}_h\|^2_{0,\Omega^{\mathrm{P}}}+\|\frac{\kappa}{\eta} \nabla q^\mathrm{P}_h\|^2_{0,\Omega^{\mathrm{P}}}.\end{equation}

Moreover, we have that
\[\bigl|M_h(\bu_h, p^\mathrm{P}_h, \varphi_h; \bv_h, q^\mathrm{P}_h, \psi_h)\bigr|\lesssim 
\VERT(\bu_h, p^\mathrm{P}_h, \varphi_h)\VERT\VERT(\bv_h, q^\mathrm{P}_h, \psi_h)\VERT,\]
for all $(\bu_h, p^\mathrm{P}_h, \varphi_h), (\bv_h, q^\mathrm{P}_h, \psi_h) \in \bV_h \times \rQ^\mathrm{P}_h \times \rZ_h$.
\end{theorem}
\begin{proof}
Let $(\bu_h, p^\mathrm{P}_h, \varphi_h)\in \bV_h \times \rQ^\mathrm{P}_h \times \rZ_h$ be arbitrary. Using the definition of the multilinear form $M_h$, we easily obtain  
\begin{align*}
&M_{h}(\bu_h, p^\mathrm{P}_h, \varphi_h; \bv_h, 0, 0)=a_1^h(\bu_h,\bv_h)+{b}_1(\bv_h,\varphi_h)
 \ge \left(\xi-\frac{1}{2\epsilon_1}\right)\|\frac{1}{\sqrt{2\mu}}\varphi_h\|_{0,\Omega}^{2}-\frac{\epsilon_1}{2}\|\bu_h\|_{\ast,\cT_h}.
\end{align*}
Selecting  $\bv=\bu_h$, $q^{\mathrm{P}}=-p_h^{\mathrm{P}}$ and $\psi=-\varphi_h$ we have
\begin{align*}
M_h(\bu_h, p^\mathrm{P}_h, \varphi_h,\bu_h, {-}p^\mathrm{P}_h, -\varphi_h)&= a_1^h(\bu_h,\bu_h)+{\tilde{a}_2(p_h^\mathrm{P},p_h^\mathrm{P})}+a_2(p_h^\mathrm{P},p_h^\mathrm{P})-2b_2(p_h^\mathrm{P}, \, \varphi_h)+a_3(\varphi_h,\varphi_h)\nonumber\\
&\ge C_2\|\bu_h\|_{\ast,\cT_h}^2+c_0\|p_h^\mathrm{P}\|_{0,\Omega^{\mathrm{P}}}^2+1/\lambda^{\mathrm{E}}\|\varphi_h\|^2_{0,\Omega^{\mathrm{E}}}+1/\lambda^{\mathrm{P}}\|\varphi_h-\alpha p^\mathrm{P}_h\|^2_{0,\Omega^{\mathrm{P}}}\\
&\qquad+c_0\|p^\mathrm{P}_h\|^2_{0,\Omega^{\mathrm{P}}}+\|\kappa/\eta \nabla p^\mathrm{P}_h\|^2_{0,\Omega^{\mathrm{P}}}.
\end{align*}
Then we can make the choice $\bv=\bu_h+\delta_1 \bv_h$, $q^{\mathrm{P}}=-p_h^{\mathrm{P}}$ and $\psi=-\varphi_h$, leading to 
\begin{align*}
&M_h(\bu_h, p^\mathrm{P}_h, \varphi_h,\bu_h+\delta_1\bv_h, -p^\mathrm{P}_h, -\varphi_h)\\
&=M_h(\bu_h, p^\mathrm{P}_h, \varphi_h,\bu_h, -p^\mathrm{P}_h, -\varphi_h)+\delta_1M_h(\bu_h, p^\mathrm{P}_h, \varphi_h,\bv_h, 0, 0)\\
&\ge C_2\|\bu_h\|_{\ast,\cT_h}^2+c_0\|p_h^\mathrm{P}\|_{0,\Omega^{\mathrm{P}}}^2+1/\lambda^{\mathrm{E}}\|\varphi_h\|^2_{0,\Omega^{\mathrm{E}}}+1/\lambda^{\mathrm{P}}\|\varphi_h-\alpha p^\mathrm{P}_h\|^2_{0,\Omega^{\mathrm{P}}}\\
&\qquad+c_0\|p^\mathrm{P}_h\|^2_{0,\Omega^{\mathrm{P}}}+\|\kappa/\eta (\nabla p^\mathrm{P}_h)\|^2_{0,\Omega^{\mathrm{P}}}+\delta_1\left(\xi-\frac{1}{2\epsilon_1}\right)\|\frac{1}{\sqrt{2\mu}}\varphi_h\|_{0,\Omega}^{2}-\delta_1\frac{\epsilon_1}{2}\|\bu_h\|_{\ast,\cT_h}\\
&\ge \left(C_2-\delta_1\frac{\epsilon_1}{2}\right)\|\bu_h\|_{\ast,\cT_h}^2+c_0\|p_h^\mathrm{P}\|_{0,\Omega^{\mathrm{P}}}^2+1/\lambda^{\mathrm{E}}\|\varphi_h\|^2_{0,\Omega^{\mathrm{E}}}+1/\lambda^{\mathrm{P}}\|\varphi_h-\alpha p^\mathrm{P}_h\|^2_{0,\Omega^{\mathrm{P}}}\\
&\qquad+c_0\|p^\mathrm{P}_h\|^2_{0,\Omega^{\mathrm{P}}}+\|\kappa/\eta (\nabla p^\mathrm{P}_h)\|^2_{0,\Omega^{\mathrm{P}}}+\delta_1\left(\xi-\frac{1}{2\epsilon_1}\right)\|\frac{1}{\sqrt{2\mu}}\varphi_h\|_{0,\Omega}^{2},
\end{align*}
where we have used Young's inequality. Assuming the values $\epsilon_1=1/\xi$ and $\delta_1 = C_2/\epsilon_1$, we then have
\[M_h(\bu_h, p^\mathrm{P}_h, \varphi_h; \bv_h, q^\mathrm{P}_h, \psi_h)\ge  \frac{1}{2}\min\left\{C_2\xi^2,C_2\right\} \VERT(\bu_h, p^\mathrm{P}_h, \varphi_h)\VERT^2,\]
and  the first part of the proof concludes after realising that 
\[
\VERT(\bv_h, q^\mathrm{P}_h, \psi_h)\VERT^2= \VERT(\bu_h+\delta_1\bv, -p^\mathrm{P}_h, -\varphi_h)\VERT^2\le 2 \VERT(\bu_h, p^\mathrm{P}_h, \varphi_h)\VERT^2.
\]
For the continuity property, it suffices to apply 
Cauchy--Schwarz inequality and the definition of $M_h$.
\end{proof}

\begin{lemma}\label{lem:aux04}
Let $(\tilde{\bu},\tilde{p}^\mathrm{P},\tilde{\varphi})$ be a generic triplet in $\bV_h \times \rQ^\mathrm{P}_h \times \rZ_h$. Then the following estimate holds
  \[\VERT(\bu-\bu_h, p-p^\mathrm{P}_h, \varphi-\varphi_h)\VERT\lesssim \VERT(\bu-\tilde{\bu}, p^{\mathrm{P}}-\tilde{p}^\mathrm{P}, \varphi-\tilde{\varphi})\VERT+\biggl( \sum_{K\in \cT_h}  h_K^2 |\sqrt{2\mu}(\bu-\tilde{\bu})|_{2,K}^2\biggr)^{1/2}.\]
\end{lemma}
\begin{proof}
Directly from triangle inequality we have 
\[\VERT(\bu-\bu_h, p-p^\mathrm{P}_h, \varphi-\varphi_h)\VERT\le\VERT(\bu-\tilde{\bu}, p^{\mathrm{P}}-\tilde{p}^\mathrm{P}, \varphi-\tilde{\varphi})\VERT+\VERT(\tilde{\bu}-\bu_h, \tilde{p}^{\mathrm{P}}-p^\mathrm{P}_h, \tilde{\varphi}-\varphi_h)\VERT.\]
Using Theorem \ref{new11} and the properties of $M_h$ gives
\begin{align*}
\VERT(\tilde{\bu}-\bu_h, \tilde{p}^{\mathrm{P}}-p^\mathrm{P}_h, \tilde{\varphi}-\varphi_h)\VERT^2&\lesssim M_h(\tilde{\bu}-\bu_h, \tilde{p}^{\mathrm{P}}-p^\mathrm{P}_h, \tilde{\varphi}-\varphi_h; \bv_h, q^\mathrm{P}_h, \psi_h) \\
&\le M_h(\tilde{\bu}, \tilde{p}^{\mathrm{P}}, \tilde{\varphi}; \bv_h, q^\mathrm{P}_h, \psi_h)-M_h(\bu_h, p^\mathrm{P}_h, \varphi_h; \bv_h, q^\mathrm{P}_h, \psi_h)\\
&\le M_h(\tilde{\bu}-\bu, \tilde{p}^{\mathrm{P}}-p^\mathrm{P}, \tilde{\varphi}-\varphi; \bv_h, q^\mathrm{P}_h, \psi_h)\\
&\lesssim \VERT(\bu-\tilde{\bu}, p^{\mathrm{P}}-\tilde{p}^\mathrm{P}, \varphi-\tilde{\varphi})\VERT\\
&\qquad \qquad +\biggl( \sum_{K\in \cT_h} h_K^2 |\sqrt{2\mu}(\bu-\tilde{\bu})|_{2,K}^2\biggr)^{1/2}\VERT(\tilde{\bu}-\bu_h, \tilde{p}^{\mathrm{P}}-p^\mathrm{P}_h, \tilde{\varphi}-\varphi_h)\VERT.
\end{align*}
\end{proof}

\subsection{Well-posedness analysis for formulation (\ref{semidis12})}
Proceeding similarly to the proof of Theorem \ref{new11}, we can establish the following result. 
\begin{theorem}\label{new12}
For every $(\bu_h, p^\mathrm{P}_h, \varphi_h) \in \bV_h \times \widetilde{\rQ}^\mathrm{P}_h \times {Z}_h$, there exists $(\bv_h, q^\mathrm{P}_h, \psi_h) \in \bV_h \times \widetilde{\rQ}^\mathrm{P}_h \times {Z}_h$ with $\VERT(\bv_h, q^\mathrm{P}_h, \psi_h)\VERT\lesssim\VERT(\bu_h, p^\mathrm{P}_h, \varphi_h)\VERT$ such that
\[
\widetilde{M}_h(\bu_h, p^\mathrm{P}_h, \varphi_h; \bv_h, q^\mathrm{P}_h, \psi_h)\gtrsim \VERT(\bu_h, p^\mathrm{P}_h, \varphi_h)\VERT^2_{\ast},
\]
where 
\[\VERT(\bv_h, q^\mathrm{P}_h, \psi_h)\VERT_{\ast}^2:=\norm{\bv_h}_{*,\cT_h}^2 + \|\frac{1}{\sqrt{2\mu}}\psi_h\|^2_{0,\Omega}+\frac{1}{\lambda^{\mathrm{E}}}\|\psi_h\|^2_{0,\Omega^{\mathrm{E}}}+\frac{1}{\lambda^{\mathrm{P}}}\|\psi_h-\alpha q^\mathrm{P}_h\|^2_{0,\Omega^{\mathrm{P}}}+c_0\|q^\mathrm{P}_h\|^2_{0,\Omega^{\mathrm{P}}}+\|q^\mathrm{P}_h\|^2_{\ast,\Omega^{\mathrm{P}}},\]
and
\[\norm{q^\mathrm{P}_h}_{\ast,\Omega^{\mathrm{P}}}^2  \coloneqq \sum_{K\in \cT_h^{\mathrm{P}}} \norm{\kappa/\eta(\nabla q^\mathrm{P}_h)}_{0,K}^2 + \sum_{e\in\cE_h^{\mathrm{P}}}\frac{\beta_{p^\mathrm{P}}}{h_e} \norm{ \jump{\kappa/\eta q^\mathrm{P}_h \nn} }_{0,e}^2.\]
In addition, we have
\[|\widetilde{M}_h(\bu_h, p^\mathrm{P}_h, \varphi_h; \bv_h, q^\mathrm{P}_h, \psi_h)|\lesssim 
\VERT(\bu_h, p^\mathrm{P}_h, \varphi_h)\VERT_\ast\VERT(\bv_h, q^\mathrm{P}_h, \psi_h)\VERT_\ast.\]
\end{theorem}
\begin{lemma}\label{formii_lem}
Let $(\tilde{\bu},\tilde{p}^\mathrm{P},\tilde{\varphi})$ be a generic triplet in $\bV_h \times \widetilde{\rQ}^\mathrm{P}_h \times \rZ_h$. 
Then the following bound holds 
\begin{align*}
\VERT(\bu-\bu_h, p-p^\mathrm{P}_h, \varphi-\varphi_h)\VERT_{\ast}\lesssim &\VERT(\bu-\tilde{\bu}, p^{\mathrm{P}}-\tilde{p}^\mathrm{P}, \varphi-\tilde{\varphi})\VERT+\biggl( \sum_{K\in \cT_h}  h_K^2 |\sqrt{2\mu}(\bu-\tilde{\bu})|_{2,K}^2\biggr)^{1/2}\\
&+\biggl( \sum_{K\in \cT_h} \frac{\kappa}{\eta} h_K^2 |\varphi-\tilde{\varphi}|_{2,K}^2\biggr)^{1/2}.
\end{align*}
\end{lemma}
\begin{proof}
Similarly to the proof of Lemma~\ref{lem:aux04}, we obtain 
\begin{align*}
C_2\VERT(\tilde{\bu}-\bu_h, \tilde{p}^{\mathrm{P}}-p^\mathrm{P}_h, \tilde{\varphi}-\varphi_h)\VERT_{\ast}^2&\le \widetilde{M}_h(\tilde{\bu}-\bu_h, \tilde{p}^{\mathrm{P}}-p^\mathrm{P}_h, \tilde{\varphi}-\varphi_h; \bv_h, q^\mathrm{P}_h, \psi_h) \\
&\le \widetilde{M}_h(\tilde{\bu}, \tilde{p}^{\mathrm{P}}, \tilde{\varphi}; \bv_h, q^\mathrm{P}_h, \psi_h)-M_h(\bu_h, p^\mathrm{P}_h, \varphi_h; \bv_h, q^\mathrm{P}_h, \psi_h)\\
&\le \widetilde{M}_h(\tilde{\bu}-\bu, \tilde{p}^{\mathrm{P}}-p^\mathrm{P}, \tilde{\varphi}-\varphi; \bv_h, q^\mathrm{P}_h, \psi_h)\\
&\lesssim \VERT(\bu-\tilde{\bu}, p^{\mathrm{P}}-\tilde{p}^\mathrm{P}, \varphi-\tilde{\varphi})\VERT_{\ast}+\biggl( \sum_{K\in \cT_h}  h_K^2 |\sqrt{2\mu}(\bu-\tilde{\bu})|_{2,K}^2\biggr)^{1/2}\\
&\qquad +\biggl( \sum_{K\in \cT_h} \frac{\kappa}{\eta} h_K^2 |\varphi-\tilde{\varphi}|_{2,K}^2\biggr)^{1/2}
\VERT(\tilde{\bu}-\bu_h, \tilde{p}^{\mathrm{P}}-p^\mathrm{P}_h, \tilde{\varphi}-\varphi_h)\VERT_{\ast}, 
\end{align*}
where we have used Theorem \ref{new12} in combination with triangle inequality. 
\end{proof}
\begin{theorem}\label{th:apriori}
Let $(\bu,p^\mathrm{P},\varphi)$ and $(\bu_h,p^\mathrm{P}_h,\varphi_h)$ be the unique solutions of the continuous and discrete problems (\ref{weakEP}) and (\ref{semidis12}), respectively. 
If $\bu\in \bV\cap\bH^{k+2}(\Omega)$, $p^\mathrm{P}\in\rQ\cap H^{k+2}(\Omega^{\mathrm{P}})$ and $\varphi\in \rZ\cap H^{k+1}(\Omega)$ with $k\ge 0$ then
\begin{align*}
\VERT(\bu-\bu_h, p^\mathrm{P}-p^\mathrm{P}_h, \varphi-\varphi_h)\VERT_{\ast}& \lesssim  h^{k+1}\biggl(|\sqrt{2\mu}\bu|_{k+2,\Omega}^2 +\|\frac{1}{\sqrt{2\mu}}\varphi\|^2_{k+1,\Omega}+\frac{1}{\lambda^{\mathrm{E}}} \|\varphi^{\mathrm{E}}\|^2_{k+1,\Omega^{\mathrm{E}}}+\frac{1}{\lambda^{\mathrm{P}}}\|\varphi^{\mathrm{P}}\|^2_{k+1,\Omega^{\mathrm{P}}}\\
&\qquad\qquad  +(c_0+\frac{\alpha^2}{\lambda^{\mathrm{P}}})\|p^\mathrm{P}\|^2_{k+1,\Omega^{\mathrm{P}}}+ \|\frac{\kappa}{\eta}\nabla p^\mathrm{P}\|^2_{k+2,\Omega^{\mathrm{P}}}\biggr).\end{align*}
\end{theorem}
\begin{proof}
Combining Lemma \ref{formii_lem} with the approximation results of H(div)-conforming spaces \eqref{eq:approx} leads to the stated result.
\end{proof}

\begin{remark}\label{quad11}
If instead of \eqref{eq:discrete-spaces} we employ $RT_{k}\times \mathbb{Q}_{k}\times \mathbb{Q}_{k} \,(k\ge 1)$ on rectangular meshes (where $RT_k$ is the Raviart--Thomas finite element space and $\mathbb{Q}_k$ is the discontinuous finite element space of degree $k$), then the \textit{a priori} error estimates from Theorem~\ref{th:apriori} are modified as follows    
\begin{align*}
\VERT(\bu-\bu_h, p^\mathrm{P}-p^\mathrm{P}_h, \varphi-\varphi_h)\VERT_{\ast}& \lesssim  h^{k}\biggl(|\sqrt{2\mu}\bu|_{k+1,\Omega}^2 +\|\frac{1}{\sqrt{2\mu}}\varphi\|^2_{k,\Omega}+\frac{1}{\lambda^{\mathrm{E}}} \|\varphi^{\mathrm{E}}\|^2_{k,\Omega^{\mathrm{E}}}+\frac{1}{\lambda^{\mathrm{P}}}\|\varphi^{\mathrm{P}}\|^2_{k,\Omega^{\mathrm{P}}}\\
&\qquad\qquad  +(c_0+\frac{\alpha^2}{\lambda^{\mathrm{P}}})\|p^\mathrm{P}\|^2_{k,\Omega^{\mathrm{P}}}+ \|\frac{\kappa}{\eta}\nabla p^\mathrm{P}\|^2_{k+1,\Omega^{\mathrm{P}}}\biggr).\end{align*}
The estimate results from using the approximation properties of the corresponding finite element family on quadrilateral meshes.
\end{remark}

\begin{remark}\label{rem:norms}
Consider the following norm
for $(\bv_h, q^\mathrm{P}_h, \psi_h) \in \bV_h \times \widetilde{\rQ}^\mathrm{P}_h \times {Z}_h$ 
\[\VERT(\bv_h, q^\mathrm{P}_h, \psi_h)\VERT_{\ast\ast}^2:=\norm{\bv_h}_{*,\cT_h}^2 + \|\frac{1}{\sqrt{2\mu}}\psi_h\|^2_{0,\Omega}+\frac{1}{\lambda^{\mathrm{E}}}\|\psi_h\|^2_{0,\Omega^{\mathrm{E}}}+\frac{1}{\lambda^{\mathrm{P}}}\|\psi_h\|^2_{0,\Omega^{\mathrm{P}}}+(c_0+\frac{\alpha^2}{\lambda^{\mathrm{P}}})\|q^\mathrm{P}_h\|^2_{0,\Omega^{\mathrm{P}}}+\|q^\mathrm{P}_h\|^2_{\ast,\Omega^{\mathrm{P}}}.\]
To establish the equivalence between $\VERT\cdot\VERT_{\ast}$ and  $\VERT\cdot\VERT_{\ast\ast}$ in $\bV_h \times \widetilde{\rQ}^\mathrm{P}_h \times {Z}_h$ we need to prove
that
\begin{subequations}
\begin{align}
  \label{eq:norm1}
  \VERT(\bu_h, p^\mathrm{P}_h,\varphi_h)\VERT_{\ast} &\lesssim \VERT(\bu_h, p^\mathrm{P}_h,\varphi_h)\VERT_{\ast\ast},\\
  \label{eq:norm2} 
  \VERT(\bu_h, p^\mathrm{P}_h,\varphi_h)\VERT_{\ast\ast} &\lesssim \VERT(\bu_h, p^\mathrm{P}_h,\varphi_h)\VERT_{\ast}.
\end{align}
\end{subequations}
Using the Cauchy--Schwarz inequality readily implies \eqref{eq:norm1}, whereas  \eqref{eq:norm2}
holds
whenever $\frac{\alpha^2}{\lambda^{\mathrm{P}}}\in[0,9c_0/10)$.

Similarly, for $(\bv_h, q^\mathrm{P}_h, \psi_h) \in \bV_h \times {\rQ}^\mathrm{P}_h \times {Z}_h$  we can establish an equivalence between the norm $\VERT\cdot\VERT$ defined in \eqref{eq:triplenorm} and 
\[\VERT(\bv_h, q^\mathrm{P}_h, \psi_h)\VERT_\sharp^2:=\norm{\bv_h}_{*,\cT_h}^2 + \|\frac{1}{\sqrt{2\mu}}\psi_h\|^2_{0,\Omega}+\frac{1}{\lambda^{\mathrm{E}}}\|\psi_h\|^2_{0,\Omega^{\mathrm{E}}}+\frac{1}{\lambda^{\mathrm{P}}}\|\psi_h\|^2_{0,\Omega^{\mathrm{P}}}+(c_0+\frac{\alpha^2}{\lambda^{\mathrm{P}}})\|q^\mathrm{P}_h\|^2_{0,\Omega^{\mathrm{P}}}+\|\frac{\kappa}{\eta} \nabla q^\mathrm{P}_h\|^2_{0,\Omega^{\mathrm{P}}}.\]
\end{remark}

\section{Residual-based \emph{a posteriori} error analysis}\label{sec:aposte}
In this section we derive robust \emph{a posteriori} estimators for the two families of mixed finite element approximations, and show reliability and efficiency independently of the sensible model parameters.  The error estimates  are obtained in a similar fashion as in, e.g.,  \cite{anaya22}.  Firstly, we discuss \emph{a posteriori} error estimation for formulation \eqref{semidis11}. 

\subsection{Definition of the bulk and edge residuals}\label{def_apost_formI}
First we 
define the local elastic error estimator $\estE$
and the elastic data oscillation $\UpsilonE$ for each $K\in\mathcal{T}_h^{\mathrm{E}}$ as
\begin{align*}
\estE^2&:= \frac{h_K^2}{\mu^{\mathrm{E}}}\|\mathbf{R}_1^\mathrm{E}
\|_{0,K}^2+\sum_{e\in\partial K}\frac{h_e}{\mu^{\mathrm{E}}} \|\mathbf{R}_e^\mathrm{E}
\|_{0,e}^2+\sum_{e\in\partial K}\frac{\beta_{\bu}\mu^{\mathrm{E}}}{h_e} \|[\![\bu^{\mathrm{E}}_h\otimes\nn]\!]_e 
\|_{0,e}^2+{\frac{1}{\frac{1}{\mu^{\mathrm{E}}}+\frac{1}{\lambda^{\mathrm{E}}}}}\|R_2^\mathrm{E}\|_{0,K}^2, \\
\UpsilonE^2&:= \frac{h_K^2}{\mu^{\mathrm{E}}}\|\bb^{\mathrm{E}}-\bb_h^{\mathrm{E}}\|_{0,K}^2,
\end{align*}
where $\ff^{\mathrm{E}}_h\in\mathbf{L}^2(\Omega^{\mathrm{E}})$ is a piecewise polynomial approximation of $\ff^{\mathrm{E}}$. Moreover, 
the element-wise residuals are  
\[
\mathbf{R}_1^\mathrm{E}:= \{\bb^\mathrm{E}_h +\bdiv( 2{\mu^\mathrm{E}} \beps(\bu^\mathrm{E}_h)- \varphi^\mathrm{E}_h \bI)\}_K,\quad 
R_2^\mathrm{E}:= \{\vdiv\bu_h^\mathrm{E} +(\lambda^{\mathrm{E}})^{-1}\varphi^\mathrm{E}_h\}_K,
\]  
and the edge residual in the elastic subdomain is defined as
\[
\mathbf{R}_e^\mathrm{E}:=\begin{cases}
\frac{1}{2}[\![( 2{\mu^\mathrm{E}} \beps(\bu^\mathrm{E}_h)- \varphi^\mathrm{E}_h\bI)\nn]\!]_e & e\in \mathcal{E}(\mathcal{T}_h^{\mathrm{E}})\setminus\Gamma,\\
(( 2{\mu^\mathrm{E}} \beps(\bu^\mathrm{E}_h)- \varphi^\mathrm{E}_h\bI)\nn)_e & e\in \Gamma^\mathrm{E}_{N},\\
\cero & e\in \Gamma^\mathrm{E}_{D}.
\end{cases}
\]

Next, we define the poroelastic local error estimator $\estP$, for each $K\in\mathcal{T}_h^{\mathrm{P}}$,  as
\[
\estP^2:= \frac{h_K^2}{\mu^{\mathrm{P}}}\|\mathbf{R}_1^\mathrm{P}
\|_{0,K}^2+\sum_{e\in\partial K}\frac{h_e}{\mu^{\mathrm{P}}} \|\mathbf{R}_e^\mathrm{P}
\|_{0,e}^2+\sum_{e\in\partial K}\frac{\beta_{\bu}\mu^{\mathrm{P}}}{h_e} \|[\![\bu^{\mathrm{P}}_h\otimes\nn]\!]_e 
\|_{0,e}^2+\rho_d\|R_2^\mathrm{P}\|_{0,K}^2
+\rho_1\|R_3^\mathrm{P}\|_{0,K}^2 +\sum_{e\in\partial K} \rho_2\|{R}_e^\mathrm{P}\|_{0,e}^2,
\]
where the elemental residuals assume the following form 
\begin{align*}
\mathbf{R}_1^\mathrm{P} &:= \{\bb^\mathrm{P}_h +\bdiv( 2{\mu^\mathrm{P}} \beps(\bu^\mathrm{P}_h)- \varphi^\mathrm{P}_h \bI)\}_K,\\ 
R_2^\mathrm{P}&:= \{\vdiv\bu_h^\mathrm{P}+(\lambda^{\mathrm{P}})^{-1}\varphi_h^\mathrm{P}-\alpha(\lambda^{\mathrm{P}})^{-1}p_h^\mathrm{P}\}_K,\\
R_3^\mathrm{P}&:=\{s^{\mathrm{P}}_h-(c_0+\alpha^2(\lambda^\mathrm{P})^{-1}p_h^\mathrm{P}+\alpha(\lambda^{\mathrm{P}})^{-1}\varphi_h^\mathrm{P}+\eta^{-1}\vdiv[\kappa(\nabla p_h^\mathrm{P}-\rho \textbf{g})]\}_K,
\end{align*}
and the edge residuals are defined as
\[
\mathbf{R}_e^\mathrm{P}:=\begin{cases}
\frac{1}{2}[\![(2{\mu^\mathrm{P}} \beps(\bu^\mathrm{P}_h)- \varphi^\mathrm{P}_h \bI)\nn]\!]_e & e\in \mathcal{E}(\mathcal{T}_h^\mathrm{P})\setminus\Gamma\\
((2{\mu^\mathrm{P}} \beps(\bu^\mathrm{P}_h)- \varphi^\mathrm{P}_h \bI)\nn)_e & e\in \Gamma^\mathrm{P}_{N}\\
\cero & e\in \Gamma^\mathrm{P}_{D}
\end{cases}, \quad 
{R}_e^\mathrm{P}:=\begin{cases}
\frac{1}{2}[\![\eta^{-1}\kappa(\nabla p_h-\rho \textbf{g})\cdot\nn]\!]_e & e\in \mathcal{E}(\mathcal{T}_h^\mathrm{P})\setminus\Gamma\\
(\eta^{-1}\kappa(\nabla p_h-\rho \textbf{g})\cdot\nn)_e & e\in \Gamma^\mathrm{P}_{D}\\
0 & e\in \Gamma^\mathrm{P}_{N}
\end{cases},
\]
with the scaling constants taken as 
\[
\rho_1:=\min\{(c_0+\alpha^2(2\mu^{\mathrm{P}}+\lambda^{\mathrm{P}})^{-1})^{-1},h_K^2 \eta \kappa^{-1}\},\quad \rho_2:= \eta\kappa^{-1}h_e, \quad \rho_d:= ((\mu^{\mathrm{P}})^{-1}+(2\mu^{\mathrm{P}}+\lambda^{\mathrm{P}})^{-1})^{-1}.
\]
On the other hand, the  poroelastic oscillation term adopts the following specification 
\[
\UpsilonP^2= h_K^2 (\mu^{\mathrm{P}})^{-1}\|\bb^{\mathrm{P}}-\bb^{\mathrm{P}}_h\|_{0,K}^2+\rho_1 \|s^{\mathrm{P}}-s^{\mathrm{P}}_h\|_{0,K}^2.
\]

Next we recall that $\estE^2$ and  $ \estP^2$ are the elasticity and poroelasticity estimators, respectively. 
Let {us define the interface and total estimators as follows 
\begin{align*}
  \estint^2&:= h_e (\mu^{\mathrm{E}}+\mu^{\mathrm{P}})^{-1}\|\mathbf{R}_{\Sigma}\|_{0,e}^2+h_e \eta \kappa^{-1}\|\widehat{R}_{\Sigma}\|_{0,e}^2 +
\frac{\beta_{\bu}\mu_0}{h_e}\|\jump{\bu_h\otimes\nn}_e\|^2_{0,e}, \quad 
\Xi^2 :=\sum_{K\in\mathcal{T}_h^{\mathrm{E}}} \estE^2+\sum_{K\in\mathcal{T}_h^{\mathrm{P}}} \estP^2+\sum_{e\in \mathcal{E}_h^\Sigma} 
\estint^2, 
\end{align*}
where}
\[
\mathbf{R}_{\Sigma}:= \{( 2{\mu^\mathrm{E}} \beps(\bu^\mathrm{E}_h)- \varphi^\mathrm{E}_h\bI)\nn -( 2{\mu^\mathrm{P}} \beps(\bu^\mathrm{P}_h)- \varphi^\mathrm{E}_h\bI)\nn\}, \quad 
\widehat{R}_{\Sigma}:=\{\kappa\xi^{-1}(\nabla p^{\mathrm{P}}_h-\rho \textbf{g})\cdot \nn\}.
\]
In addition we define the global data oscillations term $\Upsilon$ as
\begin{align*}
\Upsilon^2:=\sum_{K\in\mathcal{T}_h\cap\Omega^{\mathrm{E}}}\UpsilonE^2+\sum_{K\in\mathcal{T}_h\cap\Omega^{\mathrm{P}}}\UpsilonP^2,
\end{align*}
where $\UpsilonE$ and $\UpsilonP$ are the local data oscillations for elasticity and poroelasticity, respectively.
\subsection{Reliability estimates}

First we introduce the following modified bilinear form $\widehat{M}_h(\cdot, \cdot, \cdot; \cdot, \cdot, \cdot)$ as
\begin{align*}
\widehat{M}_h(\bu_h, p^\mathrm{P}_h, \varphi_h; \bv_h, q^\mathrm{P}_h, \psi_h)&:= \tilde{a}_1^h(\bu_h,\bv_h)+{b}_1(\bv_h,\varphi_h)  +{\tilde{a}_2(p_h^\mathrm{P},q_h^\mathrm{P})}+a_2(p_h^\mathrm{P},q_h^\mathrm{P})\\
& \qquad+b_2(q_h^\mathrm{P}, \, \varphi_h)+{b}_1(\bu_h,\psi_h) + b_2(p_h^\mathrm{P},\psi_h)-a_3(\varphi_h,\psi_h).
\end{align*}
Moreover, the following relation holds
\begin{align}
 {a}^h_1(\bu_h, \bv_h)  = \tilde{a}^h_1(\bu_h, \bv_h)  +K_h(\bu_h,\bv_h),
\end{align} 
where the last term on the right-hand side is the consistency contribution and it can be written as 
\begin{align*}
K_h(\bu_h,\bv_h):=- 2 \sum_{e\in\cE_h\cup \Gamma_D^*} \bigl( \langle \mean{{\mu} \beps_h(\bu_h)} , \jump{\bv_h\otimes \nn} \rangle_{0,e} 
+ \langle \mean{{\mu}\beps_h(\bv_h)}, \jump{\bu_h\otimes\nn }\rangle_{0,e}  \bigr).
\end{align*}
\begin{theorem}\label{connew11}
For every $(\bu, p^\mathrm{P}, \varphi) \in \bV \times \rQ^\mathrm{P} \times \rZ$, there exists $(\bv, q^\mathrm{P}, \psi) \in \bV \times \rQ^\mathrm{P} \times \rZ$ with $\VERT(\bv, q^\mathrm{P}, \psi)\VERT\lesssim \VERT(\bu, p^\mathrm{P}, \varphi)\VERT$ such that
\[
\widehat{M}_h(\bu, p^\mathrm{P}, \varphi; \bv, q^\mathrm{P}, \psi)\gtrsim \VERT(\bu, p^\mathrm{P}, \varphi)\VERT^2,
\]
where the triple norm is defined in \eqref{eq:triplenorm}. 
\end{theorem}
Next, the $\bH(\rm{div})$-conforming {displacement} approximation is decomposed uniquely into
$\bu_h=\bu_h^c+\bu_h^r$, 
where $\bu_h^c\in\bV_h^c$ and $\bu_h^r\in(\bV_h^c)^{\bot}$, with $\bu^r_h=\bu_h-\bu^c_h\in\bV_h$.
\begin{lemma}
There holds
\[
\norm{\bu_h^r}_{*,\cT_h}\le
 {C}_r \left(\sum_{e\in\cE_h\setminus\cE_h^\Sigma}\frac{\beta_{\bu}\mu}{h_e} \|[\![\bu_h\otimes\nn]\!]_e \|_{0,e}^2 + \sum_{e\in \cE_h^\Sigma}\frac{\beta_{\bu}\mu_0}{h_e} \|[\![\bu_h\otimes\nn]\!]_e \|_{0,e}^2 \right)^{1/2}.
\]
\end{lemma}
\begin{proof} It follows straightforwardly from the decomposition $\bu_h=\bu_h^c+\bu_h^r$   and 
from the facet residual.\end{proof}
\begin{theorem}[Reliability for the transmission problem]
Let $(\bu,p^{\mathrm{P}},\varphi)$ and $(\bu_h,p^{\mathrm{P}}_h,\varphi_h)$ be the solutions of the weak formulations (\ref{weakEP}) and (\ref{semidis11}), respectively. Then the following reliability bound holds
\[
\VERT(\bu-\bu_h,p^{\mathrm{P}}-p^{\mathrm{P}}_h,\varphi-\varphi_h)\VERT\le C_{\mathrm{rel}} (\Xi+\Upsilon),
\]
where $C_{\mathrm{rel}}>0$ is a constant independent of the mesh size and of the delicate model parameters.
\end{theorem}
\begin{proof} Using triangle inequality, we have
\begin{align}\label{relmeq1}
\VERT(\bu-\bu_h,p^{\mathrm{P}}-p^{\mathrm{P}}_h,\varphi-\varphi_h)\VERT\le \VERT(\bu-\bu_h^c,p^{\mathrm{P}}-p^{\mathrm{P}}_h,\varphi-\varphi_h)\VERT+\VERT(\bu_h^r,0,0)\VERT.
\end{align}
Since $(\bu-\bu_h^c,p^{\mathrm{P}}-p_h^{\mathrm{P}}, \varphi-\varphi_h)\in \bV\times\rQ^{\mathrm{P}}\times \rZ$, then from the stability result in Theorem~\ref{connew11}, we have  
\begin{align}\label{relmeq2}
C_2\VERT(\bu-\bu_h^c,p^{\mathrm{P}}-p_h^{\mathrm{P}}, \varphi-\varphi_h)\VERT^2\le \widehat{M}_h((\bu-\bu_h^c,p^{\mathrm{P}}-p_h^{\mathrm{P}}, \varphi-\varphi_h); (\bv,q^{\mathrm{P}},\psi)),
\end{align}
with $\VERT(\bv,q^{\mathrm{P}},\psi)\VERT\le C_1 \VERT(\bu-\bu_h^c,p^{\mathrm{P}}-p_h^{\mathrm{P}}, \varphi-\varphi_h)\VERT$. 
Moreover, we have
\begin{align*}
 & \widehat{M}_h((\bu-\bu_h^c,p^{\mathrm{P}}-p_h^{\mathrm{P}}, \varphi-\varphi_h); (\bv,q^{\mathrm{P}},\psi))\\ 
 &\qquad = \widehat{M}_h((\bu-\bu_h,p^{\mathrm{P}}-p_h^{\mathrm{P}}, \varphi-\varphi_h); (\bv,q^{\mathrm{P}},\psi))+ \widehat{M}_h((\bu_h^r,0, 0; (\bv,q^{\mathrm{P}},\psi))\\
 &\qquad \le  \widehat{M}_h((\bu-\bu_h,p^{\mathrm{P}}-p_h^{\mathrm{P}}, \varphi-\varphi_h); (\bv,q^{\mathrm{P}},\psi))+\norm{\bu_h^r}_{*,\cT_h} \VERT(\bv,q^{\mathrm{P}},\psi)\VERT\\
 &\qquad \le F(\bv)+G(q^\mathrm{P})- \widetilde{M}_h((\bu_h,p_h^{\mathrm{P}},\varphi_h); (\bv,q^{\mathrm{P}},\psi))+\norm{\bu_h^r}_{*,\cT_h} \VERT(\bv,q^{\mathrm{P}},\psi)\VERT.
\end{align*}
Then, we can employ the following identity
\begin{align*}
F(\textbf{I}_h\bv)+G(I_hq^\mathrm{P})- \widehat{M}_h((\bu-\bu_h,p^{\mathrm{P}}-p_h^{\mathrm{P}}, \varphi-\varphi_h); \textbf{I}_h\bv,I_hq^{\mathrm{P}},0))-K_h(\bu_h,\textbf{I}_h\bv)=0,
\end{align*}
to arrive at 
\begin{align*}
 &\widehat{M}_h((\bu-\bu_h^c,p^{\mathrm{P}}-p_h^{\mathrm{P}}, \varphi-\varphi_h); (\bv,q^{\mathrm{P}},\psi))\\
 &\quad \le F(\bv-\textbf{I}_h\bv)+G(q^\mathrm{P}-I_h q^\mathrm{P})- \widetilde{M}_h((\bu_h,p_h^{\mathrm{P}},\varphi_h); (\bv-\textbf{I}_h\bv,q^{\mathrm{P}}-I_h q^{\mathrm{P}},\psi))
  +K_h(\bu_h,\textbf{I}_h\bv)+\norm{\bu_h^r}_{*,\cT_h} \VERT(\bv,q^{\mathrm{P}},\psi)\VERT.
\end{align*}
Using integration by parts and the Cauchy--Schwarz inequality, gives
\begin{align}\label{relmeq3}
 \widehat{M}_h&((\bu-\bu_h^c,p^{\mathrm{P}}-p_h^{\mathrm{P}}, \varphi-\varphi_h); (\bv,q^{\mathrm{P}},\psi))\le C(\Xi+\Upsilon)\VERT(\bv,q^{\mathrm{P}},\psi)\VERT.
 \end{align}
 And then it suffices to combine (\ref{relmeq1}),  (\ref{relmeq2}) and  (\ref{relmeq3}) to prove the desired assertion.
\end{proof}
\subsection{Efficiency estimates}
For this step we follow the classical inverse estimate approach from \cite{verfurth13}, which necessitates an extension operator plus volume and edge bubble functions. For sake of clarity, each of the residual terms constituting the  \emph{a posteriori} error indicator is treated separately.

\subsubsection{Efficiency estimates for elastic error estimator} 
For each $K\in\mathcal{T}_h$, we can define the interior polynomial bubble function $b_K$ which is positive in the interior of $K$ and zero on $\partial K$.
From  \cite{verfurth13}, we can write the following results:
\begin{equation}\label{ele1}
\|v\|_{0,K}\lesssim \|b_K^{1/2}v\|_{0,K}, \qquad 
\|b_Kv\|_{0,K}\lesssim \|v\|_{0,K}, \qquad 
\|\nabla(b_Kv)\|_{0,K}\lesssim h_K^{-1}\|v\|_{0,K},
\end{equation}
where $v$ is a scalar-valued polynomial function defined on $K$.

\begin{lemma}\label{lemE1}
The following bound for the  local vectorial  estimator in the  bulk elasticity sub-domain holds true:
\[
h_K(\mu^{\mathrm{E}})^{-1/2}\|\mathbf{R}_1^{\mathrm{E}}\|_{0,K}\lesssim (\mu^{\mathrm{E}})^{-1/2}h_K\|\bb^{\mathrm{E}}-\bb_h^{\mathrm{E}}\|_{0,K}+(\mu^{\mathrm{E}})^{-1}\|\varphi^{\mathrm{E}}-\varphi_h^{\mathrm{E}}\|_{0,K}+(2\mu^{\mathrm{E}})^{1/2}\|\bu^{\mathrm{E}}-\bu_h^{\mathrm{E}}\|_{0,K} , \qquad K\in \cT_h^{\mathrm{E}}.
\]
\end{lemma}
\begin{proof}
For each $K\in \mathcal{T}_h$, we can define 
$\bzeta|_K=(\mu^{\mathrm{E}})^{-1}h_K^2\mathbf{R}_1^{\mathrm{E}}b_K$. 
Using (\ref{ele1}) gives
\[
h_K^2(\mu^{\mathrm{E}})^{-1}\|\mathbf{R}_1^{\mathrm{E}}\|_{0,K}^2\lesssim \int_K \mathbf{R}_1^{\mathrm{E}}\cdot ((\mu^{\mathrm{E}})^{-1}h_K^2\mathbf{R}_1^{\mathrm{E}}b_K)=\int_{K}\mathbf{R}_1^{\mathrm{E}}\cdot \bzeta.
\]
Note that $\{\bb^\mathrm{E} +\bdiv( 2{\mu^\mathrm{E}} \beps(\bu^\mathrm{E})- \varphi^\mathrm{E}_h \bI)\}=\cero|_K$. 
We subtract this from the last term and then integrate using $\bzeta|_{\partial K}=\cero$
\[
h_K^2(\mu^{\mathrm{E}})^{-1}\|\mathbf{R}_1^{\mathrm{E}}\|_{0,K}^2\lesssim \int_{K} (\bb_h^{\mathrm{E}}-\bb^{\mathrm{E}})\cdot \bzeta -2{\mu^{\mathrm{E}}}\int_K \beps(\bu^{\mathrm{E}}-\bu_h^{\mathrm{E}}) \cdot \nabla \bzeta -\int_K (\varphi^{\mathrm{E}}-\varphi_h^{\mathrm{E}})\nabla\cdot \bzeta.
\]
Then we proceed to apply Cauchy--Schwarz inequality, which readily implies
{\begin{align*}
h_K^2(\mu^{\mathrm{E}})^{-1}\|\mathbf{R}_1^{\mathrm{E}}\|_{0,K}^2\lesssim &((\mu^{\mathrm{E}})^{-1/2}h_K\|\bb^{\mathrm{E}}-\bb_h^{\mathrm{E}}\|_{0,K}+(2\mu^{\mathrm{E}})^{1/2}h_K\|\nabla(\bu^{\mathrm{E}}-\bu_h^{\mathrm{E}})\|_{0,K}+(\mu^{\mathrm{E}})^{-1/2}h_K\|\varphi^{\mathrm{E}}-\varphi_h^{\mathrm{E}}\|_{0,K}))\\
&((\mu^{\mathrm{E}})^{1/2}\|\bnabla \bzeta\|_{0,K}+(\mu^{\mathrm{E}})^{1/2}h_K^{-1}\|\bzeta\|_{0,K}).
\end{align*}}
We can complete the proof using the following result
\begin{align*}
(\mu^{\mathrm{E}})^{1/2}\|\bnabla \bzeta\|_{0,K}+(\mu^{\mathrm{E}})^{1/2}h_K^{-1}\|\bzeta\|_{0,K}&\lesssim
(\mu^{\mathrm{E}})^{1/2}(\|\bnabla \bzeta\|_{0,K}+h_K^{-1}\|\bzeta\|_{0,K})\\
&\lesssim (\mu^{\mathrm{E}})^{1/2} h_K^{-1}\|\bzeta\|_{0,K} \\
&= h_K(\mu^{\mathrm{E}})^{-1/2}\|\mathbf{R}_1^{\mathrm{E}}\|_{0,K}.
\end{align*}
\end{proof}
\begin{lemma}\label{lemE21}
The following bound holds true for the  local scalar estimator in the  bulk elasticity sub-domain:
\begin{align*}
\sqrt{\frac{1}{\frac{1}{\mu^{\mathrm{E}}}+\frac{1}{\lambda^{\mathrm{E}}}}}\|\mathbf{R}_2^{\mathrm{E}}\|_{0,K}\lesssim (\lambda^{\mathrm{E}})^{-1/2} \|\varphi-\varphi_h\|_{0,K}+\sqrt{2\mu^{\mathrm{E}}}\|\bnabla(\bu-\bu_h)\|_{0,K}, \qquad K\in \cT_h^{\mathrm{E}}.
\end{align*}
\end{lemma}
\begin{proof}
Consider $K\in \cT_h^{\mathrm{E}}$. Using the  relation $\vdiv\bu_h^\mathrm{E} +(\lambda^{\mathrm{E}})^{-1}\varphi^\mathrm{E}=0|_K$, it holds
\begin{align*}
\sqrt{\frac{1}{\frac{1}{\mu^{\mathrm{E}}}+\frac{1}{\lambda^{\mathrm{E}}}}}\|\mathbf{R}_2^{\mathrm{E}}\|_{0,K}&=\sqrt{\frac{1}{\frac{1}{\mu^{\mathrm{E}}}+\frac{1}{\lambda^{\mathrm{E}}}}}\|\vdiv\bu_h^\mathrm{E} +(\lambda^{\mathrm{E}})^{-1}\varphi^\mathrm{E}_h\|_{0,K} \\
&=\sqrt{\frac{1}{\frac{1}{\mu^{\mathrm{E}}}+\frac{1}{\lambda^{\mathrm{E}}}}}\|\vdiv\bu_h^\mathrm{E} -\vdiv\bu^\mathrm{E}+(\lambda^{\mathrm{E}})^{-1}(\varphi^\mathrm{E}_h-\varphi^\mathrm{E})\|_{0,K}\\
&\lesssim (\lambda^{\mathrm{E}})^{-1/2} \|\varphi-\varphi_h\|_{0,K}+\sqrt{2\mu^{\mathrm{E}}}\|\bnabla(\bu-\bu_h)\|_{0,K}.
\end{align*}
\end{proof}

Let $b_e$ be the edge polynomial bubble function on $e$ which is an interior edge (or interior facet in 3D) shared by two elements $K$ and $K'$.
Moreover, $b_e$  is positive in the interior of the
patch $P_e$ formed by $K\cup K'$, and is zero on the boundary of the patch.
Then, we can conclude the following estimates from \cite{verfurth13}:
\begin{gather}\label{edgee1}
\|q\|_{0,e}\lesssim \|b_e^{1/2}q\|_{0,e},\quad 
\|b_eq\|_{0,K}\lesssim h_e^{1/2}\|q\|_{0,e}, \quad 
\|\nabla(b_eq)\|_{0,K}\lesssim h_e^{-1/2}\|q\|_{0,e}\qquad \forall K\in P_e,
\end{gather}
where $q$ is the scalar-valued polynomial function which is defined on the edge $e$.

\begin{lemma}\label{lemE4}
With regards to the edge contribution to the local estimator on the elastic sub-domain, we have that
\begin{align*}
&(\sum_{e\in\partial K}h_e(\mu^{\mathrm{E}})^{-1} \|\mathbf{R}_e^{\mathrm{E}}\|_{0,e}^2)^{1/2}\\
&\qquad \lesssim
\sum_{K\in P_e}((\mu^{\mathrm{E}})^{-1/2}h_K\|\bb^{\mathrm{E}}-\bb_h^{\mathrm{E}}\|_{0,K} +(\mu^{\mathrm{E}})^{-1/2}\|\varphi^{\mathrm{E}}-\varphi_h^{\mathrm{E}}\|_{0,K}+(2\mu^{\mathrm{E}})^{1/2}\|\bu^{\mathrm{E}}-\bu_h^{\mathrm{E}}\|_{0,K}). 
\end{align*}
\end{lemma}
\begin{proof}
For each $e\in \mathcal{E}_h$,  we  introduce
$\bzeta_e=(\mu^{\mathrm{E}})^{-1}h_e\mathbf{R}_e^{\mathrm{E}}b_e$. 
Then, the estimates (\ref{edgee1}) gives 
\[
h_e(\mu^{\mathrm{E}})^{-1}\|\mathbf{R}_e^{\mathrm{E}}\|_{0,e}^2\lesssim \int_e \mathbf{R}_e^{\mathrm{E}}\cdot ((\mu^{\mathrm{E}})^{-1}h_e\mathbf{R}_e^{\mathrm{E}}b_e)=\int_{e}\mathbf{R}_e^{\mathrm{E}}\cdot \bzeta_e.
\]
Using the relation $[\![(2\mu^{\mathrm{E}}\beps( \bu^{\mathrm{E}})-\varphi^{\mathrm{E}}\bI)\nn]\!]_e=\cero$ implies
 \begin{align*}
 \int_e [\![(2\mu^{\mathrm{E}}(\beps( \bu_h^{\mathrm{E}})-\beps( \bu^{\mathrm{E}}))-(\varphi_h^{\mathrm{E}}-\varphi^{\mathrm{E}})\bI)\cdot\nn]\!]_e \cdot \bzeta_e&=\sum_{K\in P_e}\int_K(2\mu^{\mathrm{E}}(\beps( \bu_h^{\mathrm{E}})-\beps( \bu^{\mathrm{E}}))+\nabla (\varphi_h^{\mathrm{E}}-\varphi^{\mathrm{E}}))\cdot \bzeta_e\\
 &\quad+\sum_{K\in P_e}\int_K(2{\mu^{\mathrm{E}}}(\bu_h^{\mathrm{E}}-\bu^{\mathrm{E}})\cdot  \nabla \bzeta_e+ (\varphi_h^{\mathrm{E}}-\varphi^{\mathrm{E}}) \nabla\cdot  \bzeta_e),
 \end{align*}
 where integration by parts have been used element-wise.
Recalling that $\{\bb^\mathrm{E} +\bdiv( 2{\mu^\mathrm{E}} \beps(\bu^\mathrm{E})- \varphi^\mathrm{E} \bI)\}=\cero|_K$, we have 
\begin{align*}
\frac{h_e}{\mu^{\mathrm{E}}}\|\mathbf{R}_e^{\mathrm{E}}\|_{0,e}^2&\lesssim \sum_{K\in P_e}\int_{K} \left((\bb_h^{\mathrm{E}}-\bb^{\mathrm{E}})\cdot \bzeta_e
+2{\mu^{\mathrm{E}}}\int_K (\bu^{\mathrm{E}}-\bu_h^{\mathrm{E}}) \cdot \nabla \bzeta_e +\int_K (\varphi_h^{\mathrm{E}}-\varphi_h^{\mathrm{E}})\nabla\cdot \bzeta \right)+ \sum_{K\in P_e}\int_{K} \mathbf{R}_1^{\mathrm{E}} \cdot \bzeta_e.
\end{align*}
From the Cauchy--Schwarz inequality we can conclude that 
\begin{align*}
h_e(\mu^{\mathrm{E}})^{-1}\|\mathbf{R}_e^{\mathrm{E}}\|_{0,e}^2\lesssim 
&\sum_{K\in P_e}((\mu^{\mathrm{E}})^{-1/2}h_K\|\bb^{\mathrm{E}}-\bb_h^{\mathrm{E}}\|_{0,K}+(\mu^{\mathrm{E}})^{-1/2}\|\varphi^{\mathrm{E}}-\varphi_h^{\mathrm{E}}\|_{0,K}+\|\bnabla\bu^{\mathrm{E}}-\bnabla\bu_h^{\mathrm{E}}\|_{0,K})\times\\
&\quad((\mu^{\mathrm{E}})^{1/2}\|\bnabla \bzeta_e\|_{0,K}+(\mu^{\mathrm{E}})^{1/2}h_K^{-1}\|\bzeta_e\|_{0,K}).
\end{align*}
And the rest of the desired estimate follows from the following bound
\[
(\mu^{\mathrm{E}})^{1/2}\|\bnabla \bzeta_e\|_{0,K}+(\mu^{\mathrm{E}})^{1/2}h_K^{-1}\|\bzeta_e\|_{0,K}
\lesssim (\mu^{\mathrm{E}})^{1/2} h_K^{-1}\|\bzeta_e\|_{0,K} 
= h_e^{1/2}(\mu^{\mathrm{E}})^{-1/2}\|\mathbf{R}_e^{\mathrm{E}}\|_{0,e}.
\]
\end{proof}

\begin{lemma}\label{lemE5}
The elastic local bulk a posteriori estimator satisfies 
\begin{align*}
\left(\sum_{K\in\mathcal{T}_h^{\mathrm{E}}} \estE^2\right)^{1/2}\lesssim \sum_{K\in\mathcal{T}_h^{\mathrm{E}}}\left(\frac{h_K}{\sqrt{\mu^{\mathrm{E}}}} \|\bb^{\mathrm{E}}-\bb_h^{\mathrm{E}}\|_{0,K} +\frac{1}{\sqrt{\mu^{\mathrm{E}}}}\|\varphi^{\mathrm{E}}-\varphi_h^{\mathrm{E}}\|_{0,K}+\frac{1}{\sqrt{2\mu^{\mathrm{E}}}}\|\bnabla(\bu^{\mathrm{E}}-\bu_h^{\mathrm{E}})\|_{0,K}\right).
\end{align*}
\end{lemma}
\begin{proof}
Combining Lemmas \ref{lemE1}-\ref{lemE4} implies the stated result.
\end{proof}

\subsubsection{Efficiency estimates for poroelastic error estimator} 
\begin{lemma}\label{lemP1}
The first vectorial contribution to the local bulk poroelastic error estimator satisfies the following bound 
\[
h_K(\mu^{\mathrm{P}})^{-1/2}\|\mathbf{R}_1^{\mathrm{P}}\|_{0,K}\lesssim (\mu^{\mathrm{P}})^{-1/2}h_K\|\bb^{\mathrm{P}}-\bb_h^{\mathrm{P}}\|_{0,K}+(\mu^{\mathrm{P}})^{-1}\|\varphi^{\mathrm{P}}-\varphi_h^{\mathrm{P}}\|_{0,K}+\|\bu^{\mathrm{P}}-\bu_h^{\mathrm{P}}\|_{0,K}.
\]
\end{lemma}
\begin{proof}
It proceeds similarly to Lemma \ref{lemE1}.
\end{proof}
\begin{lemma}\label{lemE22}
The second scalar contribution to the local bulk poroelastic error estimator satisfies the following bound 
\begin{align*}
\rho_d^{1/2}\|{R}_2^{\mathrm{P}}\|_{0,K}\lesssim (\lambda)^{-1/2} \|(\varphi^\mathrm{P}_h-\varphi^\mathrm{P}-\alpha(p_h^{\mathrm{P}}-p^{\mathrm{P}}))\|_{0,K}+\sqrt{2\mu^{\mathrm{P}}}\|\bnabla(\bu-\bu_h)\|_{0,K}.
\end{align*}
\end{lemma}
\begin{proof}
The constitutive relation $\vdiv\bu_h^\mathrm{P} +(\lambda^{\mathrm{P}})^{-1}\varphi^\mathrm{P}-\alpha(\lambda^{\mathrm{P}})^{-1}p^\mathrm{P}=0|_K$ implies that 
\begin{align*}
\rho_d^{1/2}\|{R}_2^{\mathrm{P}}\|_{0,K}&=\rho_d^{1/2}\|\vdiv\bu_h^\mathrm{P} +(\lambda^{\mathrm{P}})^{-1}\varphi^\mathrm{P}_h-\alpha(\lambda^{\mathrm{P}})^{-1}p^\mathrm{P}_h\|_{0,K} \\
&=\rho_d^{1/2}\|\vdiv\bu_h^\mathrm{P} -\vdiv\bu^\mathrm{P}+(\lambda^{\mathrm{P}})^{-1}(\varphi^\mathrm{P}_h-\varphi^\mathrm{P}-\alpha(p_h^{\mathrm{P}}-p^{\mathrm{P}}))\|_{0,K}\\
&\lesssim (\lambda^{\mathrm{P}})^{-1/2} \|\varphi^\mathrm{P}_h-\varphi^\mathrm{P}-\alpha(p_h^{\mathrm{P}}-p^{\mathrm{P}})\|_{0,K}+\sqrt{2\mu^{\mathrm{P}}}\|\bnabla(\bu-\bu_h)\|_{0,K}.
\end{align*}
\end{proof}
\begin{lemma}\label{lemP41}
The third scalar contribution to the local bulk poroelastic error estimator satisfies the following bound 
\begin{align*}
(\rho_1)^{1/2}\|{R}_3^{\mathrm{P}}\|_{0,K}&\lesssim (\rho_1)^{1/2}\|s^{\mathrm{P}}-s^{\mathrm{P}}_h\|_{0,K}+(c_0)^{1/2}\|p^\mathrm{P}-p_h^{\mathrm{P}}\|_{0,K}+(\kappa/\xi)^{1/2}\|\nabla(p^{\mathrm{P}}-p_h^{\mathrm{P}})\|_{0,K}\\
&\quad +(1/\lambda^{\mathrm{P}})^{-1/2}\|\varphi-\varphi_h^{\mathrm{P}}+\alpha(p^\mathrm{P}-p_h^{\mathrm{P}})\|_{0,K}.
\end{align*}
\end{lemma}
\begin{proof}
For each $K\in \mathcal{T}_h$, we define 
$\omega|_K=\rho_1{R}_3b_K$. 
Then, 
invoking (\ref{ele1}), we conclude that
\[
\rho_1\|{R}_3^{\mathrm{P}}\|_{0,K}^2\lesssim \int_K {R}_3^{\mathrm{P}} (\rho_1{R}_3^{\mathrm{P}}b_K) =\int_{K}{R}_3^{\mathrm{P}} \omega.
\]
Using the relation $s-[c_0+\alpha^2(\lambda^{\mathrm{P}})^{-1}]p^{\mathrm{P}}+\alpha(\lambda^{\mathrm{P}})^{-1}\varphi^{\mathrm{P}}+\xi^{-1}\vdiv[\kappa(\nabla p^{\mathrm{P}}-\rho \bg)]_K=0$
in the last term and then integrating  with $\omega|_{\partial K}=0$, we can assert that 
\begin{align*}
(\rho_1)^{-1}\|{R}_3^{\mathrm{P}}\|_{0,K}^2&\lesssim \int_{K} (s^{\mathrm{P}}_h-s^{\mathrm{P}}) \omega +(c_0)^{-1}\int_K (p^{\mathrm{P}}-p_h^{\mathrm{P}}) \omega +\xi^{-1}\int_K \kappa\nabla(p^{\mathrm{P}}-p^{\mathrm{P}}_h)\cdot\nabla \omega \\
&\quad+\alpha(\lambda)^{-1}\int_K (\varphi^{\mathrm{P}}-\varphi_h^{\mathrm{P}}+\alpha(p^{\mathrm{P}}-p_h^{\mathrm{P}}) ) \omega.
\end{align*}
Then, Cauchy--Schwarz inequality gives
\begin{align*}
\rho_1\|{R}_3^{\mathrm{P}}\|_{0,K}^2\lesssim &((\rho_1)^{1/2}\|s^{\mathrm{P}}-s^{\mathrm{P}}_h\|_{0,K}+[c_0]^{1/2}\|p^{\mathrm{P}}-p^{\mathrm{P}}_h\|_{0,K}+\xi^{-1/2}\|\kappa^{1/2}\nabla(p^{\mathrm{P}}-p^{\mathrm{P}}_h)\|_{0,K}\\
&\quad +(\lambda^{\mathrm{P}})^{-1/2}\|(\varphi^{\mathrm{P}}-\varphi_h^{\mathrm{P}}+\alpha(p^{\mathrm{P}}-p_h^{\mathrm{P}}) )\|_{0,K})((\kappa/\xi)^{1/2}\|\nabla \omega\|_{0,K}+(\rho_1)^{-1/2})\|\omega\|_{0,K}.
\end{align*}
And the proof follows after noting that 
\begin{align*}
(\frac{\kappa}{\xi})^{1/2}\|\nabla \omega\|_{0,K}+(\rho_1)^{-1/2}\|\omega\|_{0,K}&\lesssim
(\frac{\kappa}{\xi})^{1/2}h_K^{-1}\| \omega\|_{0,K}+\rho_1^{-1/2}\|\omega\|_{0,K})\lesssim (\rho_1)^{-1/2} \|\omega\|_{0,K}
= (\rho_1)^{1/2}\|{R}_3^{\mathrm{P}}\|_{0,K}.
\end{align*}
\end{proof}
\begin{lemma}\label{lemP42}
The  edge contribution to the local poroelastic error estimator satisfies the following bound 
\begin{align*}
&(\sum_{e\in\partial K}h_e(\mu^{\mathrm{P}})^{-1} \|\mathbf{R}_e^{\mathrm{P}}\|_{0,e}^2)^{1/2}\\
&\qquad \lesssim
\sum_{K\in P_e}((\mu^{\mathrm{P}})^{-1/2}h_K\|\bb^{\mathrm{P}}-\bb_h^{\mathrm{P}}\|_{0,K} +(\mu^{\mathrm{P}})^{-1/2}\|\varphi^{\mathrm{P}}-\varphi_h^{\mathrm{P}}\|_{0,K}+(2\mu^{\mathrm{P}})^{1/2}\|\bnabla(\bu^{\mathrm{P}}-\bu_h^{\mathrm{P}})\|_{0,K}). 
\end{align*}
\end{lemma}
\begin{proof}
The proof is conducted similarly to that of Lemma \ref{lemE4}.
\end{proof}
\begin{lemma}\label{lemP5}
There holds:
\begin{align*}
\left(\sum_{K\in\mathcal{T}_h^{\mathrm{P}}} \estP^2\right)^{1/2}\lesssim &\sum_{K\in\mathcal{T}_h^{\mathrm{P}}}\left(((\mu^{\mathrm{P}})^{-1/2}h_K\|\bb^{\mathrm{P}}-\bb_h^{\mathrm{P}}\|_{0,K} +(\mu^{\mathrm{P}})^{-1/2}\|\varphi^{\mathrm{P}}-\varphi_h^{\mathrm{P}}\|_{0,K}+(2\mu^{\mathrm{P}})^{1/2}\|\bnabla(\bu^{\mathrm{P}}-\bu_h^{\mathrm{P}})\|_{0,K})\right.\\
 &\quad(\rho_1)^{1/2}\|s^{\mathrm{P}}-s^{\mathrm{P}}_h\|_{0,K}+(c_0)^{1/2}\|p^\mathrm{P}-p_h^{\mathrm{P}}\|_{0,K}+(\kappa/\xi)^{1/2}\|\nabla(p^{\mathrm{P}}-p_h^{\mathrm{P}})\|_{0,K}\\
&\left.\quad +(1/\lambda^{\mathrm{P}})^{-1/2}\|\varphi-\varphi_h^{\mathrm{P}}+\alpha(p^\mathrm{P}-p_h^{\mathrm{P}})\|_{0,K}\right).
\end{align*}
\end{lemma}
\begin{proof}
The results follows after combining Lemmas \ref{lemP1}-\ref{lemP42}.
\end{proof}
\subsubsection{Efficiency estimates for interface estimator}

\begin{lemma}\label{eleEP3}
There holds:
\begin{align*}
  &\biggl(\sum_{e\in\Sigma}h_e(\mu^{\mathrm{E}}+\mu^{\mathrm{P}})^{-1} \|\mathbf{R}_{\Sigma}\|_{0,e}^2\biggr)^{1/2} \\
 &\quad \lesssim \sum_{e\in\Sigma}\Big(\sum_{K\in P_e\cap\Omega^{\mathrm{E}}}((2\mu^{\mathrm{E}})^{-1/2}h_K\|\bb^{\mathrm{E}}-\bb_h^{\mathrm{E}}\|_{0,K}+(2\mu^{\mathrm{E}})^{-1/2}\|\varphi^{\mathrm{E}}-\varphi_h^{\mathrm{E}}\|_{0,K}+(2\mu^{\mathrm{E}})^{1/2}\|\bnabla_h(\textbf{u}^{\mathrm{E}}-\textbf{u}_h^{\mathrm{E}})\|_{0,K})\\
&\qquad\quad +\sum_{K\in P_e\cap\Omega^{\mathrm{P}}}((2\mu^{\mathrm{P}})^{-1/2}h_K\|\bb^{\mathrm{P}}-\bb_h^{\mathrm{P}}\|_{0,K}+(2\mu^{\mathrm{P}})^{-1/2}\|\varphi^{\mathrm{P}}-\varphi_h^{\mathrm{P}}\|_{0,K}+(2\mu^{\mathrm{P}})^{1/2}\|\bnabla_h(\bu^{\mathrm{P}}-\textbf{u}_h^{\mathrm{P}})\|_{0,K})\Big).
\end{align*}
\end{lemma}
\begin{proof}
For each $e\in \mathcal{E}_h^\Sigma$,  
$\bzeta_e$ is defined locally as 
$ \bzeta_e=(\mu^{\mathrm{E}}+\mu^{\mathrm{P}})^{-1}h_e\mathbf{R}_{\Sigma}b_e.$ 
Using (\ref{edgee1}) gives 
\[
h_e(\mu^{\mathrm{E}}+\mu^{\mathrm{P}})^{-1}\|\mathbf{R}_{\Sigma}\|_{0,e}^2\lesssim \int_e \mathbf{R}_{\Sigma}\cdot ((\mu^{\mathrm{E}}+\mu^{\mathrm{P}})^{-1}h_e\mathbf{R}_{\Sigma}b_e) =\int_{e}\mathbf{R}_{\Sigma}\cdot \bzeta_e.
\]
Integration by parts implies
\begin{align*}
\int_e &(2{\mu^{\mathrm{E}}}( \beps(\bu^\mathrm{E}_h)- \beps(\bu^\mathrm{E})) -(\varphi_{h}^{\mathrm{E}}-\varphi^{\mathrm{E}})\bI)\nn  \cdot \bzeta_e -(2{\mu^{\mathrm{P}}}( \beps(\bu^\mathrm{P}_h)- \beps(\bu^\mathrm{P})) -(\varphi_{h}^{\mathrm{P}}-\varphi^{\mathrm{P}})\bI)\nn \cdot \bzeta_e\\
&=\sum_{K\in P_e\cap\Omega^{\mathrm{E}}}\int_K(\bdiv(2{\mu^{\mathrm{E}}}( \beps(\bu^\mathrm{E}_h)- \beps(\bu^\mathrm{E}))) +\nabla(\varphi_{h}^{\mathrm{E}}-\varphi^{\mathrm{E}}))\cdot \bzeta_e\\
&\qquad -\sum_{K\in P_e\cap\Omega^{\mathrm{E}}}\int_K(2{\mu^{\mathrm{E}}}( \beps(\bu^\mathrm{E}_h)- \beps(\bu^\mathrm{E})) -(\varphi_{h}^{\mathrm{E}}-\varphi^{\mathrm{E}})\bI):\bnabla  \bzeta_e \\
&\qquad - \sum_{K\in P_e\cap\Omega^{\mathrm{P}}}\int_K(\bdiv(2{\mu^{\mathrm{P}}}( \beps(\bu^\mathrm{P}_h)- \beps(\bu^\mathrm{P}))) +\nabla(\varphi_{h}^{\mathrm{P}}-\varphi^{\mathrm{P}}))\cdot \bzeta_e\\
&\qquad -\sum_{K\in P_e\cap\Omega^{\mathrm{P}}}\int_K(2{\mu^{\mathrm{P}}}( \beps(\bu^\mathrm{P}_h)- \beps(\bu^\mathrm{P})) +(\varphi_{h}^{\mathrm{P}}-\varphi^{\mathrm{P}})\bI):\bnabla  \bzeta_e.
\end{align*}
Note that $\bb^{\mathrm{P}}+\bdiv(2{\mu^{\mathrm{P}}}\beps(\bu^\mathrm{P}) -p^{\mathrm{P}}\bI)=\cero|_K$ and $\bb^{\mathrm{E}}+\bdiv(2{\mu^{\mathrm{E}}}\beps(\bu^\mathrm{E}) -p^{\mathrm{E}}\bI)=\cero|_K$. Then, we can assert that 
\begin{align*}
\frac{h_e}{\mu^{\mathrm{E}}+\mu^{\mathrm{P}}}\|\mathbf{R}_{\Sigma}\|_{0,e}^2 
&\lesssim \sum_{K\in P_e\cap\Omega^{\mathrm{E}}}\int_{K} \left((\bb_h^{\mathrm{E}}-\bb^{\mathrm{E}})\cdot \bzeta_e -\int_K 2{\mu^{\mathrm{E}}}( \beps(\bu^\mathrm{E}_h)- \beps(\bu^\mathrm{E})) : \bnabla \bzeta_e +\int_K (p_h^{\mathrm{E}}-p^{\mathrm{E}})\nabla\cdot \bzeta \right)\\
&\quad+ \sum_{K\in P_e\cap\Omega^{\mathrm{E}}}\int_{K} \mathbf{R}_1^{\mathrm{E}} \cdot \bzeta_e+ \sum_{K\in P_e\cap\Omega^{\mathrm{P}}}\int_{K} \mathbf{R}_1^{\mathrm{P}} \cdot \bzeta_e \\
&\quad+ \sum_{K\in P_e\cap\Omega^{\mathrm{P}}}\int_{K} \left((\bb_h^{\mathrm{P}}-\bb^{\mathrm{P}})\cdot \bzeta_e -\int_K 2{\mu^{\mathrm{P}}}( \beps(\bu^\mathrm{P}_h)- \beps(\bu^\mathrm{P})) : \bnabla \bzeta_e +\int_K (p_h^{\mathrm{P}}-p^{\mathrm{P}})\nabla\cdot \bzeta \right).
\end{align*}
 Applying Cauchy--Schwarz inequality gives 
\begin{align*}
\frac{h_e}{\mu^{\mathrm{E}}+\mu^{\mathrm{P}}}\|\mathbf{R}_e\|_{0,e}^2 & \lesssim
\sum_{K\in P_e\cap\Omega^{\mathrm{E}}}((2\mu^{\mathrm{E}})^{-1/2}h_K\|\bb^{\mathrm{E}}-\bb_h^{\mathrm{E}}\|_{0,K}+(2\mu^{\mathrm{E}})^{-1/2}\|\varphi^{\mathrm{E}}-\varphi_h^{\mathrm{E}}\|_{0,K}+(2\mu)^{1/2}\|\beps(\bu^\mathrm{E}_h)- \beps(\bu^\mathrm{E})\|_{0,K})\times\\
&\qquad\qquad\quad((2\mu^{\mathrm{E}})^{1/2}\|\bnabla \bzeta\|_{0,K}+(2\mu^{\mathrm{E}})^{1/2}h_K^{-1}\|\bzeta\|_{0,K})\\
&\quad+\sum_{K\in P_e\cap\Omega^{\mathrm{P}}}((2\mu^{\mathrm{P}})^{-1/2}h_K\|\bb^{\mathrm{P}}-\bb_h^{\mathrm{P}}\|_{0,K}+(2\mu^{\mathrm{P}})^{-1/2}\|\varphi^{\mathrm{P}}-\varphi_h^{\mathrm{P}}\|_{0,K}+(2\mu)^{1/2}\|\beps(\bu^\mathrm{P}_h)- \beps(\bu^\mathrm{P})\|_{0,K})\times\\
&\qquad\qquad\quad((\mu^{\mathrm{P}})^{1/2}\|\bnabla \bzeta\|_{0,K}+(\mu^{\mathrm{P}})^{1/2}h_K^{-1}\|\bzeta\|_{0,K}).
\end{align*}
And as a consequence of the bounds 
\begin{align*}
(2\mu^{\mathrm{E}})^{1/2}\|\bnabla \bzeta\|_{0,K}+(2\mu^{\mathrm{E}})^{1/2}h_K^{-1}\|\bzeta\|_{0,K}
&\lesssim (2\mu^{\mathrm{E}})^{1/2} h_K^{-1}\|\bzeta\|_{0,K}\lesssim h_e^{1/2}(\mu^{\mathrm{E}}+\mu^{\mathrm{P}})^{-1/2}\|\mathbf{R}_e\|_{0,e},\\
(2\mu^{\mathrm{P}})^{1/2}\|\bnabla \bzeta\|_{0,K}+(2\mu^{\mathrm{P}})^{1/2}h_K^{-1}\|\bzeta\|_{0,K}
&\lesssim (2\mu^{\mathrm{P}})^{1/2} h_K^{-1}\|\bzeta\|_{0,K}\lesssim h_e^{1/2}(\mu^{\mathrm{E}}+\mu^{\mathrm{P}})^{-1/2}\|\mathbf{R}_e\|_{0,e},
\end{align*}
the desired estimates hold true.
\end{proof}
\begin{theorem}[Efficiency]
Let $(\bu,p^{\mathrm{P}},\varphi)$ and $(\bu_h,p^{\mathrm{P}}_h,\varphi_h)$ be the solutions of the weak formulations \eqref{weakEP} and \eqref{semidis11}, respectively. Then the following efficiency bound holds
\[
\Xi\le C_{\mathrm{eff}} (\VERT(\bu-\bu_h,p^{\mathrm{P}}-p^{\mathrm{P}}_h,\varphi-\varphi_h)\VERT+\Upsilon),
\]
where $C_{\mathrm{eff}}>0$ is a   constant independent of $h$ and of the sensible model parameters.
\end{theorem}
\begin{proof}
The bound results from combining Lemmas \ref{lemE5}, \ref{lemP5} and \ref{eleEP3}.
\end{proof}

\begin{remark}
To introduce  \emph{a posteriori} error estimation for formulation \eqref{semidis12}, we modify the proposed estimator for formulation \eqref{semidis11}.
Specifically, we add one extra jump term for discontinuous fluid pressure so that the  modified \emph{a posteriori}  error estimator is as follows:
\begin{align}\label{apost_formII}
\Xi^2 &:=\sum_{K\in\mathcal{T}_h^{\mathrm{E}}} \estE^2+\sum_{K\in\mathcal{T}_h^{\mathrm{P}}} \widetilde{\estP}^2+\sum_{e\in \mathcal{E}_h^\Sigma} 
\estint^2,
\end{align}
with 
\begin{align*}
\widetilde{\estP}^2={\estP}^2+\sum_{e\in\partial K}\frac{\beta_{p^{\mathrm{P}}}\kappa}{h_e\eta} \|[\![p^{\mathrm{P}}_h\nn]\!]_e\|_{0,e}^2,
\end{align*}
where $\estE$, $\estP$ and $\estint$ are defined in Section \ref{def_apost_formI}. The proposed \emph{a posteriori} estimator (\ref{apost_formII})
is also reliable, efficient and robust. The idea of proofs of reliability and efficiency  is similar to the \emph{a posteriori}  estimation associated with formulation \eqref{semidis11}.
\end{remark}

\section{Robust block preconditioning}\label{sec:robustness}

Building upon the analysis results in Sections~\ref{sec:FE} and~\ref{sec:apriori}, our goal now is to construct norm-equivalent block diagonal preconditioners for the discrete systems \eqref{semidis11} and \eqref{semidis12} that are robust with respect to  (e.g., high interface contrast in the) physical parameters and mesh size $h$. 

To this end, we begin by writing system \eqref{weakEP} in the following operator form $ \cM \vec{x} = \cJ$, with $\vec{x} = (\bu,p^{\mathrm{P}},\varphi)$, $\cJ = (F,G,0)$,  and 
\begin{equation}\label{eq:operator-form}
  \cM =   \begin{pmatrix} \cA_1 &0 & \cB'_1 \\
0  & -\cC_1 & \cB'_2\\
\cB_1& \cB_2  & -\cC_2 
\end{pmatrix}.
\end{equation}
The block operators in $\cM$ are induced by the respective bilinear forms as: 
\begin{align*}
\cA_1 &: \bV \to \bV', \quad \langle \cA_1(\bu),\bv\rangle := a_1(\bu,\bv) = \int_\Omega 2\mu\beps(\bu):\beps(\bv), \\
\cB_1 &: \bV\to \rZ', \quad \langle \cB_1(\bv),\psi\rangle := b_1(\bv,\psi) = -\int_\Omega \psi\vdiv\bv,\\
\cB_2 &: \rQ^{\mathrm{P}} \to \rZ', \quad \langle \cB_2(p^{\mathrm{P}}),\psi\rangle := b_2(p^{\mathrm{P}},\psi) = \int_{\Omega^{\mathrm{P}}}\frac{\alpha}{\lambda}p^{\mathrm{P}}\psi, \\
\cC_1 &: \rQ^{\mathrm{P}} \to {\rQ^{\mathrm{P}}}', \quad \langle\cC_1(p^{\mathrm{P}}),q^{\mathrm{P}}\rangle := \tilde{a}_2(p^{\mathrm{P}},q^{\mathrm{P}}) + a_2(p^{\mathrm{P}},q^{\mathrm{P}})=\int_{\Omega^{\mathrm{P}}}\left( \left(c_0+\frac{\alpha^2}{\lambda}\right)p^{\mathrm{P}}q^{\mathrm{P}} + \frac{\kappa}{\eta}\nabla p^{\mathrm{P}} \nabla q^{\mathrm{P}} \right),\\
\cC_2 &: \rZ \to \rZ', \quad \langle\cC_2(\varphi),\psi\rangle=a_3(\phi,\psi) := \frac{1}{\lambda}\int_\Omega\varphi\psi.
\end{align*}
Similarly, the discrete systems \eqref{semidis11} and \eqref{semidis12} can be cast, respectively, in the following matrix block-form:
\begin{equation}\label{eq:matrix-form11}
\underbrace{  \begin{pmatrix} \cA_{1h} &0 & \cB'_1 \\
	0  & -\cC_1 & \cB'_2\\
	\cB_1& \cB_2  & -\cC_2 
	\end{pmatrix}}_{=:\cM_h}\begin{pmatrix} \bu_h \\p_h^\mathrm{P}\\ \varphi_h \end{pmatrix} = \begin{pmatrix} F \\ G\\ 0 \end{pmatrix}, \quad \text{and}\quad 
\underbrace{  \begin{pmatrix} \cA_{1h} &0 & \cB'_1 \\
	0  & -\cC_{1h} & \cB'_2\\
	\cB_1& \cB_2  & -\cC_2 
	\end{pmatrix}}_{=:\widetilde{\cM}_h}\begin{pmatrix} \bu_h \\p_h^\mathrm{P}\\ \varphi_h \end{pmatrix} = \begin{pmatrix} F \\ G\\ 0 \end{pmatrix},
\end{equation}
where $\cM_h$ and $\widetilde{\cM}_h$ are  induced by the multilinear forms $M_h$ and $\widetilde{M}_h$, respectively, and 
\begin{align*}
\cA_{1h}&:\bV_h \to \bV_h', \quad \langle \cA_{1h}(\bu_h),\bv_h\rangle := a_1^h(\bu_h,\bv_h),\\
\cC_{1h}&: \rQ_h^{\mathrm{P}} \to {\rQ_h^{\mathrm{P}}}', \quad \langle\cC_1(p^{\mathrm{P}}_h),q^{\mathrm{P}}_h\rangle := \tilde{a}_2(p^{\mathrm{P}}_h,q^{\mathrm{P}}_h) + a^h_2(p^{\mathrm{P}}_h,q^{\mathrm{P}}_h).
\end{align*}

Next, and following \cite{lee17} (see also \cite[Remark 5]{hong19} and \cite{hong16b}), a preconditioner for the linear systems in \eqref{eq:matrix-form11} can be constructed from the discrete version of the continuous Riesz map block-diagonal operator. This latter continuous map is defined as follows:
\begin{align}\label{eq:riesz} 
& \nonumber\cP: \bV \times \rQ^{\mathrm{P}}\times\rZ  \to (\bV \times \rQ^{\mathrm{P}}\times\rZ)',\\	
&\quad \cP:=\begin{pmatrix} [\cA_1]^{-1} & 0&0\\0 &\![\cC_1]^{-1} & 0 \\ 0 &0& [\cC_2']^{-1}\end{pmatrix} =
\begin{pmatrix}
2\mu\bdiv\beps & 0 & 0 \\
0 &\! \!\! \! \! \!\!\!\!\left(C_0 + \frac{\alpha^2}{\lambda} \right)\bI-\vdiv(\frac{\kappa}{\eta}\nabla) & 0 \\
0 & 0 &\! \!\!\! \!\!  \left(\frac{1}{\lambda} + \frac{1}{2\mu}\right)\bI
\end{pmatrix}^{-1}.
\end{align}
Note that, when comparing the previous expression with the main block diagonal of \eqref{eq:operator-form}, $\cC_2$ is replaced by $\cC_2'$, which contains the additional term $\frac{1}{2\mu}$. Furthermore, we define the discrete weighted space $\bX_{h,\epsilon} := \bV_h\times \rQ_h^{\mathrm{P}} \times \rZ_h$ which contains all triplets $(\bu_h,p_h^{\mathrm{P}},\varphi_h)$ that are bounded in the discrete weighted norm  $\VERT\cdot\VERT$, and, similarly, $\bX_{h,\epsilon,\ast} := \bV_h\times \widetilde{\rQ}_h^{\mathrm{P}} \times \rZ_h$, with the norm $\VERT\cdot\VERT_{\ast}$, in the discontinuous fluid pressure case. Here  the subindex $\epsilon$ represents all weighting parameters $(\mu,c_0,\alpha,\lambda,\kappa,\eta)$.

Note that the discrete solution operator $\cM_h$ (and also $\widetilde{\cM}_h$ for the case of discontinuous pressure) is self-adjoint and indefinite on  $\bX_{h,\epsilon}$ (resp. on $\bX_{h,\epsilon,\ast}$). The stability of this operator in the triple norm has been proven in Theorem~\ref{new11}, which implies that it is a uniform isomorphism (see also Theorem~\ref{new12} for the case of discontinuous pressure and using the norm $\VERT\cdot\VERT_{\ast}$). Based on the discrete solution operators and the Riesz map \eqref{eq:riesz}, we have the following form for the discrete preconditioners:
\begin{equation}\label{eq:Ph}
\cP_h = \begin{pmatrix} [\widehat{\cA}_{1h}]^{-1} & 0&0\\0 & [\cC_1]^{-1} & 0 \\ 
0 & 0 & [\cC_2']^{-1}\end{pmatrix}, \quad \widetilde{\cP}_h = \begin{pmatrix} [\widehat{\cA}_{1h}]^{-1} & 0&0\\0 & [\cC_{1h}]^{-1} & 0 \\ 
0 & 0 & [\cC_2']^{-1}\end{pmatrix}, 
\end{equation}
where $\widehat{\cA}_{1h}$ is defined as follows:
\[\widehat{\cA}_{1h}:\bV_h \to \bV_h', \quad \langle \widehat{\cA}_{1h}(\bu_h),\bv_h\rangle := \sum_{K\in\cT_h}2\mu (\beps(\bu_h),\beps(\bv_h))_K + \sum_{e\in\cE_h \cup \Gamma_D^*}\frac{2\mu \beta_{\bu}}{h_e}\langle\jump{\bu_h\otimes\nn},\jump{\bv_h\otimes\nn}\rangle_e,\]
i.e., the operator defining the $\|\cdot\|_{*,\cT_h}$ norm.

The preconditioners $\cP_h,\widetilde{\cP}_h$ represent self-adjoint and positive-definite operators. They can be thus used to accelerate the convergence of the MINRES iterative solver for the solution of the symmetric indefinite linear systems \eqref{semidis11} and \eqref{semidis12}, respectively. A suitable norm equivalence result (see, e.g., Remark \ref{rem:norms}) implies that the matrices in \eqref{eq:Ph} are indeed canonical block-diagonal preconditioners which are robust with respect to all model parameters. In addition, and owing to the discrete inf-sup conditions stated in Section~\ref{sec:FE}, we can state that the discrete preconditioners are also robust with respect to the discretisation parameter $h$. This will be confirmed with suitable numerical experiments in Section~\ref{sec:exp_robustness}.

\section{Representative computational results} \label{sec:results}
In this section we perform some numerical examples that confirm the validity of the derived error estimates. All of these tests were conducted using the open source finite element libraries {\tt FEniCS} \cite{alnaes} (using {\tt multiphenics} \cite{multiphenics} for the handling of subdomains and incorporation of restricted finite element spaces) and {\tt Gridap} \cite{badia20} (version 0.17.12). 
{\tt Gridap} is a free and open-source finite element framework written in Julia that combines a high-level user interface to define the weak form in a syntax close to the mathematical notation and a computational backend based on the Julia JIT-compiler, which generates high-performant code tailored for the user input \cite{Verdugo2022}. 
We have taken advantage of the extensible and modular nature of \texttt{Gridap} to implement the new methods in this paper. The high-level API in {\tt Gridap} provides all the ingredients required for the definition of the forms in (\ref{semidis11}) and (\ref{semidis12}), e.g., integration on facets in $\mathcal{E}_h$, $\Gamma_D^*$, and $\mathcal{E}_{h}^{\Sigma}$ and jump/mean operators in (\ref{eq:mean-jump}). It also provides those tools required to build and apply the block diagonal preconditioners in (\ref{eq:Ph}). In the future, we plan to leverage the implementation with \texttt{GridapDistributed} \cite{Badia2022} so that we can tackle large-scale application problems on petascale distributed-memory computers.
For the sake of reproducibility, the Julia software used in this paper is available publicly/openly at \cite{hdiv_biot_elasticity_paper_software}. 
Except for the preconditioning tests collected in Section~\ref{sec:exp_robustness}, all sparse linear systems are solved with UMFPACK for the Julia codes or 
with the  Multifrontal Massively Parallel Solver MUMPS \cite{mumps} otherwise.


\begin{table}[t]
	\centering
\begin{tabular}{||cr|gg|cc|cc|cc|c||}
	\toprule 
$k$ &{\tt DoF}& ${\tt e}(\bu,p^\mathrm{P},\varphi)$ & {\tt rate} &  ${\tt e}_{\ast}(\bu)$ &  {\tt rate} & ${\tt e}(p^\mathrm{P})$ &  {\tt rate} 
& ${\tt e}(\varphi)$ & {\tt rate}  & ${\tt eff}(\Xi)$\\
\midrule
\multirow{6}{*}{0}&
    81 & 5.84e+03 & * & 2.61e+02 & * & 1.22e+00 & * & 8.24e+03 & * & 1.35e-01\\
 &  296 & 2.63e+03 & 1.15 & 7.84e+01 & 1.73 & 4.30e-01 & 1.51 & 3.71e+03 & 1.15 & 1.22e-01\\
 & 1134 & 1.28e+03 & 1.04 & 2.51e+01 & 1.64 & 1.91e-01 & 1.17 & 1.81e+03 & 1.04 & 1.19e-01\\
 & 4442 & 6.36e+02 & 1.01 & 6.94e+00 & 1.85 & 9.25e-02 & 1.05 & 8.98e+02 & 1.01 & 1.18e-01\\
 &17586 & 3.17e+02 & 1.00 & 2.90e+00 & 1.26 & 4.59e-02 & 1.01 & 4.48e+02 & 1.00 & 1.18e-01\\
 &69986 & 1.59e+02 & 1.00 & 1.83e+00 & 1.03 & 2.29e-02 & 1.00 & 2.24e+02 & 1.00 & 1.18e-01\\
 \midrule
	\multirow{6}{*}{1} & 
   204 & 1.36e+03 & * & 6.05e+01 & * & 3.36e-01 & * & 1.92e+03 & * & 6.60e-02 \\
&   774 & 3.45e+02 & 1.98 & 3.15e+01 & 0.94 & 8.67e-02 & 1.95 & 4.86e+02 & 1.98 & 6.92e-02\\
&  3018 & 8.85e+01 & 1.96 & 8.51e+00 & 1.89 & 2.15e-02 & 2.01 & 1.24e+02 & 1.96 & 6.96e-02\\
& 11922 & 2.22e+01 & 1.99 & 2.04e+00 & 2.06 & 5.38e-03 & 2.00 & 3.13e+01 & 1.99 & 6.97e-02\\
& 47394 & 5.57e+00 & 2.00 & 4.95e-01 & 2.04 & 1.35e-03 & 2.00 & 7.84e+00 & 2.00 & 6.98e-02\\
&188994 & 1.40e+00 & 2.00 & 1.23e-01 & 2.01 & 3.37e-04 & 2.00 & 1.96e+00 & 2.00 & 6.99e-02\\
\midrule
	\multirow{6}{*}{2} & 
   383 & 2.75e+02 & * & 5.29e+01 & * & 3.58e-02 & * & 3.81e+02 & * & 5.85e-02 \\
&  1476 & 3.25e+01 & 3.08 & 6.35e+00 & 3.06 & 3.09e-03 & 3.54 & 4.50e+01 & 3.08 & 4.32e-02\\
&  5798 & 3.73e+00 & 3.12 & 4.37e-01 & 3.86 & 3.74e-04 & 3.05 & 5.22e+00 & 3.10 & 3.96e-02\\
& 22986 & 4.52e-01 & 3.04 & 2.93e-02 & 3.90 & 4.65e-05 & 3.01 & 6.38e-01 & 3.03 & 3.85e-02\\
& 91538 & 5.61e-02 & 3.01 & 1.92e-03 & 3.93 & 5.81e-06 & 3.00 & 7.92e-02 & 3.01 & 3.82e-02\\
&365346 & 7.02e-03 & 3.00 & 2.20e-04 & 3.13 & 7.26e-07 & 3.00 & 9.90e-03 & 3.00 & 3.83e-02\\
\bottomrule
	\end{tabular}
	\caption{Example 1. Error history and effectivity indexes for polynomial degrees $k=0,1,2$,   going up to $T = 1$. Discretisation with continuous Biot fluid pressure.}\label{table:h}
\end{table}

\subsection{Verification of convergence to smooth solutions}\label{sec:verification} We manufacture a closed-form displacement and fluid pressure
\[ \bu = \begin{pmatrix} \sin(\pi [x+y])\\\cos(\pi[x^2+y^2]) \end{pmatrix}, \quad p^\mathrm{P} = \sin(\pi x+ y) \sin(\pi y),\]
which, together with $\varphi^\mathrm{P} = \alpha p^\mathrm{P} - \lambda^{\mathrm{P}} \vdiv\bu$, $ \varphi^\mathrm{E} = - \lambda^{\mathrm{E}} \vdiv\bu$, constitute the 
solutions to \eqref{eq:coupled}. For this test we consider the unit square domain $\Omega = (0,1)^2$ divided into $\Omega^{\mathrm{E}} = (0,1)\times(0.5,1)$ and $\Omega^{\mathrm{P}} = (0,1)\times(0,0.5)$ and separated by the interface $\Sigma=(0,1)\times\{0.5\}$. The boundaries are taken as $\Gamma^{\mathrm{E}}_D = \partial\Omega^{\mathrm{E}}\setminus \Sigma$ and $\Gamma^{\mathrm{P}}_D = \partial\Omega \setminus \Gamma^{\mathrm{E}}_D$, which implies that a real Lagrange multiplier is required constraining the mean value of the global pressure to coincide with the exact value. The parameter values are taken as follows 
\begin{gather*} 
\alpha= 1, \quad  \mu^\mathrm{P}=10, \quad \lambda^\mathrm{P} =2\cdot 10^4 , \quad  \mu^\mathrm{E}=20, \quad \lambda^\mathrm{E} = 10^4, \\
c_0 = 1, \quad \kappa = 1, \quad \eta =1, \quad \gamma = 1, \quad \Delta t = 1, \quad T = 1, \quad \beta_{\bu} =\beta_{p^\mathrm{P}} = 2.5\cdot 10^{2k+1}. 
\end{gather*}
We note that the stress on the interface $\Sigma$ is not continuous. As a result, we must add the following term:
\[
\sum_{e \in \mathcal{E}_h^\Sigma}
\left( \mean{\boldsymbol{v}} , \jump{  (2 \mu \boldsymbol{\varepsilon}(\boldsymbol{u}) -  \varphi \bI)\nn} \right)_{0,e},
\] 
to the right-hand side of (\ref{semidis11}) and (\ref{semidis12}) evaluated at the exact solution. We must also include additional terms for non-homogeneous Neumann and Dirichlet boundary conditions.

For the discretisation using continuous fluid pressure approximation, errors between exact and approximate solutions are computed using the norms
\begin{gather*}
  {\tt e}(\bu,p^\mathrm{P},\varphi) := \VERT(\bu-\bu_h,p^\mathrm{P}-p^\mathrm{P}_h,\varphi-\varphi_h)\VERT, \quad  {\tt e}_{\ast}(\bu) := \|\bu-\bu_h\|_{*,\cT_h}, \\
 e(p^\mathrm{P}):= (c_0+ \alpha^2/\lambda^{\mathrm{P}})\|p^\mathrm{P}-p^\mathrm{P}_h\|_{0,\Omega^{\mathrm{P}}} + \frac{\kappa}{\eta} \|\nabla_h(p^\mathrm{P}-p^\mathrm{P}_h)\|_{0,\Omega^{\mathrm{P}}}, \quad {\tt e}(\varphi) := \frac{1}{\mu^{\mathrm{E}}}\|\varphi-\varphi_h\|_{0,\Omega^{\mathrm{E}}} + \frac{1}{\mu^{\mathrm{P}}}\|\varphi-\varphi_h\|_{0,\Omega^{\mathrm{P}}}, 
\end{gather*}
while for discontinuous pressure approximations the following norms are modified 
\[  {\tt e}_{\ast}(\bu,p^\mathrm{P},\varphi) := \VERT(\bu-\bu_h,p^\mathrm{P}-p^\mathrm{P}_h,\varphi-\varphi_h)\VERT_{\ast}, \quad  
e_{\ast}(p^\mathrm{P}):=  \|p^\mathrm{P}-p^\mathrm{P}_h\|_{*,\Omega^\mathrm{P}}.
\]
The experimental rates of convergence   are computed as 
\[{\tt r}  =\log({\tt e}_{(\cdot)}/\tilde{{\tt e}}_{(\cdot)})[\log(h/\tilde{h})]^{-1}, \]
where ${\tt e},\tilde{{\tt e}}$ denote errors generated on two consecutive  meshes of sizes $h$ and~$\tilde{h}$, respectively. 
Such an error history is displayed   in Tables~\ref{table:h}-\ref{table:h-dg}. In these cases we note that uniform mesh refinement is sufficient to obtain optimal convergence rates of $\mathcal{O}(h^{k+1})$ in the corresponding broken energy norm (we also tabulate the individual errors in their natural norms). These results are consistent with the theoretical error estimates derived in Theorem \ref{th:apriori}. 

\begin{table}[t]
	\centering
\begin{tabular}{||cr|gg|cc|cc|cc|c||}
	\toprule 
$k$ &{\tt DoF}& ${\tt e}_{\ast}(\bu,p^\mathrm{P},\varphi)$ & {\tt rate} &  ${\tt e}_{\ast}(\bu)$ &  {\tt rate} & ${\tt e}_{\ast}(p^\mathrm{P})$ &  {\tt rate} 
& ${\tt e}(\varphi)$ & {\tt rate}  & ${\tt eff}(\Xi)$\\
\midrule
\multirow{5}{*}{0}&     97 & 5.84e+03 & * & 2.61e+02 & * & 6.94e-01 & * & 8.24e+03 & * & 1.25e-01\\
&   369 & 2.63e+03 & 1.15 & 7.84e+01 & 1.73 & 3.57e-01 & 0.96 & 3.71e+03 & 1.15 & 1.22e-01\\
&  1441 & 1.28e+03 & 1.04 & 2.51e+01 & 1.64 & 1.81e-01 & 0.98 & 1.81e+03 & 1.04 & 1.19e-01\\
&  5697 & 6.36e+02 & 1.01 & 6.94e+00 & 1.85 & 9.14e-02 & 0.99 & 8.98e+02 & 1.01 & 1.18e-01\\
& 22657 & 3.17e+02 & 1.00 & 2.90e+00 & 1.26 & 4.66e-02 & 0.97 & 4.48e+02 & 1.00 & 1.18e-01\\
 \midrule
	\multirow{5}{*}{1} &    229 & 1.36e+03 & * & 6.06e+01 & * & 1.48e-01 & * & 1.92e+03 & * & 6.60e-02\\
&   889 & 3.45e+02 & 1.98 & 3.15e+01 & 0.95 & 4.19e-02 & 1.83 & 4.85e+02 & 1.98 & 6.93e-02\\
&  3505 & 8.85e+01 & 1.96 & 8.51e+00 & 1.89 & 1.09e-02 & 1.94 & 1.24e+02 & 1.96 & 6.97e-02\\
& 13921 & 2.22e+01 & 1.99 & 2.03e+00 & 2.06 & 2.79e-03 & 1.97 & 3.13e+01 & 1.99 & 6.98e-02\\
& 55489 & 5.57e+00 & 2.00 & 4.95e-01 & 2.04 & 7.43e-04 & 1.91 & 7.84e+00 & 2.00 & 6.98e-02\\
	\midrule
	\multirow{5}{*}{2} &    417 & 2.73e+02 & * & 5.29e+01 & * & 2.46e-02 & * & 3.78e+02 & * & 4.06e-02 \\
&  1633 & 3.24e+01 & 3.07 & 6.35e+00 & 3.06 & 2.94e-03 & 3.07 & 4.49e+01 & 3.07 & 4.02e-02\\
&  6465 & 3.72e+00 & 3.12 & 4.37e-01 & 3.86 & 3.69e-04 & 2.99 & 5.22e+00 & 3.10 & 3.95e-02\\
& 25729 & 4.52e-01 & 3.04 & 2.94e-02 & 3.89 & 4.84e-05 & 2.93 & 6.37e-01 & 3.03 & 3.85e-02\\
& 102657 & 5.63e-02 & 3.02 & 7.81e-03 & 3.10 & 5.91e-06 & 2.97 & 8.16e-02 & 3.01 & 3.84e-02\\
\bottomrule
	\end{tabular}
	\caption{Example 1. Error history and effectivity indexes for polynomial degrees $k=0,1,2$,   going up to $T = 1$. Discretisation with discontinuous Biot fluid pressure.}\label{table:h-dg}
\end{table}

The robustness of the \emph{a posteriori} error estimators is quantified in terms of the effectivity index of the  indicator   
\[ \texttt{eff}(\Xi) =  (\texttt{e}_{\ast}(\bu)^2 + \texttt{e}(p^\mathrm{P})^2+ \texttt{e}(\varphi)^2)^{1/2}/ \Xi,\]
(or $\texttt{eff}(\Xi) =  (\texttt{e}_{\ast}(\bu)^2 + \texttt{e}_{\ast}(p^\mathrm{P})^2+ \texttt{e}(\varphi)^2)^{1/2}/ \Xi$ in the case of discontinuous fluid pressures) and $\texttt{eff}$ is expected to remain constant independently of the number of degrees of freedom associated with each mesh refinement. In both tables  the effectivity index is asymptotically constant for all polynomial degrees. This fact confirms the efficiency and reliability of the estimator. 
Similar results are also obtained even when the Poisson ratio in each subdomain is close to 0.5.


\begin{figure}[t!]
  \begin{center}
    \includegraphics[width=0.325\textwidth]{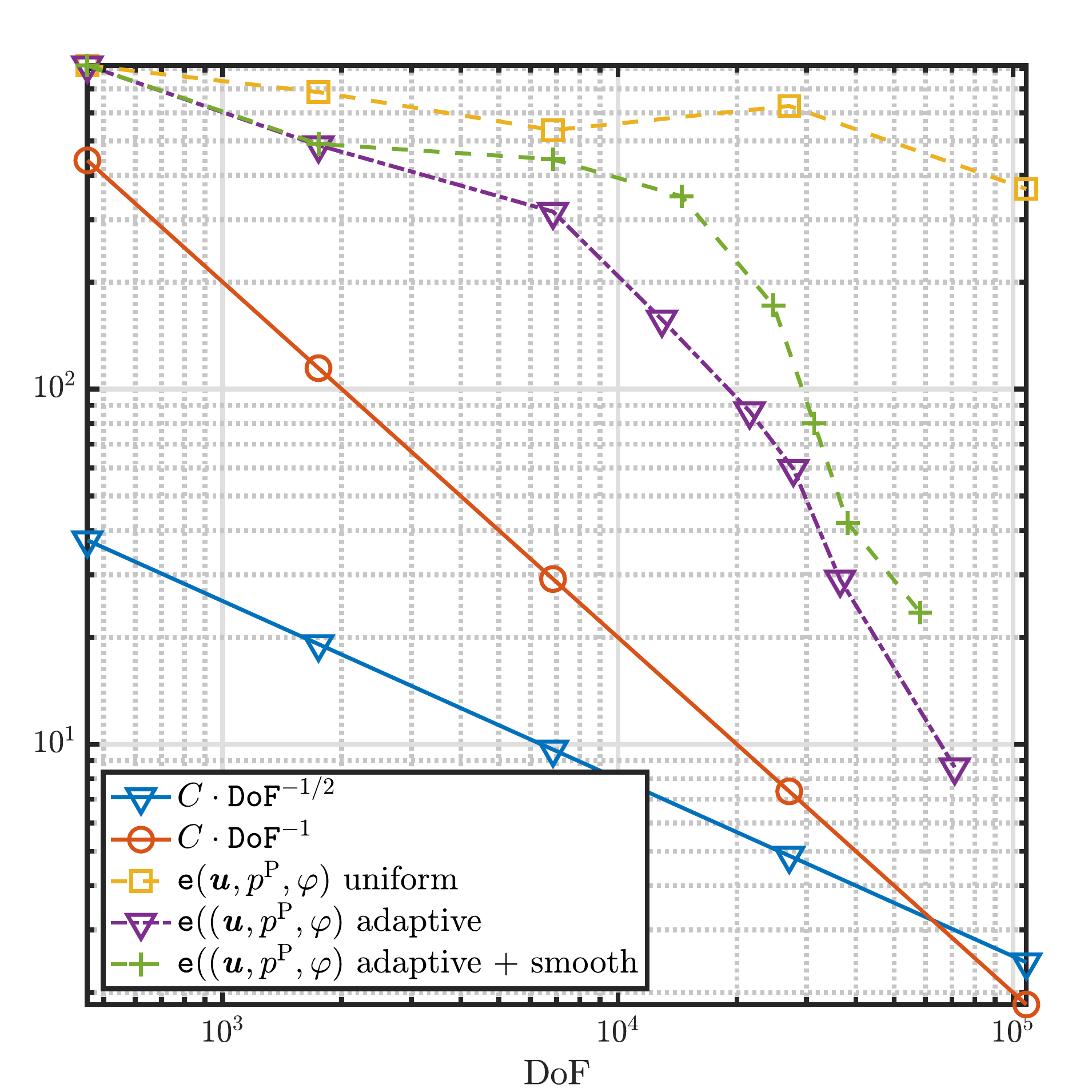}
    \includegraphics[width=0.325\textwidth]{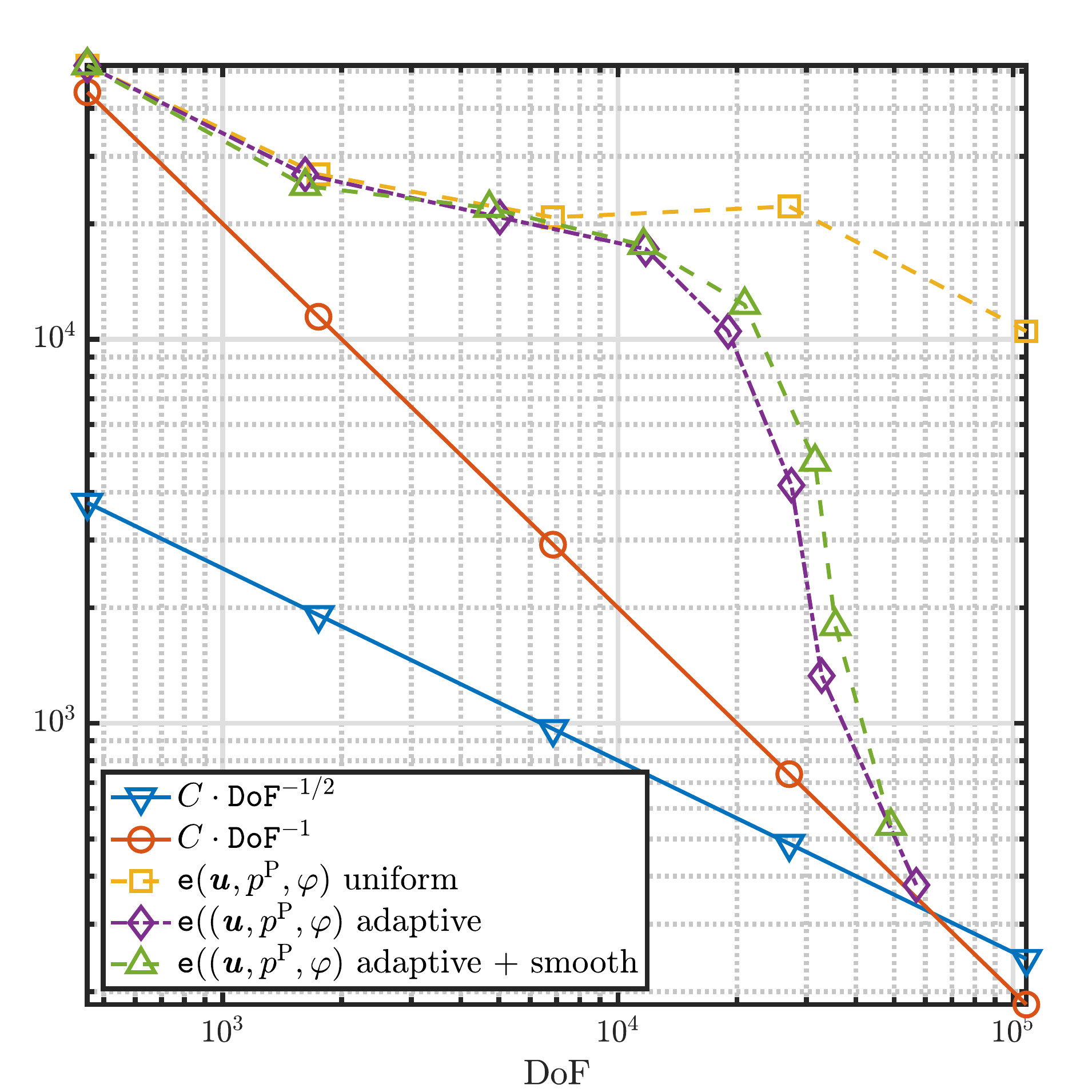}
    \includegraphics[width=0.325\textwidth]{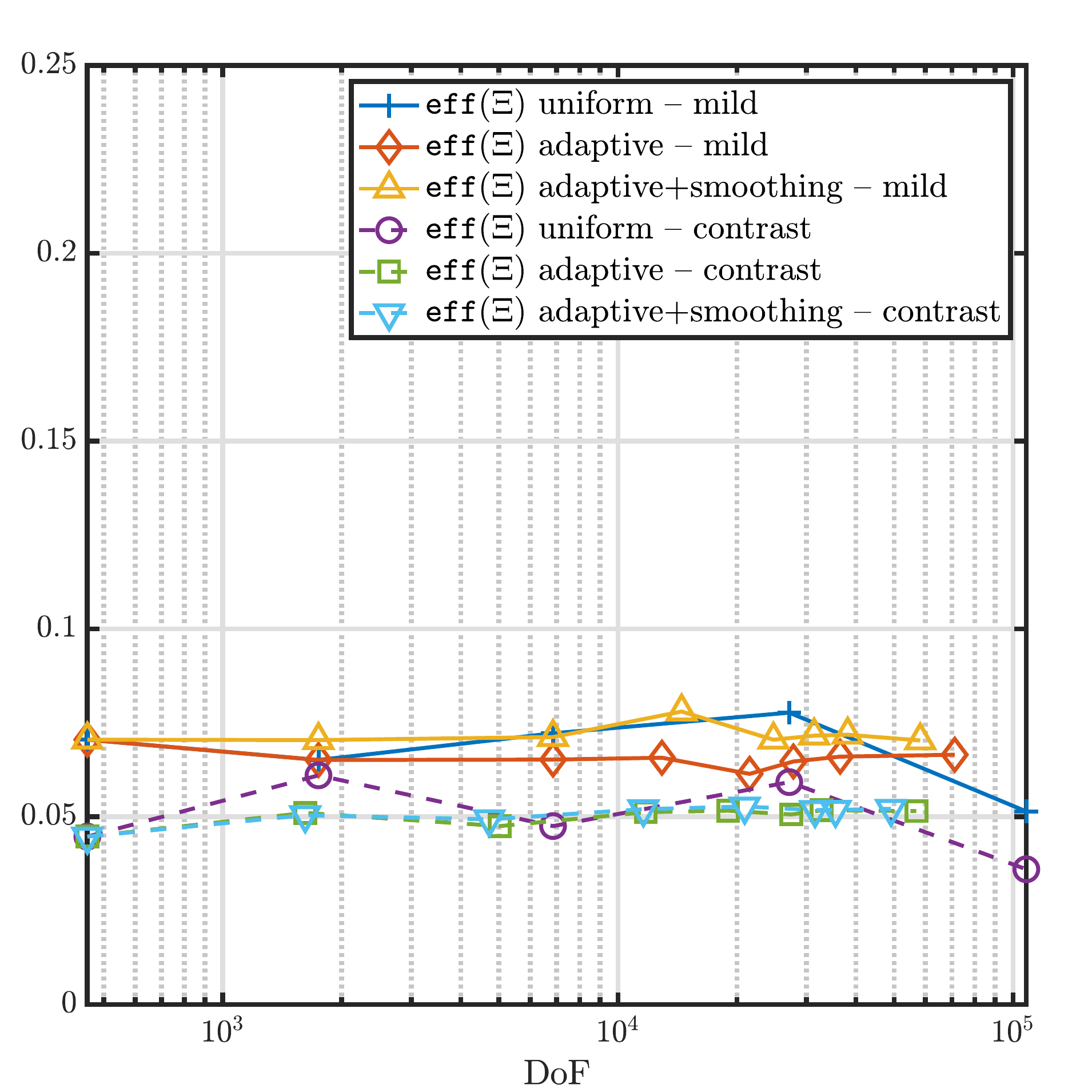}
  \end{center}
  \caption{Example 2. Error decay and effectivity indexes for the convergence test on an L-shaped domain using mild and high-contrast elastic parameters with $k=1$.}\label{fig2:tables}
  \end{figure}

\subsection{Verification of \emph{a posteriori} error estimates}
To assess the performance of the proposed estimators, we use the L-shaped domain $\Omega = (-1,1)^2\setminus (0,1)^2$, the  interface is zig-zag-shaped and going from the reentrant corner $(0,0)$ to the bottom-left corner of the domain $(-1,-1)$, and the porous domain is the one above the interface. We consider manufactured solutions with high gradients near the reentrant corner 
\[
\bu = 10^{-2}\begin{pmatrix} ((x-x_a)^2+(y-y_a)^2)^{-2/3}\\((x-x_a)^2+(y-y_a)^2)^{-2/3}\end{pmatrix}, \quad p^{\mathrm{P}} = ((x-x_a)^2+(y-y_a)^2)^{-2/3},
\]
with $(x_a,y_a)= (0.01,0.01)$. 
We employ adaptive mesh refinement consisting in the usual steps of solving, then computing the local and global estimators, marking, refining, and smoothing. The marking of elements for refinement follows the classical D\"orfler approach \cite{dorfler_sinum96}: a given $K\in \cT_h$ is \emph{marked} (added to the marking set $\cM_h\subset\cT_h$)  whenever the local error indicator $\Xi_K$ satisfies 
\[ \sum_{K \in \cM_h} \Xi^2_K \geq \zeta \sum_{K\in\cT_h} \Xi_K^2,\]
where $\zeta$ is a user-defined bulk density parameter. All edges in the elements in $\cM_h$ are marked for refinement. Additional edges are marked for the sake of closure, and an additional smoothing step (Laplacian smoothing  on the refined mesh to improve the shape regularity of the new mesh) is applied before starting a new iteration of the algorithm. When computing convergence rates under adaptive mesh refinement, we use the expression 
\[\texttt{r}_{(\cdot)} = -2\log(e_{(\cdot)}/\tilde{e}_{(\cdot)})[\log({\tt DoF}/\widetilde{{\tt DoF}})]^{-1}.\]

\begin{figure}[t!]
  \begin{center}
    \includegraphics[width=0.325\textwidth]{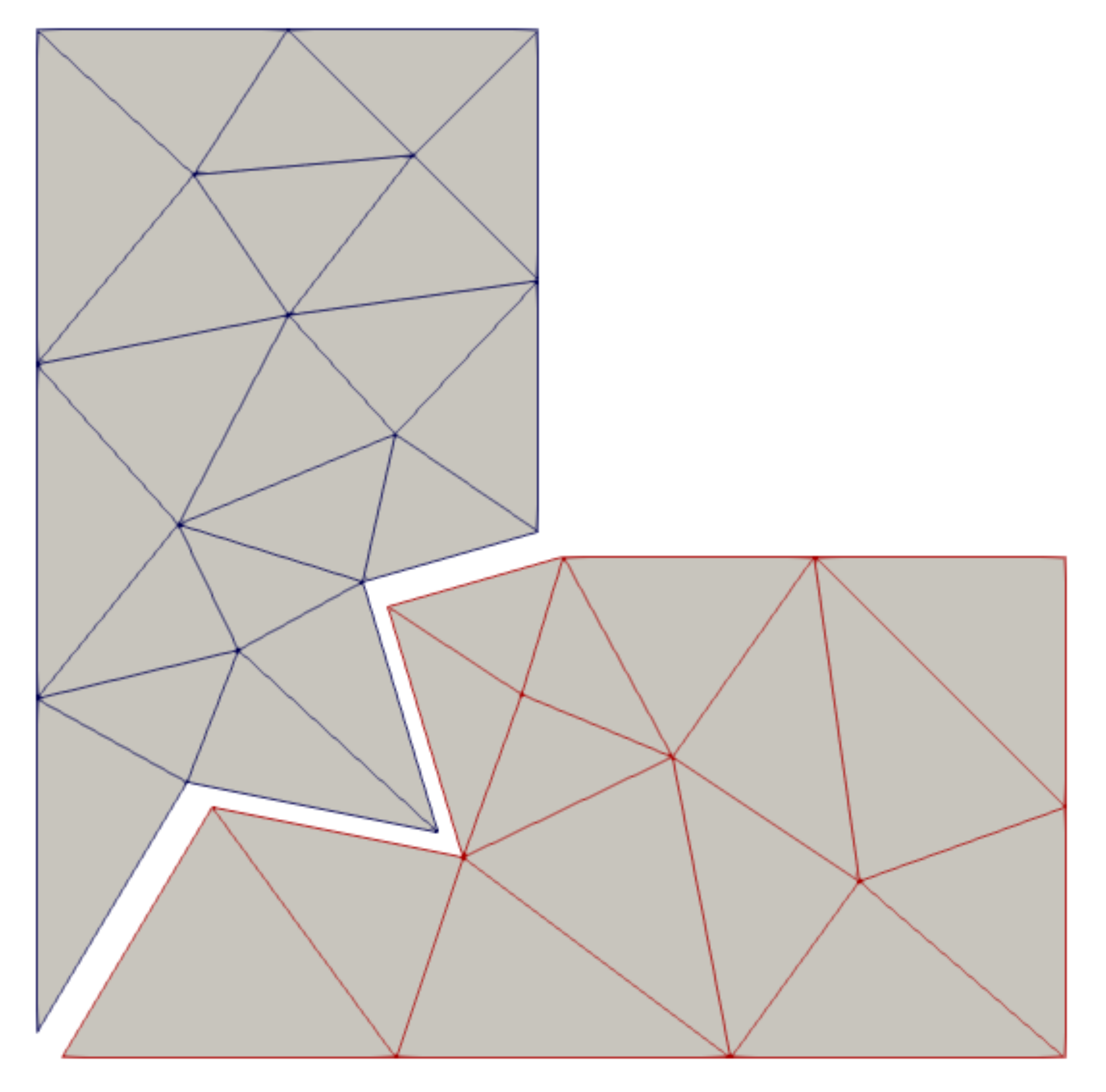}
    \includegraphics[width=0.325\textwidth]{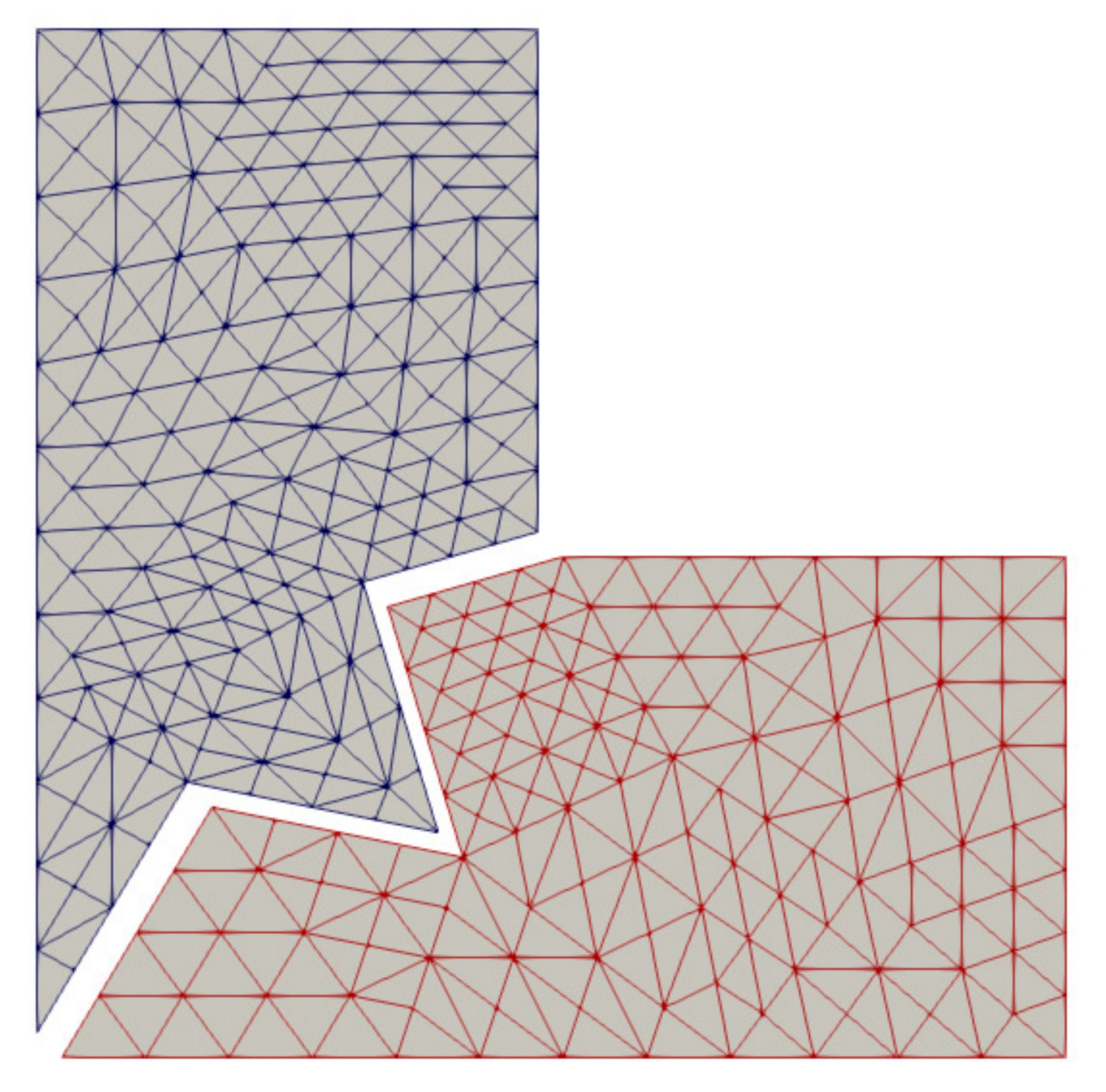}
    \includegraphics[width=0.325\textwidth]{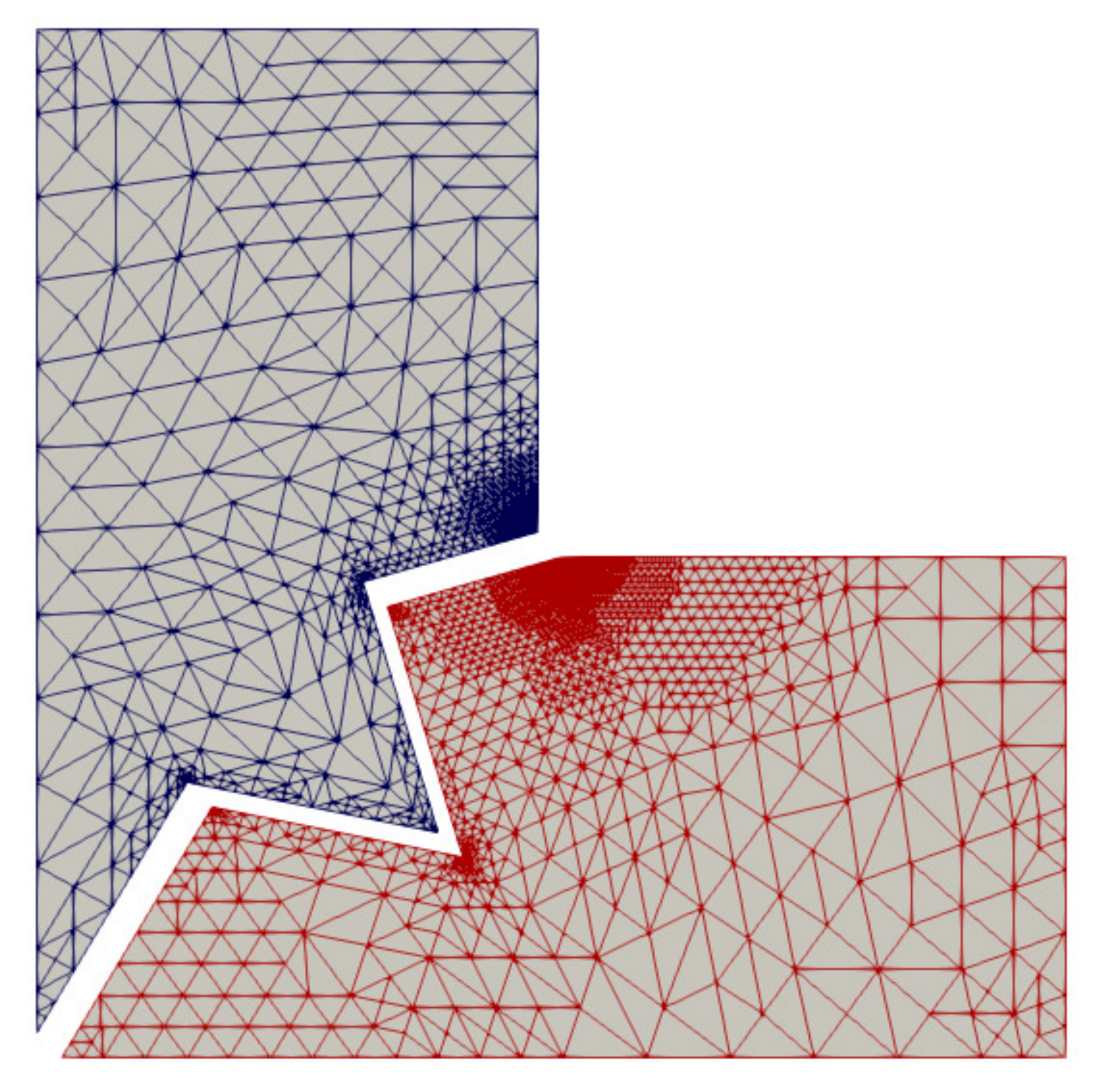}\\
    \includegraphics[width=0.325\textwidth]{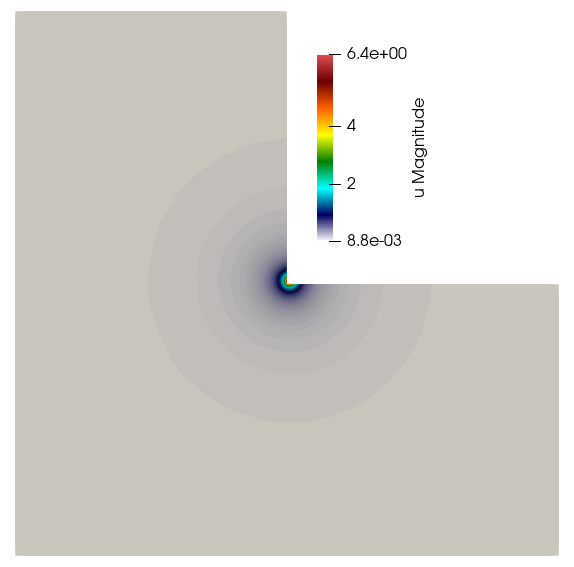}
    \includegraphics[width=0.325\textwidth]{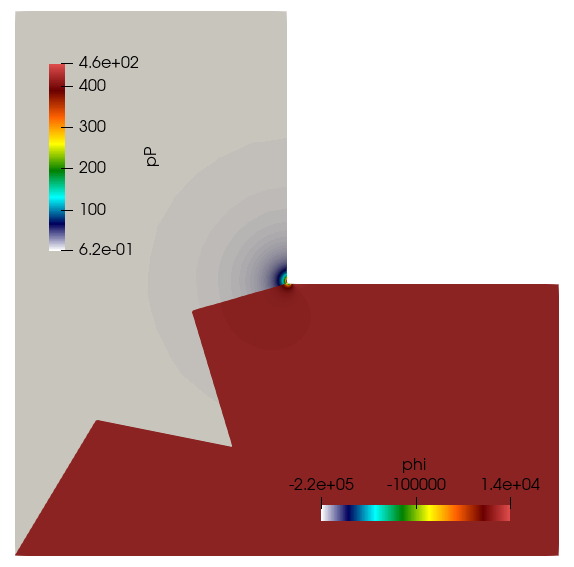}
    \includegraphics[width=0.325\textwidth]{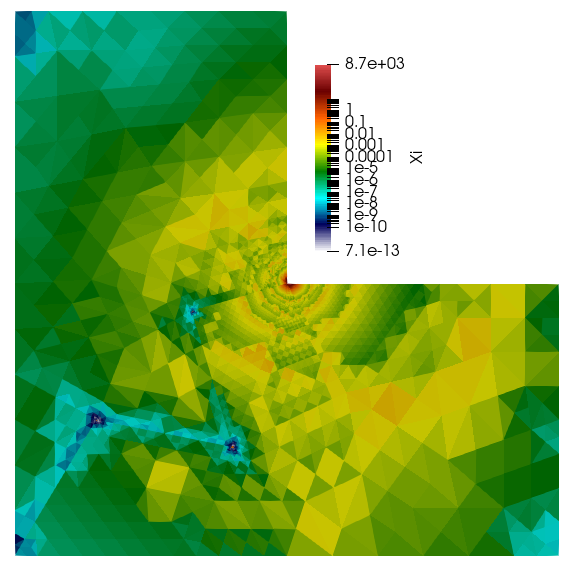}
  \end{center}
  \caption{Example 2. Initial meshes for poroelastic and elastic subdomains, and meshes after 2 and 7 steps of adaptive refinement guided by $\Xi$ (top).
The bottom row shows, at the finest level and for the case without mesh smoothing, the approximate global displacement, Biot fluid pressure, and the cell-wise value of the \emph{a posteriori} error indicator. Here we use $k=1$.}\label{fig2:meshes}
  \end{figure}

We set the following parameter values 
$c_0 = 0.01$, $\alpha = 0.5$, $\eta = 0.01$, $\kappa = 10^{-3}$ and consider two cases for the Young and Poisson moduli:
first $E^\mathrm{E} = 10$, $E^\mathrm{P} = 100$, $\nu^\mathrm{E} = 0.495$, $\nu^\mathrm{P} = 0.4$, and secondly larger contrast:   $E^\mathrm{E} = 1000$, $E^\mathrm{P} = 10$, $\nu^\mathrm{E} = 0.499$, $\nu^\mathrm{P} = 0.25$. Moreover, we only use the polynomial degree $k=1$, $\beta_{\bu} = 500$, and  $\zeta = 10^{-7}$. For this case we consider continuous fluid pressure approximations. The error history is presented in the left and centre panels of  Figure~\ref{fig2:tables}. There we plot the error decay vs the number of degrees of freedom for the case of uniform mesh refinement, and adaptive mesh refinement with or without a smoothing step, and using the mild vs high contrast mechanical parameters. For comparison we also plot an indicative of the orders $\mathcal{O}(h)$ and $\mathcal{O}(h^2)$ (thanks to the relation $\texttt{DoF}^{-1/d}\lesssim h \lesssim \texttt{DoF}^{-1/d}$, in 2D we take $C\,\texttt{DoF}^{-1/2}$ and $C\,\texttt{DoF}^{-1}$, respectively).  
We note that for high contrast parameters, the performance of the three methods is very similar. However, as the mesh is refined, for roughly the same computational cost, the two adaptive methods render a much better approximate solution. Another observation is that the effectivity indexes (plotted in the right panel) have a slightly higher oscillation than in the adaptive cases, but overall they do not show a systematic increase/decrease. 
We also show samples of adaptive meshes in Figure~\ref{fig2:meshes}.
The \emph{a posteriori} error indicator correctly identifies and guides the agglomeration of elements near the zones of high gradients (the reentrant corner), plus the zones where the contrast occurs (the interface corners). The figure also portrays examples of approximate solutions together with the value of $\Xi_K$ locally.

\begin{figure}[t!]
  \begin{center}
   \includegraphics[width=0.325\textwidth]{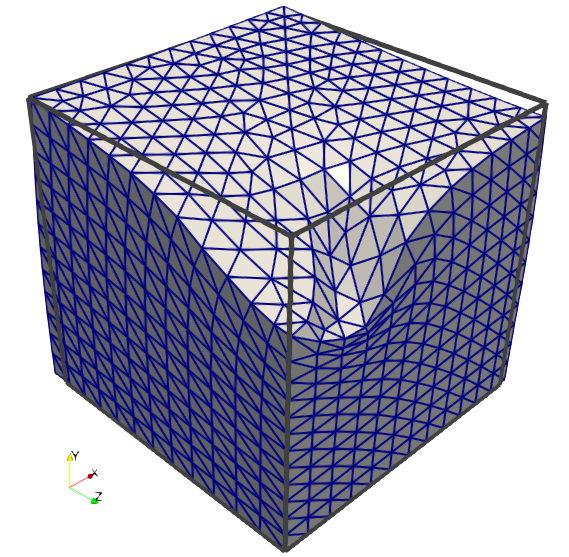}
   \includegraphics[width=0.325\textwidth]{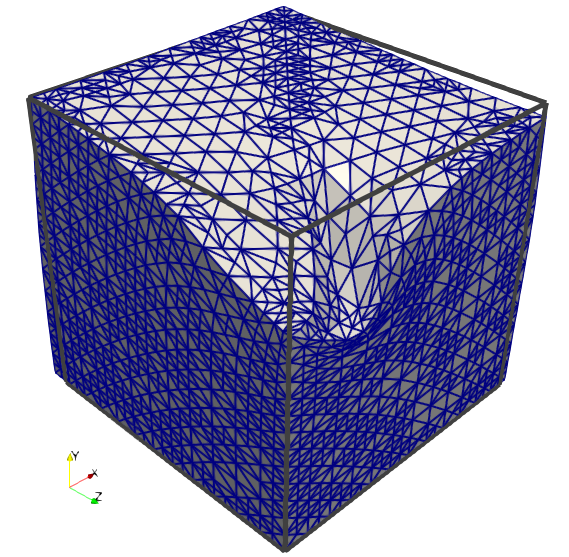}
         \includegraphics[width=0.325\textwidth]{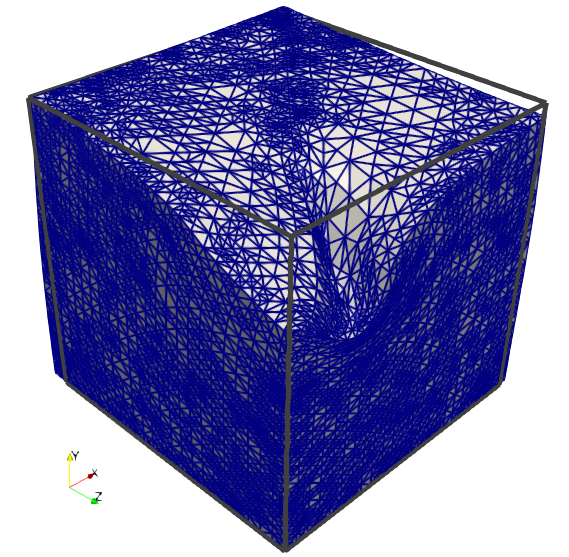}\\   
   \includegraphics[width=0.325\textwidth]{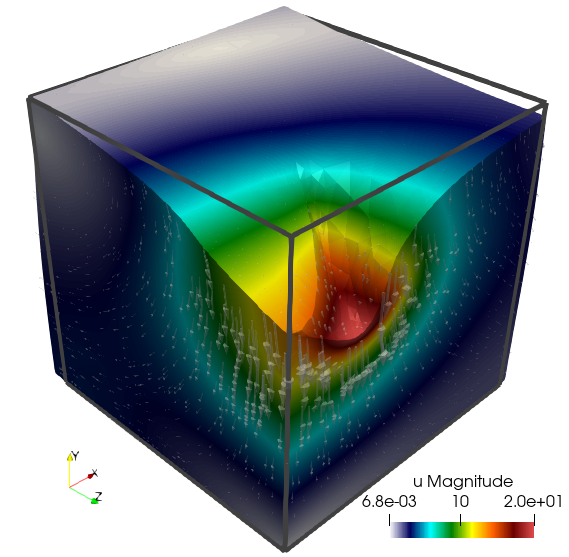}
      \includegraphics[width=0.325\textwidth]{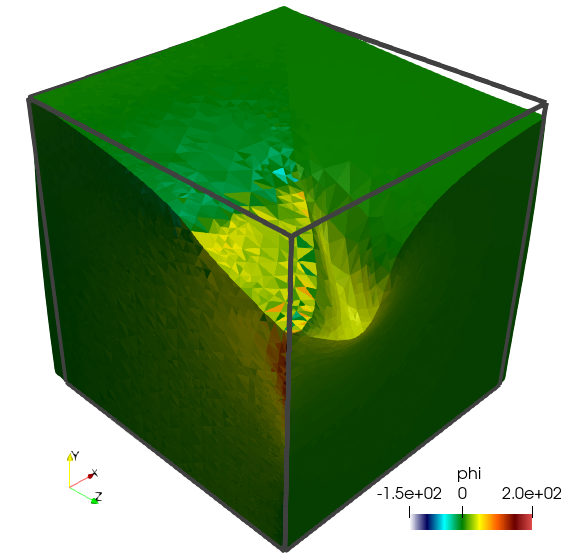}
  \raisebox{-2mm}{\includegraphics[width=0.325\textwidth]{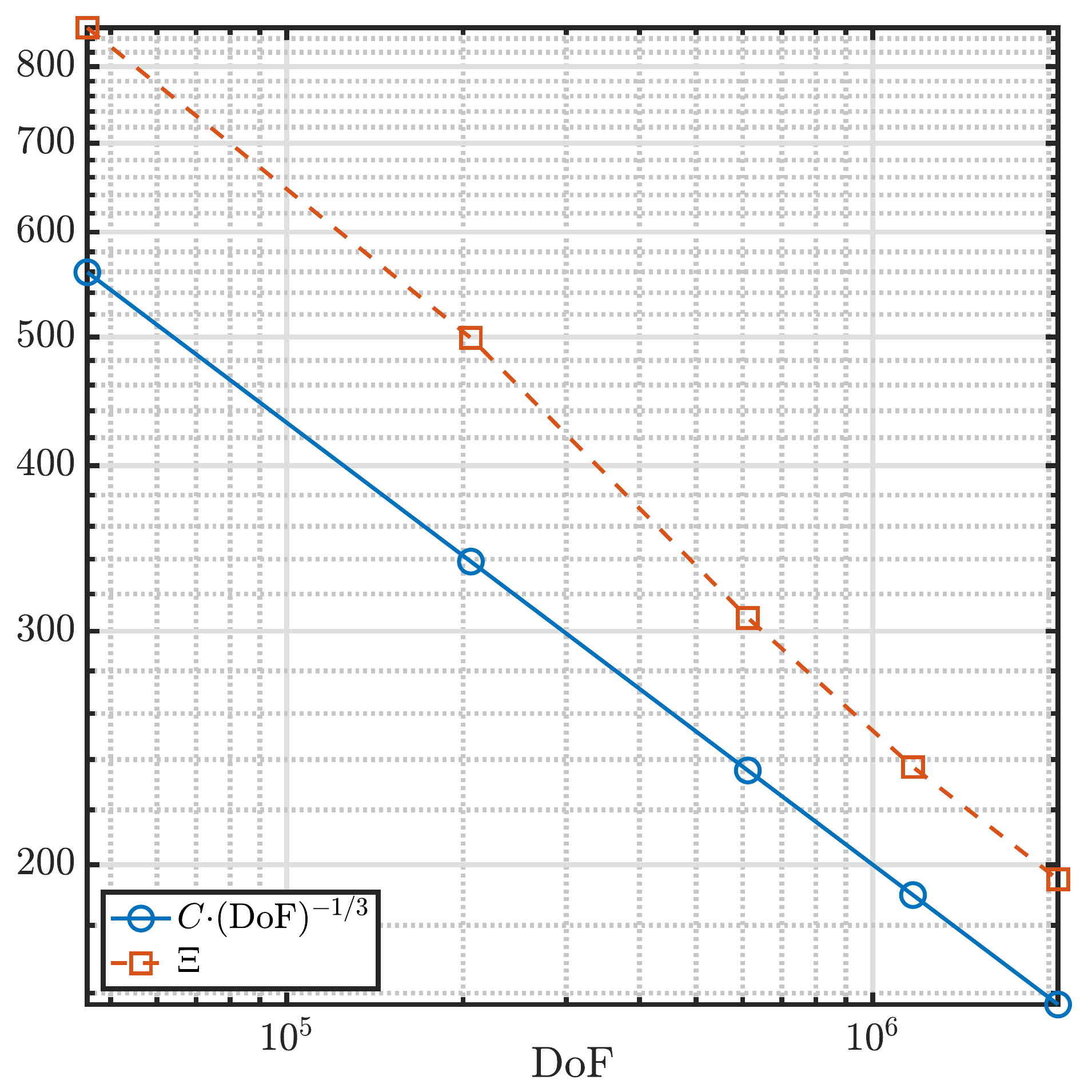}}
  \end{center}
  \caption{Example 3. Meshes on the deformed domain after 1, 2, and 4 steps of adaptive refinement guided by $\Xi$. The bottom row shows, at the finest level, the approximate displacement and total pressure on both subdomains, as well as the convergence of the global \emph{a posteriori} error estimator.}\label{fig3:cube}
  \end{figure}

\subsection{A simple simulation of indentation in a 3D layered material}
We now consider the punch problem (the drainage of a body  by an induced compression loading \cite{cihan21,kumar20}). The full  domain is $\Omega = (0,50)^3$\,mm$^3$, and it is equi-separated into elastic and poroelastic subdomains by a diagonal plane. The elastic moduli, hydromechanical coupling constants, and fluid model parameters are
\begin{gather*}
E^\mathrm{E} = 50\,\text{kN/mm}^2, \quad E^\mathrm{P} = 210\,\text{kN/mm}^2, \quad \nu^\mathrm{E} = 0.3, \quad \nu^\mathrm{P} = 0.499, \\
 \alpha =0.85, \quad c_0 = 0.1, \quad \kappa = 10^{-3}\,\text{mm},\quad \eta = 10^{-2} \text{kN/mm}^2\text{s}.
\end{gather*}
A normal surface load is applied on a quarter of the plane $y = 50$\,mm, near the corner $(0,50,50)$, where the traction has magnitude 60\,N/mm$^2$. On the three planes $x=0$, $y=0$, and $z=50$\,mm we prescribe zero normal displacement $\bu\cdot\nn = 0$ together with an influx condition $\frac{\kappa}{\eta}\nabla p^{\mathrm{P}}\cdot\nn = -1.7$\,mm/s, and on the remainder of the boundary we set stress-free conditions for the solid phase and zero Biot pressure. 
Figure~\ref{fig3:cube} shows the deformed configuration after a few steps of mesh adaptation, which clearly illustrates the jump in material properties.   The meshes are more densely refined near the interface, indicating that the estimator captures correctly the error in these regions. The bottom plots suggest a difference in compliance, as observed across the interface, the deformations are more pronounced in the elastic domain. Note that the method guided by the \emph{a posteriori} error estimate is particularly effective in capturing high solution gradients (of global displacement and of global total pressure). 
The rightmost panel of the figure also shows, in log-log scale, the decay of the global error estimator $\Xi$ vs the number of degrees of freedom, for comparison we also plot an indicative of the mesh size, for 3D, $C\,{\tt DoF}^{-\frac13}$, which confirms that the estimator converges at least with $\mathcal{O}(h)$ to zero.  Note that we are using here the lowest-order method $k=0$ with continuous fluid pressure approximation, and the mesh agglomeration parameter is chosen as $\zeta = 0.02$.

\subsection{A test with realistic model parameters (application to brain multiphysics)} \label{sec:application_problem}

The Biot-elasticity system described in this paper is useful in a range of applications. We show here, as an example, how it can be used to calculate the displacement in a system consisting of the wall of a penetrating vessel and the interstitium surrounding it. A proper network of vessels is shown in \cite{goirand2021network,DVNQJDUUA_2021}, but we will do a simplified 2D illustration. Here, the vessel is T-shaped with the interstitium in a surrounding box. The boundaries are divided into Dirichlet and Neumann boundaries for both the vessel wall and the interstitium domains. The bottom boundary of the vessel wall and the top and bottom boundary of the interstitium have Dirichlet boundaries while the side boundaries for both the vessel wall and interstitium have Neumann boundaries. The mesh is shown in Figure \ref{fig:transient_solution} (top left). The displacement is driven by a pressure wave along the inside of the vessel wall, represented as a sinus wave with a period of one second and a maximum value of 1 kPa. The value is taken from \cite{vinje2019respiratory} where the intraventricular intracranial pressure is reported to be mostly in the range 0.1-1 kPa. The vessel wall for rats are reported to be between 3.8 and 5.8 $\mu$m and the diameter of the vessels is reported to be 43 and 63 $\mu$m \cite{baumbach1988mechanics}.  We choose the vessel to have a diameter of 50 $\mu$m while the vessel wall have a thickness of 5 $\mu$m in this example. The Lame parameters in the interstitium are $\mu= 1$ kPa and $\lambda= 1$ MPa in our example. This is in the range of $\mu=[590, 2.5\cdot 10^3]$ Pa and $\lambda=[529, 1.0\cdot10^{11}]$ Pa given in \cite{smith2007interstitial}. In the vessel wall, they are chosen to be 1 MPa and 1 GPa respectively, which is similar to the reported ranges of $\mu=[3.3\cdot 10^3, 8.2 \cdot 10^5]$ Pa and $\lambda=[3.0\cdot 10^4, 3.4 \cdot 10^{12}]$ Pa given in \cite{bergel1961static, smith2007interstitial}. Additionally, the permeability is 100 nm$^2$ in the interstitium which within the ranges of 10 to 2490 nm$^2$ presented in \cite{holter2017interstitial, smith2007interstitial}. We use mixed boundary conditions as follows
\begin{subequations}
\begin{align}
	[2{\mu^\mathrm{E}} \beps(\bu^\mathrm{E})
	- \varphi^\mathrm{E}\bI]\nn^{\partial\OmE} &= 0 \text{ on } \Gamma^{P}_{N} \\
	\bu^\mathrm{E} &= 0 \text{ on } \Gamma^{P}_{D} \\
	[2{\mu^\mathrm{P}} \beps(\bu^\mathrm{P})
	- \varphi^\mathrm{P}\bI]\nn^{\partial\OmP} &= 0 \text{ on } \Gamma^{P}_{N} \\
	\frac{1}{\eta}\langle \kappa \nabla p^P\cdot \nn,q^P\rangle_{\partial\Omega_P} &= 0  \text{ on } \Gamma^{P}_{N} \\
	\bu^\mathrm{P} &= 0 \text{ on } \Gamma^{P}_{D},
\end{align}\end{subequations}
where $\Gamma^{E}_N$ and $\Gamma^{E}_D$ is the Neumann and Dirichlet boundaries for the vessel wall and $\Gamma^{P}_N$ and $\Gamma^{P}_D$ is the Neumann and Dirichlet boundaries for the interstitium. The previously mentioned pressure wave is described by
\begin{align}
[2{\mu^\mathrm{E}} \beps(\bu^\mathrm{E})
- \varphi^\mathrm{E}\bI]\nn^{\partial\OmE} &= g(t)\nn^{\partial\OmE}\ \ \text{on}\ \Gamma_{\text{traction}},
\end{align}
where $\Gamma_{\text{traction}}$ is the inside of the vessel wall and $g(t)=10^3\cdot\sin(2\pi t)$. The simulation runs over 0.5 seconds which encompass a full expansion and contraction of the vessel wall by the pressure sine wave.

\begin{figure}[t!]
	\begin{center}
		\includegraphics[width=0.89\textwidth]{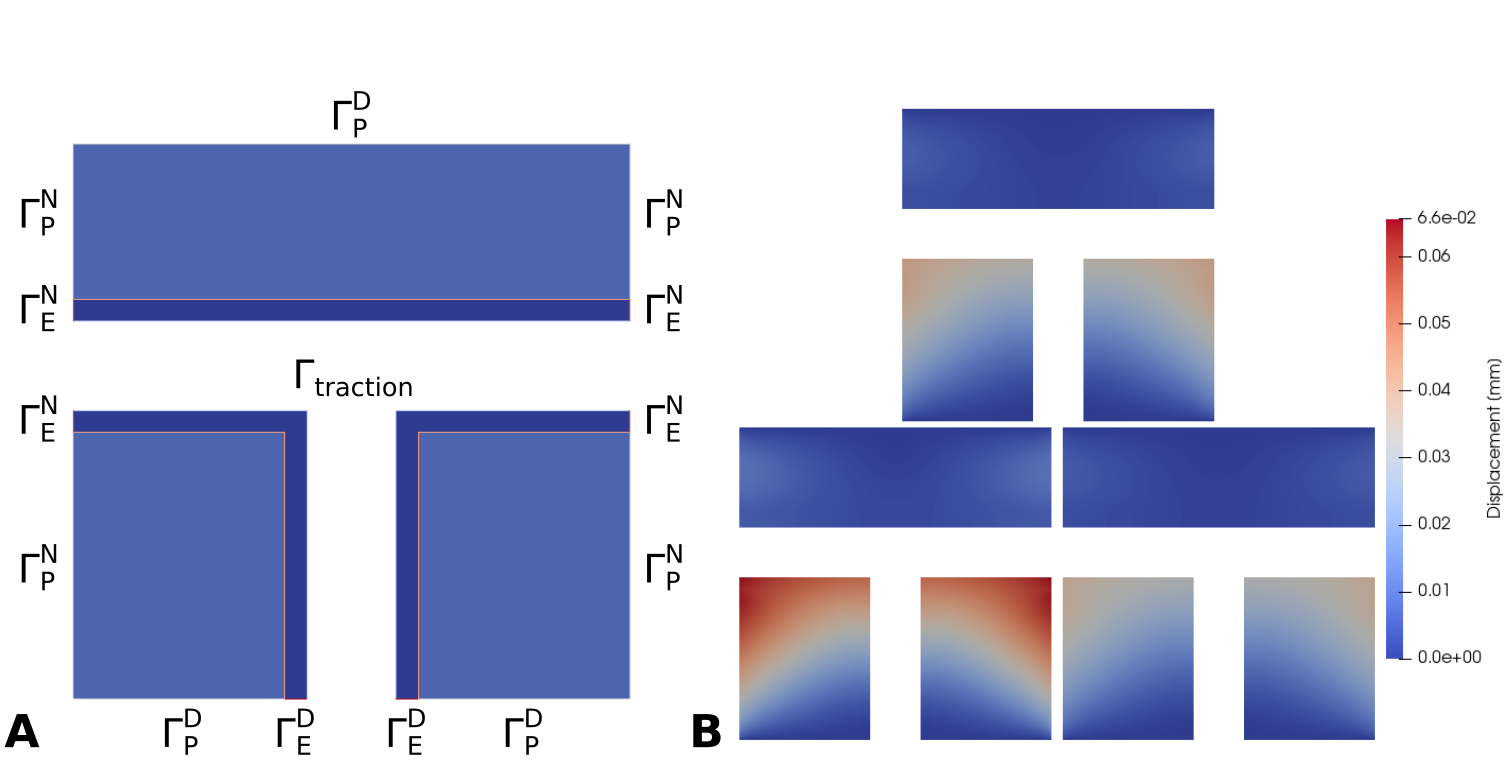}
	\end{center}
	\caption{The solution of the Biot-elasticity system on a T-vessel domain. \textbf{A)} shows the domain with boundaries. The dark blue is the vessel wall and the lighter blue is the interstitium. The top and the bottom of the interstitium domain and the bottom of the vessel wall domain are Dirichlet boundaries, while the sides of both the interstitium and the vessel wall are Neumann boundaries. \textbf{B)} shows the displacement after 0.10 seconds (top), 0.25 seconds (bottom left) and 0.40 seconds (bottom right).}
	\label{fig:transient_solution}
\end{figure}
We time discretize the transient problem using Crank-Nicolson's method with a constant time step $\Delta t = 0.01$ seconds, and solve the problem at each step with a sparse direct solver. From the solution, Figure \ref{fig:transient_solution}, we observe that the pressure wave causes the displacement to spread through the vessel wall and into the interstitium. As the pressure wave goes back to zero, the displacement in the interstitium relaxes back to zero as well. The maximum displacement is 66 $\mu$m. This displacement is quite large considering the maximum arterial wall velocity is 18-25 $\mu$m/s in mice, as reported in \cite{mestre2018flow}.

\subsection{Evaluation of preconditioning robustness}\label{sec:exp_robustness}
 
  We thoroughly evaluated the robustness of $\cP_h$ in \eqref{eq:Ph} for a wide span of physical parameter-value ranges of interest, different mesh resolutions, H(div)-conforming approximations for the displacements (in particular, BDM and Raviart--Thomas), different rectangular domain shapes, and elastic and poroelastic rectangular subdomain shapes. 
  Overall, our results confirm  an asymptotically constant number of MINRES iterations with mesh resolution even with material parameters that exhibit very large jumps accross the interface (e.g., we tested up to 3 orders of magnitude jumps in $\mu$ and $\lambda$), and/or very small or very large values
  (e.g.,  $\kappa\in[10^{-3},10^{-5},10^{-7}]\text{ m}^2$; $\lambda,\mu \in[1,10^3,10^6,10^9]$ Pa), including the extreme cases of near incompressibility, near impermeability, and near zero storativity. We note, however, that the 
  value of the penalty parameter $\beta_{\bu}$ has to be chosen carefully (typically via numerical experimentation), as it can have a significant impact on preconditioner efficiency.

  For conciseness, in this section, we only show results for the particularly challenging (and realistic) combination of physical parameter values corresponding to the problem in Section~\ref{sec:application_problem}. We use the upper part of the domain in this problem, namely the rectangular domain $\Omega = [0,0.25] \times [0.17,0.25]$, with elastic subdomain spanning the 
  thin stripe $[0,0.25] \times [0.17,0.1705]$. As usual, the poroelastic domain is defined as the complement of 
  the elastic domain. We used a triangular uniform mesh generator parametrizable by the number of layers of triangles in the thinner dimension of the elastic domain, which we refer to as $\ell$. Note that $h=0.05/\ell$. We tested in particular with three mesh resolutions corresponding to $\ell=2,4,8$. We solve problem \eqref{semidis11} with the known manufactured solution described in Section~\ref{sec:verification}. We report results only for BDM with $k=0$, although we stress that the number of iterations obtained for  Raviart-Thomas with $k=1$ were very similar to those reported herein. The preconditioned MINRES solver is used in conjunction with $\cP_h$ in \eqref{eq:Ph}, and convergence is claimed whenever the Euclidean norm of the (unpreconditioned) residual of the whole system is reduced by a factor of $10^{6}$. The action of the preconditioner was computed by LU decomposition in all cases. As an illustration of the challenge at hand, for $\ell=2$, $\beta_{\bu}=20$, the condition number of the unpreconditioned system \eqref{semidis11} is approximately as large as $1.25\times 10^{26}$ (as computed by the \texttt{cond} Julia function). 
	By a suitable scaling of the system, this large number could be reduced to $1.25\times 10^{14}$ (i.e., by 12 orders of magnitude). In particular, we scale (\ref{eq:poro-momentum}) and (\ref{eq:Elast})   with $1/ \max (\mu^\mathrm{P},\mu^\mathrm{E})$ and solve for the scaled pressures $\varphi^{\mathrm{P,E}} \leftarrow \varphi^{\mathrm{P,E}}/ \max(\mu^\mathrm{P},\mu^\mathrm{E})$ and $p^\mathrm{P} \leftarrow p^{\mathrm{P}}/ \max(\mu^\mathrm{P},\mu^\mathrm{E})$; see also {\tt scale\_parameters} function in \cite{hdiv_biot_elasticity_paper_software} for full details. We stress that, while the number of preconditioned MINRES iterations
  to solve the original and scaled systems was almost equivalent in all cases tested, we observed a much more reliable behaviour of the error curves for the scaled equations, in particular, under the assumption of a relatively coarse fixed residual tolerance. Thus, in the sequel, we report the errors we obtained in the presence of such an scaling. Finally, in order to study the dependence of the number of iterations and error on the penalisation parameter $\beta_{\bu}$, we tested with different values of $\beta_{\bu} \in [10,5000]$.

 In Figure~\ref{fig:prec_iter_plot} (left), we report the number of preconditioned MINRES iterations versus number of degrees of freedom. While the value chosen for $\beta_{\bu}$ has an impact on the number of iterations for fixed mesh resolution, we can observe nevertheless an asymptotically constant number of iterations with mesh resolution for most of the values of $\beta_{\bu}$ tested. This can be better observed in Figure~\ref{fig:prec_iter_plot} (right), where we report the number of iterations versus $\beta_{\bu}$ for different values of mesh resolution $\ell$; the number of iterations to achieve convergence is approximately constant with mesh resolution except for values of  $\beta_{\bu}$ within $[50,200]$. In order to have a more complete picture, in the curves labeled as ``solver=pminres'' in Figure~\ref{fig:prec_error_plot}, we report the errors in the displacements (left), fluid pressure (center), and global pressure (right) versus $\beta_{\bu}$, obtained with the preconditioned MINRES solver. As a reference, in the curves labeled as ``solver=lu'' in Figure~\ref{fig:prec_error_plot}, we report the corresponding errors obtained with a robust sparse direct solver (UMFPACK). We can observe that the value of $\beta_{\bu}$ has an impact on the error obtained, 
 particularly noticeable in the case of the preconditioned MINRES solver; only for the smallest value of $\beta_{\bu}=10$, the accuracy of the sparse direct solver can be matched for all unknowns. We have checked that, as expected, as we use finer residual tolerances, the error curves for the MINRES solver become closer to those obtained with the sparse direct solver (at the price of some more iterations).

\begin{figure}
	\begin{center}
		\includegraphics[width=0.49\textwidth]{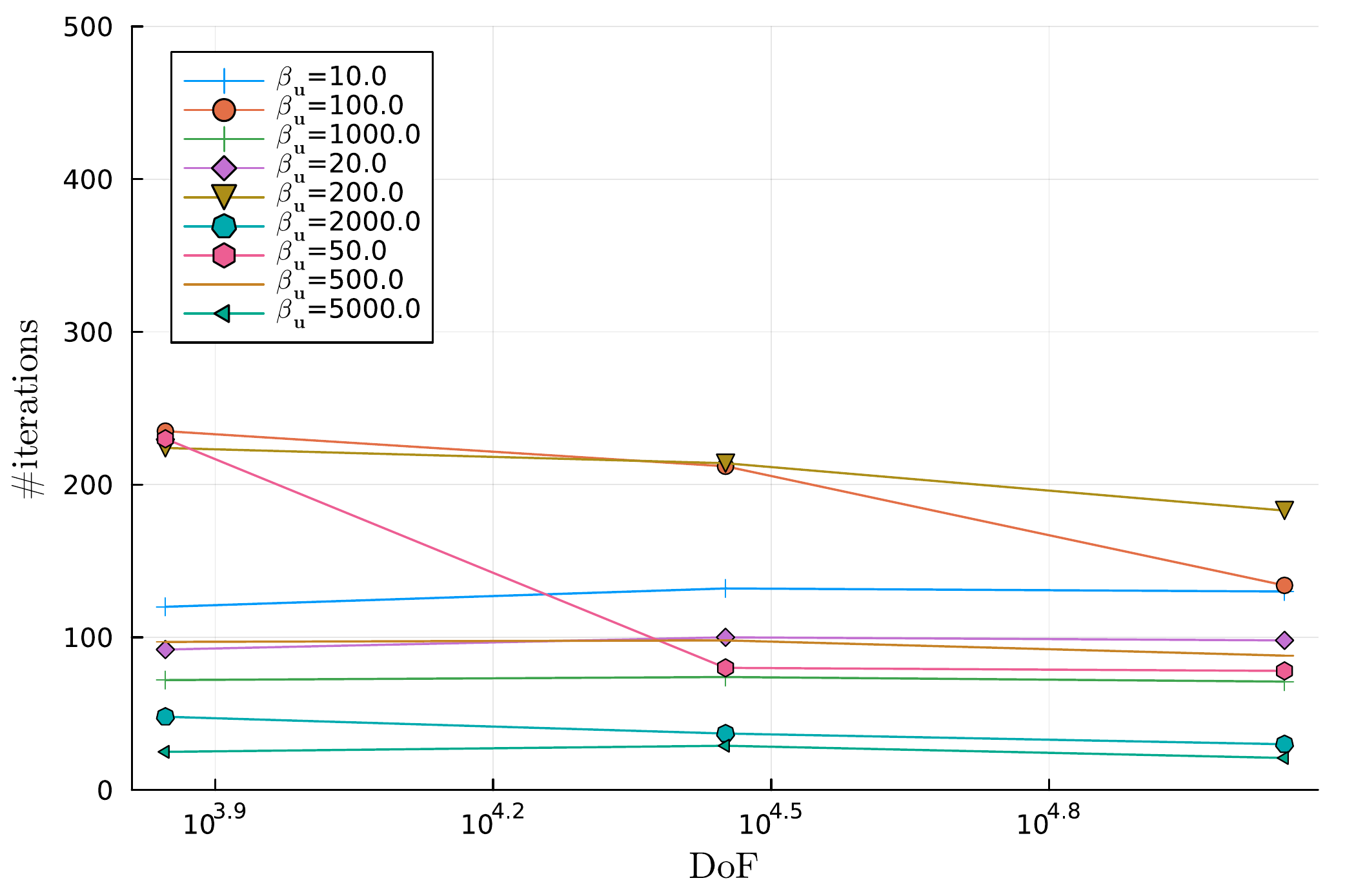}
		\includegraphics[width=0.49\textwidth]{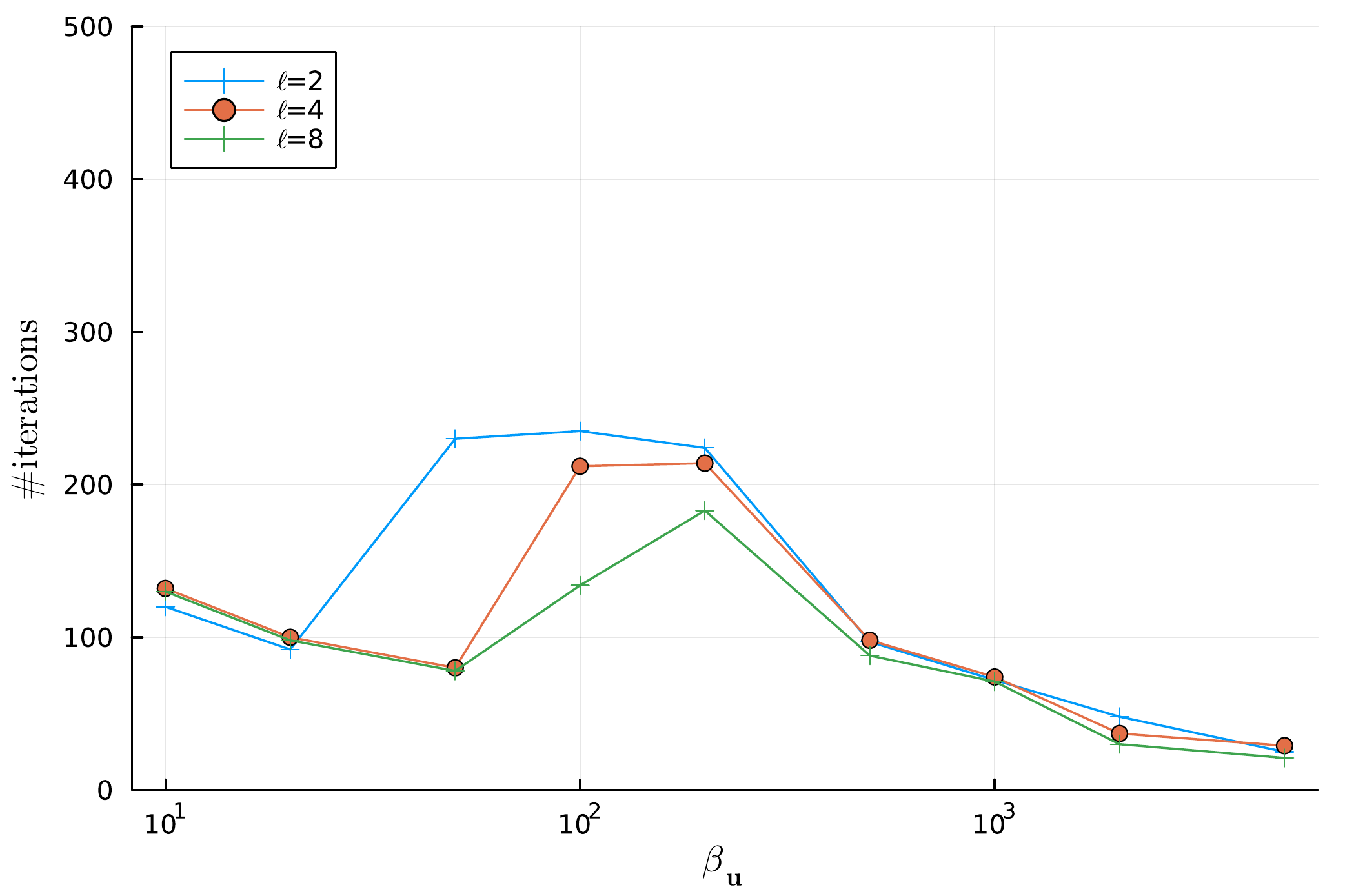}
	\end{center}
	\caption{Number of preconditioned MINRES iterations versus number of degrees of freedom for different values of $\beta_{\bu}$ (left) and  versus $\beta_{\bu}$ for different values of mesh resolution $\ell$ (right).}
	\label{fig:prec_iter_plot}
\end{figure}

\begin{figure}
	\begin{center}
		\includegraphics[width=0.33\textwidth]{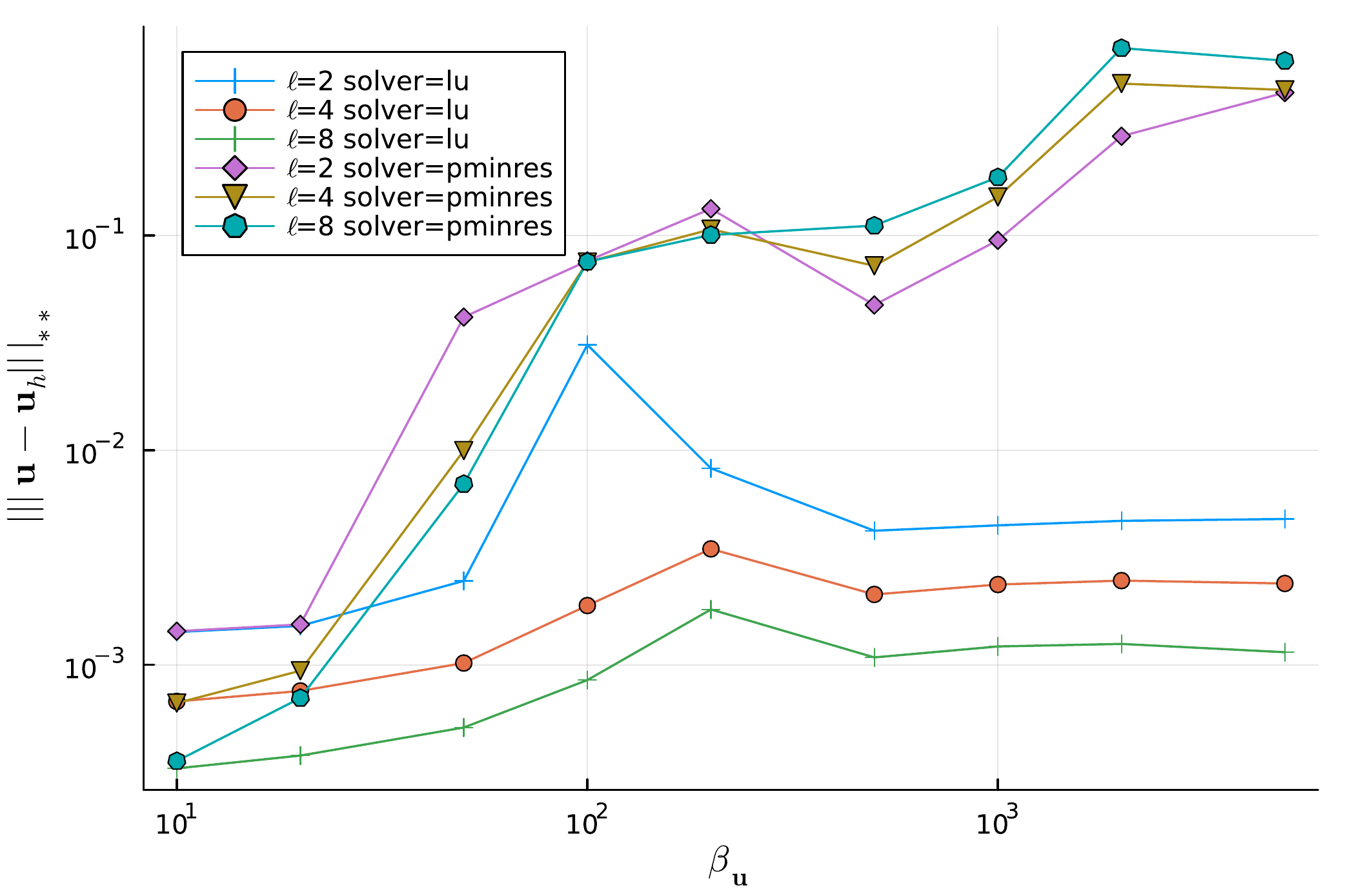}
		\includegraphics[width=0.33\textwidth]{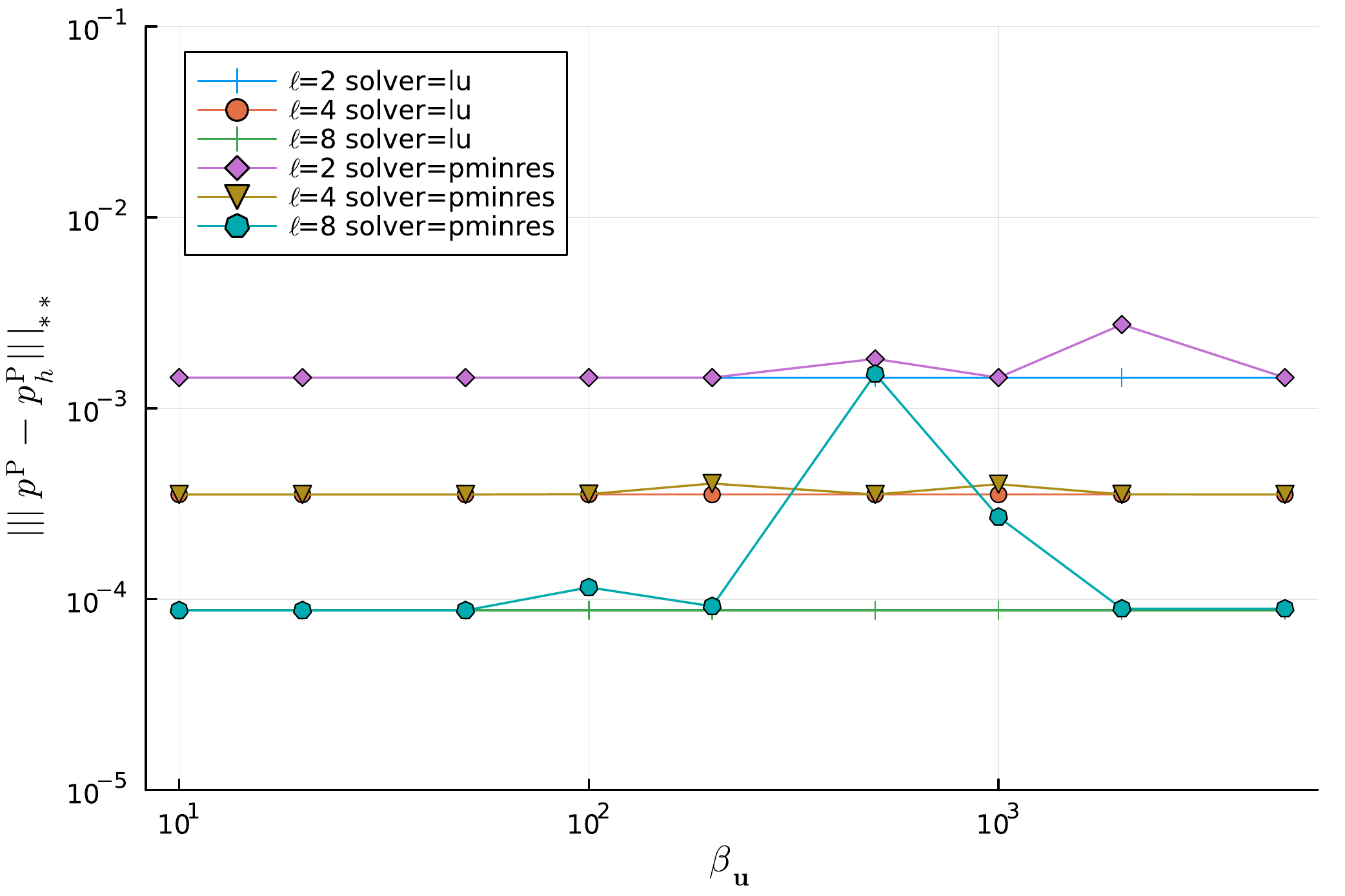}
    \includegraphics[width=0.33\textwidth]{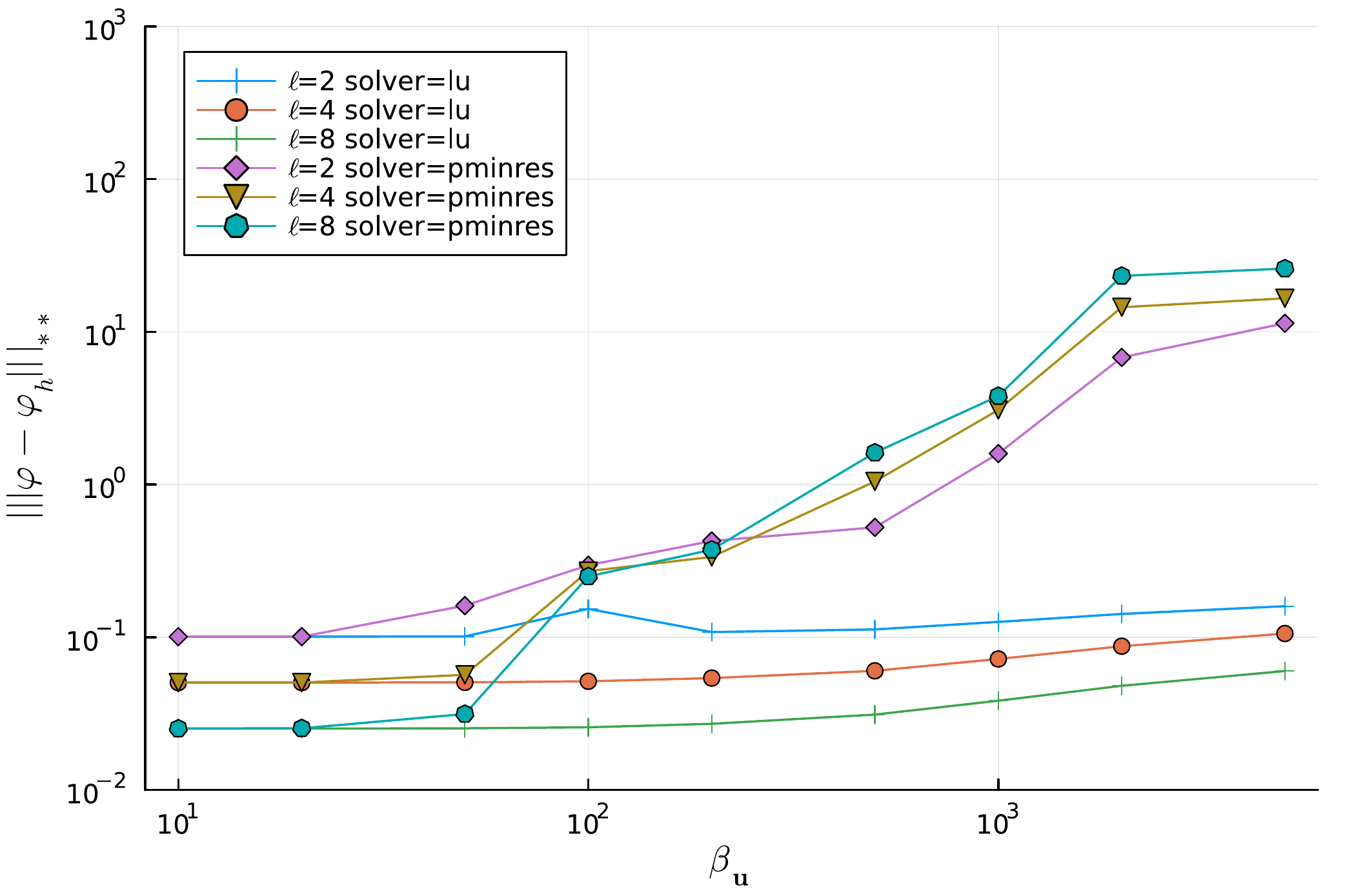}
	\end{center}
	\caption{Error in the displacements (left), fluid pressure (center), and global pressure (right) versus $\beta_{\bu}$ for different values of mesh resolution $\ell$.}
	\label{fig:prec_error_plot}
\end{figure}


\section*{Acknowledgments} This work was initiated during a stay of the second author at Monash University. We acknowledge the support received by the  Sponsored Research \& Industrial Consultancy (SRIC), Indian Institute of Technology Roorkee, India through the faculty initiation grant MTD/FIG/100878; by SERB MATRICS grant MTR/2020/000303;  by the Australian Research Council through the Discovery Project grant DP220103160 and the Future Fellowship grant FT220100496; by the Monash Mathematics Research Fund S05802-3951284; and 
by the Australian Government through the National Computational Infrastructure (NCI) under the NCMAS and ANU Merit Allocation Schemes.